\documentclass[11pt]{amsart}
\usepackage{url}
\usepackage{hyperref,amssymb}
\hypersetup{colorlinks=true,linkcolor=blue,
 urlcolor=cyan,citecolor=magenta}
\usepackage[hyperpageref]{backref}
\usepackage{frcursive,calrsfs,mathrsfs}
\usepackage{xcolor}
\usepackage{appendix}
\newcommand{\Q}{{\mathbb{Q}}}
\newcommand{\Z}{{\mathbb{Z}}}
\newcommand{\F}{{\mathbb{F}}}
\newcommand{\ds}{\displaystyle}
\newcommand{\sst}{\scriptstyle}
\newcommand{\wt}{\widetilde}
\newcommand{\ov}{\overline}
\newcommand{\ft}{\footnotesize}
\newcommand{\ns}{\normalsize}
\newcommand{\BN}{{\bf N}}
\newcommand{\BJ}{{\bf J}}
\newcommand{\BH}{{\bf H}}
\newcommand{\BNu}{\hbox{\Large $\nu$}}
\newcommand{\CH}{{\mathcal H}}
\newcommand{\CR}{{\mathcal R}}
\newcommand{\CW}{{\mathcal W}}
\newcommand{\CU}{{\mathcal U}}
\newcommand{\CO}{{\mathcal O}}
\newcommand{\CT}{{\mathcal T}}
\newcommand{\CI}{{\mathcal I}}
\newcommand{\GI}{{\mathfrak I}}
\newcommand{\CM}{{\mathcal M}}
\newcommand{\CN}{{\mathcal N}}
\newcommand{\order}{\raise0.8pt \hbox{${\scriptstyle \#}$}}
\newcommand{\lien}{\mathrel{\mkern-4mu}}
\newcommand{\too}{\relbar\lien\rightarrow}
\newcommand{\tooo}{\relbar\lien\relbar\lien\too}
\newcommand{\toooo}{\relbar\lien\relbar\lien\tooo}

\newcommand{\plus}{\ds\mathop{\raise 0.5pt \hbox{$\bigoplus$}}\limits}
\newcommand{\prd}{\ds\mathop{\raise 1.0pt \hbox{$\prod$}}\limits}
\newcommand{\sm}{\displaystyle\mathop{\raise 2.0pt \hbox{$\sum$}}\limits}
\newcommand{\ffrac}[2]{\hbox{\ft $\displaystyle\frac{#1}{#2}$}}

\newcommand{\Gal}{{\rm Gal}}
\newcommand{\Lbda}{{\bf \Lambda}}
\newcommand{\rk}{{\rm rk}_3}
\newcommand{\rkp}{{\rm rk}_p}
\newcommand{\tor}{{\rm tor}}
\newcommand{\Ker}{{\rm Ker}}
\newcommand{\pr}{{\rm pr}}
\newcommand{\ta}{{\rm ta}}
\newcommand{\ram}{{\rm ram}}
\newcommand{\prim}{{\rm prim}}
\newcommand{\ab}{{\rm ab}}
\newcommand{\nr}{{\rm nr}}

\newcommand{\bp}{{\rm bp}}
\newcommand{\cyc}{{\rm c}}
\newcommand{\acyc}{{\rm ac}}
\newcommand{\capitul}{{\hbox{\tiny${\rm cap}$}}}
\newcommand{\Ccl}{c\hskip-0.1pt{\ell}}
\newcommand{\Log}{{\rm Log}}
\newcommand{\Art}{{\rm Art}}
\newcommand{\Artmap}{{\rm Artmap}}
\newtheorem{theorem}{Theorem}[section]
\newtheorem{lemma}[theorem]{Lemma}
\newtheorem{corollary}[theorem]{Corollary}

\newtheorem{proposition}[theorem]{Proposition}
\newtheorem{definition}[theorem]{Definition}

\newtheorem{example}[theorem]{Example}

\newtheorem{remark}[theorem]{Remark}
\newtheorem{remarks}[theorem]{Remarks}
\newtheorem{algorithm}[theorem]{Algorithm}
\numberwithin{equation}{section}
\usepackage[english]{babel}
\usepackage{url}
\usepackage{hyperref,amssymb}
\hypersetup{colorlinks=true,linkcolor=blue,
urlcolor=cyan,citecolor=magenta} 
\usepackage[hyperpageref]{backref}
\usepackage{frcursive,calrsfs,mathrsfs}
\usepackage{xcolor} 
\title[Initial layer of the anti-cyclotomic $\Z_p$-extension]
{Initial layer of the anti-cyclotomic $\Z_p$-extension \\
of $\Q(\sqrt{-m})$ and capitulation phenomenon}

\author[Georges Gras]{Georges Gras}

\address{Villa la Gardette, 4 chemin de Ch\^ateau Gagni\`ere, 
F-38520 Le Bourg d'Oisans}
\email{g.mn.gras@wanadoo.fr}

\urladdr{\url{http://orcid.org/0000-0002-1318-4414}}



\begin{document}

\date{September 2, 2025}

\begin{abstract} 
Let $k=\Q(\sqrt{-m})$ be an imaginary quadratic field. We consider 
the properties of capitulation of the $p$-class group of $k$ in the 
anti-cyclotomic $\Z_p$-extension $k^\acyc$ of $k$; for this, using 
a new approach based on the $\Log_p$-function (Theorems 
\ref{artingroup}, \ref{Val}), we determine the first layer 
$k_1^\acyc$ of $k^\acyc$ over $k$, and we show that some partial 
capitulation may exist in $k_1^\acyc$, even when $k^\acyc/k$ is totally 
ramified. We have conjectured that this phenomenon of capitulation is 
specific of the $\Z_p$-extensions of $k$, distinct from the cyclotomic one. 
For $p=3$, we characterize a sub-family of fields $k$ (Normal Split cases) 
for which $k^\acyc$ is not linearly disjoint from the Hilbert class field 
(Theorem \ref{disjunction}). No assumptions are made on the splitting of 
$3$ in $k$ and in $k^*=\Q(\sqrt{3m})$, nor on the structures 
of their $3$-class groups. Four {\sc pari/gp} programs (\ref{program1}, 
\ref{program2}, \ref{program3}, \ref{program4} depending on the
classification of Definition \ref{fourcases}) are given, computing a 
defining cubic polynomial of $k_1^\acyc$, and the main invariants 
attached to the fields $k$, $k^*$, $k_1^\acyc$; some relations 
with Iwasawa's invariants are discussed (Theorem \ref{lambdamu}).
\end{abstract}

\maketitle

\section{Initial layer of the \texorpdfstring{$\Z_p$}{Lg}-extension 
\texorpdfstring{$k^\acyc$}{Lg} of \texorpdfstring{$k=\Q(\sqrt {-m})$}{Lg}}

\subsection{Introduction}
In several papers \cite{HW2010, HW2018, HBW2019, KO2004, 
KW2023, KW2024a}, Hubbard--Washington, Hubbard--Br\"oker--Washington, 
Kim--Oh, Kundu--Washin\-gton, perform large studies of the problem of determining, 
for $p \geq 3$, the layers of $\Z_p$-extensions of imaginary fields, especially for 
the layers of the anti-cyclotomic $\Z_p$-extension $k^\acyc$ of an imaginary 
quadratic field $k$, by means of various approaches including 
complex multiplication and reflection principle of Kummer theory.
In the imaginary quadratic case, $k$ admits two $\Z_p$-extension,
Galois over $\Q$, the cyclotomic one $k^\cyc = k \Q^\cyc$, and the
pro-dihedral one $k^\acyc$; we have $k^\cyc \cap k^\acyc = \Q$,
where $k^\cyc/k$ is totally ramified but where $k^\acyc$ may be 
non linearly disjoint from the $p$-Hilbert class field $H_k^\nr$ .

\smallskip
Such methods and then many others come after the results by 
Carroll--Kisilevsky \cite{CK1976} and Brink \cite{Br2007}, in which one finds 
probably the first use of the torsion group $\CT_k$ of the Galois group of 
the maximal abelian $p$-ramified pro-$p$-extension $H_k^\pr$ of $k$,
whose order of magnitude and the $p$-rank play a fundamental role in these 
questions; so, $\CT_k$ fixes the compositum $\wt k$ of the $\Z_p$-extensions
of $k$. Recall that $H_k^\pr$ contains (or is equal to) the Bertrandias--Payan 
field $H_k^\bp$ (maximal abelian pro-$p$-extension of $k$ in which any 
cyclic extension of $k$ is embeddable in cyclic $p$-extensions of $k$ of arbitrary 
degree; see \cite{BP1972} about the characterization of this field). Then
$\CW_k^\bp := \Gal(H_k^\pr/H_k^\bp) \simeq \oplus_{v \mid p} 
\mu_p(k_v)/\mu_p(k)$. For $p \ne 3$, $\CW_k^\bp = 1$; for $p=3$, 
$\CW_k^\bp$ is of order $3$ if and only if $m \equiv 3 \pmod 9$, $m \ne 3$.
Put $\CT_k^\bp \simeq \CT_k/\CW_k^\bp$.

\smallskip
We will see that $H_k^\bp$ is the compositum $\wt k H_k^\nr$ of $\wt k$
with the $p$-Hilbert class field of $k$.

\smallskip
In the mentioned papers, explicit results are given for small primes $p\geq 3$ 
for various purposes, as statistics, heuristics, or partial results on the Iwasawa 
invariants of $k^\acyc$. For $p=2$, see Carrol \cite{Car1975}, and \cite{Gra1983}.
Then many authors have undertook the study of the torsion group $\CT_k$ in a 
cohomological point of view, like Nguyen Quang Do in \cite{NQD1986}, and find
again class field theory approach in a more conceptual way, but less accessible 
for computing in experimental number theory.
See for instance \cite{MR2019, Mai2011}, and their bibliographies,
for generalizations to the $\Sigma$-ramification context, where $\Sigma$ is a 
suitable set of places, which constitutes a tool for the study of pro-$p$-extensions 
and the approach of many conjectures as that of Fontaine--Mazur.
In \cite{Fu2013}, numerous results give bounds for the Iwasawa invariants 
as improvements of some Sands ones \cite{San1993}.

\smallskip
To our knowledge, there is no published programs allowing the systematic 
computation of a cubic polynomial defining $k_1^\acyc$ over any imaginary 
quadratic field $k$; however, see Brink \cite[Theorem 2]{Br2007} used in
\cite[\S\,7.3]{KW2023}.

\smallskip
Our purpose is to give a new method (with {\sc pari/gp} programs 
\cite{Pari2019}) using the $\Log_p$-function introduced in our 
Crelle's papers (1982/83) to compute the first layer $k_1^\acyc$ of 
$k^\acyc$, then to study the phenomenon of capitulation of  
$p$-classes of $k$ in $k^\acyc$, in direction of some heuristics 
about Iwasawa's invariants of $k^\acyc$; the main reason is 
to obtain effective results from the base field and not results requiring 
$n \gg 0$ as done in many theoretical papers.

\smallskip
This is motivated by a result of Jaulent \cite{Jau2019}, about
some consequences of the property of capitulation, clamming that,
{\it in the totally real case}, Greenberg's conjecture \cite{Gree1976}
holds if and only if {\it the logarithmic class group} of $k$ \cite{Jau1994, 
Jau2016b, Jau2019, Jau2024b} capitulates in the cyclotomic 
$\Z_p$-extension $k^\cyc$ of $k$, and recall that we have conjectured 
that, for an imaginary quadratic field, 
$k^\acyc$ as well as the non-Galois $\Z_p$-extensions of $k$, 
behave in some sense as $k^\cyc$ in the real case (see \cite{FK2002} 
for some examples in this direction about $\lambda, \mu, \nu$ invariants). 

\smallskip
The initial criteria in \cite{Gree1976} were only about $p$-class
groups instead of logarithmic class groups, but we know there
are sufficient similarities between the two notions to justify the
study of class groups capitulations (see \cite[Section 2]{Gra2024b}
for more details). Many numerical tests of Greenberg's conjecture, 
for real quadratic fields, are based on cyclotomic units and analytic 
formulas (e.g., \cite{Pag2022, MPS2025}).
 
\smallskip
An important fact is when $k^\acyc$ is not linearly disjoint from 
the $p$-Hilbert class field $H_k^\nr$; indeed, in the unramified cyclic 
extension $k^\acyc \cap H_k^\nr/k$ some $p$-classes of $k$ may 
capitulate since they capitulate in $H_k^\nr$, but this is going in the 
right direction (see the limit case of Theorem \ref{CHcyclic} leading 
to the conclusion $\lambda = \mu = \nu = 0$), even if capitulation
in the ramified case (disjunction) is more significant.

\smallskip
Recall that $\rkp(A)$ (the $p$-rank of $A$) means $\dim_{\F_p}^{}(A/A^p)$, 
for $\Z_p$-modules of finite type. The following elementary result will 
be crucial in various programing contexts:

\begin{lemma}\label{eta}
Let $\CH$ be a finite abelian $p$-group, and let $\CH_0 = \langle h_0 \rangle$ be a 
cyclic subgroup of $\CH$. Then $\ov \CH_0 := \CH_0 \cdot \CH^p/\CH^p$ 
is a {\it non-trivial} direct factor of $\ov \CH := \CH/\CH^p$ if and only if $h_0 
\notin \CH^p$. Thus, $h_0 \in \CH^p$ is equivalent to $\rkp(\CH) = \rkp(\CH/\CH_0)$.

If $h_0$ is of order $p$, $\CH_0 = \langle h_0 \rangle$ is a direct factor of 
$\CH$ if and only if $h_0 \notin \CH^p$. 
\end{lemma}

\begin{proof}
It suffices to consider the following exact sequence of $\F_p$-vector spaces:
$$1 \too \CH_0 \cdot \CH^p/\CH^p \simeq \CH_0 / \CH_0 \cap \CH^p
\too \CH/\CH^p \too \CH/\CH_0 \cdot \CH^p \too 1, $$
the rank relation $\rkp(\CH)-\rkp(\CH/\CH_0) \in \{0,1\}$ characterizing 
the two possibilities.

\smallskip
Assume $h_0$ of order $p$ and set $\CH = \langle h_1,h_2, \ldots , h_n \rangle$.
Put $h_0 = \prod_i h_i^{a_i}$; if $h_0 \notin \CH^p$, we may assume for instance 
$a_1=1$, so that $\CH=\langle h_0h_2^{-a_2} \cdots h_n^{-a_n}, h_2, \ldots , h_n 
\rangle = \langle h_0,h_2, \ldots , h_n \rangle$, of order $\order \langle h_0 \rangle 
\times \order \langle h_2, \ldots , h_n \rangle = \order\langle h_1 \rangle \times
\order \langle h_2, \ldots , h_n \rangle$ since by definition, $\CH$ is the direct sum 
of the $\langle h_i \rangle$'s and since $h_0 \notin \langle h_2, \ldots , h_n \rangle$. 
Whence $\order \langle h_0 \rangle = \order \langle h_1 \rangle = p$ and the claim,
the reciprocal being obvious.
\end{proof}

By abuse of language, we will say, in the description of the outputs of
the programs, that  $\CH_0$ is a direct factor in $\CH$ (modulo $\CH^p$
being implied) if and only if $h_0 \notin \CH^p$, hence if and only if 
$\rkp(\CH) = \rkp(\CH/\CH_0)+1$ (use $\CH \simeq \Z/p\Z \times \Z/p^3\Z$,
$h_0 = h_1\, h_2^p$, to see that $\CH = \langle h_0 \rangle \cdot \langle h_2 \rangle$
is not direct).

\subsection{Main statements}
Let's give an overview of the various statements applying to $k^\acyc$ and 
its first layer $k_1^\acyc$, for $k = \Q(\sqrt{-m})$; these statements take 
place in ``minus parts'' (for the action of complex conjugation; see \S\,\ref{minus}), as 
$\Gal(k^\acyc/k)$, $\Gal(H_k^\pr/k^\cyc)$, $\Gal(H_k^\nr/k) \simeq \CH_k$, the
$p$-class group of $k$ (the ``plus parts'' being related to $k^\cyc/k$). 

\smallskip
See more generally Section \ref{characters} for the use of the 
characters of $M = k(\mu_p^{})$ in Kummer duality that we briefly recall 
for $p=3$:
let $\chi$ be the character of $k$ and let $\omega_3$ be that of $\Q(\mu_3^{})$; 
let $k^*=\Q(\sqrt{3m})$ (or $\Q(\sqrt{m/3})$), of character $\chi^* = 
\omega_3 \chi^{-1}$, be the reflection of the field $k^*$ in
the mirror involution, $\CH_{k^*}$ 
its $3$-class group, and $W_{\chi^*}$ the ``Kummer radical''
(of character $\chi^*$) for the $3$-ramification over $k$, that is to say, 
the sub-module of $k^{* \times}/k^{* \times 3}$ associated to the 
$\chi$-component of the maximal $3$-elementary $3$-ramified extension 
$H_{k,1}^\pr$ of $k$, giving $M \big(\sqrt[3]{W_{\chi^*}} \big)$ by 
composition with $M$. 

\medskip\noindent
{\bf Result A} (Theorem \ref{degreecap}, $p \geq 3$).
Let $K$ be a $\Z_p$-extension of $k = \Q(\sqrt{-m})$.
Let $\Ccl({\mathfrak a}) \in \CH_k$ of order $p^e$, $e \geq 1$, and let 
$\alpha_0^{} \in k^\times$ be the generator of ${\mathfrak a}^{p^e}$.
If $\Ccl({\mathfrak a})$ capitulates in the layer $K_n$, then $n \geq e$, and
capitulation holds if and only if there exists $\varepsilon^\capitul_n \in E_n$,
the group of units of $K_n$, such that $\alpha_0^{} \cdot \varepsilon^\capitul_n 
\in (K_n^\times)^{p^e}$; if so, $\varepsilon^\capitul_n \not\in E_n^p$ and
is unique modulo $E_n^{p^e}$.                                                                                                                                                                                                                                                                                                                    

\medskip\noindent
{\bf Result B} (Artin group of $k_1^\acyc$, Theorem \ref{artingroup}, 
$p \geq 3$). Let $I_k$ (resp. $P_k$) be the group of prime-to-$p$ 
ideals (resp. principal ideals) of $k$. Let ${\mathfrak a} 
\in I_k \otimes \Z_p$; for the imaginary quadratic field $k$ (for which 
$\log_p(E_k)=0$) and its $p$-completions $k_{\mathfrak p}$, 
${\mathfrak p} \mid p$, we have defined (see \eqref{calcul} for 
the general definition):
$$\Log_p({\mathfrak a}) := \ffrac{1}{p^e}\log_p(\alpha) = \ffrac{1}{p^e} 
\big(\log_{\mathfrak p}(\alpha) \big)_{{\mathfrak p} \mid p} \in 
\hbox{$\prod_{{\mathfrak p} \mid p}$}\, k_{\mathfrak p}, $$ 
as soon as ${\mathfrak a}^{p^e} =: (\alpha)$, $\alpha \in k^\times \otimes \Z_p$. 
Let $\wt k$ be the compositum of the $\Z_p$-extensions of $k$; class 
field theory claims the existence of a canonical isomorphism of the 
form $\Gal(\wt k/k) \simeq \Log_p(I_k \otimes \Z_p)$. When $k=\Q(\sqrt{-m})$, 
$\wt k = k^\cyc \cdot k^\acyc$, then  $\Gal(k^\acyc/k) 
\simeq \Log_p(I_k \otimes \Z_p)^- := \frac{1-\tau}{2} \cdot 
\Log_p(I_k \otimes \Z_p)$.

\smallskip
Let $H_k^\pr$ be the maximal abelian $p$-ramified pro-$p$-extension of $k$ 
and let $H_{k,1}^\pr$, of Artin group $\Art(\hbox{\ft$H_{k,1}^\pr$}) \subseteq 
I_k \otimes \Z_p$, be the maximal elementary $p$-sub-extension of $H_k^\pr$.
Let ${\mathfrak b}^{}_1, \ldots , {\mathfrak b}_{t_k} \in I_k \otimes \Z_p$, be 
such that $\langle\, {\mathfrak b}^{}_1, \ldots , {\mathfrak b}_{t_k} \,\rangle 
\cdot \Art(\hbox{\ft$H_{k,1}^\pr$}) = \big\{ {\mathfrak a} \in I_k \otimes \Z_p,\ 
\Log_p({\mathfrak a}) \in p \,\Log_p(I_k \otimes \Z_p)\big \}$. 
Then, $\Art(k_1^\acyc)$ is the minus part of the left member;
so $\Gal(k^\acyc/k_1^\acyc) \simeq p\,\Log_p (I_k \otimes \Z_p)^-$. 

\medskip\noindent
{\bf Result C} (Theorem \ref{Val}, $p \geq 3$).
(i) We have, where $H_k^\nr$ is the $p$-Hilbert class field and
$\CT_k^\bp$ the Bertrandias--Payan module isomorphic to 
$\CT_k/\CW_k^\bp$, where $\CW_k^\bp 
\simeq \bigoplus_{v \mid p} \mu_p(k_v) \big/ \mu_p(k)$:
$$\big[\Log_p(I_k \otimes \Z_p)^- : \Log_p(P_k \otimes \Z_p)^- \big] =
\big[k^\acyc \cap H_k^\nr : k \big] = \frac{\order \CH_k}
{\order \CT_k^\bp}. $$

(ii) Let $\hbox{\ft{\sc exp}}(k)$ be the exponent of the $p$-class group 
$\CH_k$, $\hbox{\ft{\sc expta}}(k)$ be the exponent of the tame part of
$\BH_k$, and let ${\rm h}(k) := \order \BH_k$. 
Let ${\mathfrak q}$ be a prime ideal dividing a prime $q$ split in~$k$.
Let $\beta$ be the generator of the principal ideal 
${\mathfrak q}^{\hbox{\tiny{\sc exp}}(k) \cdot \hbox{\tiny{\sc expta}}(k)}$.
Set $\log_p(\beta) = C_0 + C_1 \sqrt{-m} \in  \Z_p \oplus \Z_p \sqrt{-m}$, hence
$\log_p(\beta)^- = C_1 \sqrt{-m}$.
Then the condition $\Log_p({\mathfrak q})^- \in p\,
\Log_p(I_k \otimes \Z_p)^-$ holds as soon as, taking the 
$p$-adic valuation ${\bf v}_p$, normalized by ${\bf v}_p(p)=1$:
$${\bf v}_p^{}(C_1) \geq \hbox{\sc Val} := {\bf v}_p^{}(\hbox{\ft{\sc exp}}(k))
+{\bf v}_p^{}(\order \CT_k) - {\bf v}_p^{}({\rm h}(k))+2-{\bf v}_p^{}(m). $$

\noindent
{\bf Result D} (Theorem \ref{DF}, $p=3$).
Let $k = \Q(\sqrt{-m})$ with $m \equiv 3 \pmod 9$, let $\CW_k^\bp$ (of 
order~$3$) be the subgroup of $\CT_k$ fixing the Bertrandias--Payan field 
$H_k^\bp$. Let's consider two cases:

\smallskip
(i) $k^\acyc$ disjoint from $H_k^\nr$ over $k$ (i.e., $k^\acyc/k$ totally ramified). 
Then $\CW_k^\bp$ is direct factor of $\CT_k$ if and only if $\rk(\CT_k) = \rk(\CH_k)+1$.

\smallskip
(ii) $k^\acyc$ not disjoint from $H_k^\nr$ over $k$ (i.e., $k_1^\acyc/k$ unramified). 
Then $\CW_k^\bp$ is direct factor of $\CT_k$ if and only if there exists a radical 
$w \in W_{\chi^*}$, of conjugate $w'$, defining a {\it ramified} cubic extension 
$K_1$ of $k$. In other words, using {\sc pari/gp} notations, if ${\sf Q = 
x^3-3*a*x-t}$ (${\sf a^3 = w*w'}$, ${\sf t = w+w'}$, cf. Proposition \ref{Q}) 
and ${\sf R = polcompositum(Q,x^2+m)[1]}$, $\CW_k^\bp$ is direct 
factor if and only if ${\sf valuation(nfdisc(R),3)} > 3$.

\medskip\noindent
{\bf Result E} (Theorems \ref{disjunction}, \ref{WvsR}, $p=3$).
Let $k := \Q(\sqrt{-m})$, $m \equiv 3 \pmod 9$, and let $k^*$
be the reflection of $k$ in the mirror involution. Assume that $\CH_{k^*} \ne 1$
and $\rk(\CT_k) = \rk(\CH_{k^*})$ (Normal Split case); then:

\smallskip
(i) The first layer $k_1^\acyc/k$ is always unramified.

\smallskip
(ii) The normalized $3$-adic regulator $\CR_{k^*} 
\simeq \Gal(H_{k^*}^\pr/H_{k^*}^\nr)$ 
(\cite[Section 5]{Gra2018} for the definition) 
is not direct factor in $\CT_{k^*}/\CT_{k^*}^3$ if and 
only if $\CW_k^\bp$ is direct factor in $\CT_k$.

\medskip\noindent
{\bf Result F} (Theorems \ref{capitule2}, \ref{capitule1}). Assume that 
$k^\acyc/k$ is totally ramified, that $\CH_k \simeq \Z/p\Z$ and 
$\CH_{k_1^\acyc} \simeq \Z/p\Z \times \Z/p\Z$. Then $\CH_k$ 
capitulates in $k_1^\acyc$.

\medskip\noindent
{\bf Result G} (Theorems \ref{lambdamu}, Corollary \ref{cyclotomic}). 
Let $K/k$ be a totally ramified $\Z_p$-extension of the imaginary quadratic 
field $k$, in the case where $p \geq 3$ does not split in $k$. Then 
$\lambda_p(K/k) = \mu_p(K/k) = 0$ if and only if $\CH_k$ capitulates in $K$.
It follows that, if $\CH_k \ne 1$, the cyclotomic $\Z_p$-extension $k^\cyc/k$
is such that $\lambda_p(k^\cyc/k) \geq 1$ and $\mu_p(k^\cyc/k) = 0$.

\medskip\noindent
{\bf  Computational Results} (Algorithms \eqref{method}, \eqref{abcd}, 
$p = 3$). We distinguish in Definition \ref{fourcases} four arithmetic 
properties of the field $k^* = \Q(\sqrt{3m})$ (or $\Q(\sqrt{m/3})$)
(Non Split, Normal Split, Special Split and Trivial), corresponding to  
programs (\eqref{program1}, \eqref{program2}, \eqref{program3}, 
\eqref{program4}). 

\smallskip
These programs lead to outputs of the following form, where the 
first data are the main invariants of ${k}$ and ${k^*}$, where ${L_w}$ 
is the list of the radicals $w_J^{} \in W_{\chi^*}$ of the 
$3$-ramified cyclic cubic fields of $k$, defined by the polynomials $Q$; 
they are indexed by the ${(3^{{\rm rt(k)}+1}-1)/2}$ values of $J$ depending on 
the $3$-rank ${\rm rt(k)}$ of $\CT_k$ (see \S\,\ref{notations} for the specific
{\sc pari/gp} notations):

\smallskip
\ft\begin{verbatim}
m=157019 Disc=157019 kronecker(-m,3)=1 h_k=135 H_k=[45,3] T_k=[9,3]  
H_kstar=[3,3] #T_k^bp=27 #W_k^bp=1  k_1^ac/k is Ramified 
Lw=List([
Mod(-381*x-261494,x^2-471057),
Mod(7850*x-5387737,x^2-471057),
Mod(-103*x+70628,x^2-471057),
Mod(24614*x+16891619,x^2-471057),
Mod(72157379882980011220138260*x-49524201965286630108480893399,x^2-471057),
Mod(-946345768982667613775421*x+649511097383218404100625686,x^2-471057),
Mod(12411347473034251552608*x-8518353630809707578446027,x^2-471057),
Mod(-777529970846087123062729061699770*x+533646750653812276531575169634285063,x^2-471057),
Mod(10197324228799635133629457047377*x-6998789942644566371270229378089032,x^2-471057),
Mod(-133738152002179343848094880046*x+91789298072647800633466465420499,x^2-471057),
Mod(8378253985168811268921483768324271465040*x-
    5750296686920553899462594924114216274058931,x^2-471057),
Mod(-109881002098260105810426670642062717049*x+
    75415279059292515445513767333492466789234,x^2-471057),
Mod(1441091979724050966458814918285313052*x-
    989073194871404415565709148797397703463,x^2-471057)])
J=1 q=5 Qq irreducible
SOLUTION:J=2 w=Mod(7850*x-5387737,x^2-471057) 
Q^ac=x^3-318*x-4067    H_kacyc=[135,3,3,3]    Val=4    a=1
J=3 q=5 Qq irreducible
J=4 q=5 Qq irreducible
J=5 q=19 Qq irreducible
J=6 q=5 Qq irreducible
J=7 q=5 Qq irreducible
J=8 q=19 Qq irreducible
J=9 q=5 Qq irreducible
J=10 q=5 Qq irreducible
J=11 q=19 Qq irreducible
J=12 q=5 Qq irreducible
J=13 q=5 Qq irreducible
W_k^bp=1 is not direct factor of T_k
Algebraic norm of H_kacyc=[135,3,3,3] H_k=[45,3]
Norm of the component 1 of H_kacyc: [3,0,0,0]
Norm of the component 2 of H_kacyc: [0,0,0,0]
Norm of the component 3 of H_kacyc: [0,0,0,0]
Norm of the component 4 of H_kacyc: [0,0,0,0]
PARTIAL CAPITULATION OF H_k
\end{verbatim} \ns

\smallskip
See Algorithm \eqref {abcd} for the interpretation of the above matrix 
leading to the kernel of capitulation. Numerical examples, 
complements and illustrations are given in Appendix \ref{AppA}.

\section{The arithmetic of \texorpdfstring{$k^\acyc/k$}{Lg} -- Algorithms}
\subsection{Unit groups of the layers of \texorpdfstring{$k^\acyc$}{Lg}}
Our main goal is to verify the philosophy suggesting that, for imaginary 
quadratic fields~$k$, $k^\acyc/k$, but also the non-Galois $\Z_p$-extensions 
$K/k$, behave like the totally real case of Greenberg's conjectures related to 
$k'^\cyc/k'$ when $k'$ is any totally real number field. The reason is that there 
is, a priori, no systematic obstruction for capitulations of the $p$-class groups 
in the $\Z_p$-extensions $K \ne k^\cyc$; this is due to the groups of 
units $E_n$ of the layers $K_n$ (the $\Z$-rank of $E_n$ being $p^n-1$), 
which consist of {\it non-totally real} units. This applies to the non-cyclotomic 
$\Z_p$-extensions of $k$ and in particular to $k^\acyc$; refer to Remark 
\ref{cyclocap} explaining why the case of $k^\cyc$, for which 
$E_{k_n^\cyc} = E_{\Q_n^\cyc}$, does not work.

\smallskip
It is well known that for all $n \geq 0$, there exists an injective map:
$$\Ker_{\CH_k} \big (\BJ_{k_n^\acyc/k} \big) \too {\bf H}^1(G_n,E_n)
 = E_n/E_n^{1- \sigma_n}, $$
where $\BJ$ is the transfer map, and
$\sigma_n$ a generator of $G_n = \Gal(k_n^\acyc/k)$, which 
associates with the capitulation in $k_n^\acyc$ of $\Ccl({\mathfrak a})$, the unit 
$\eta_n = \alpha_n^{1- \sigma_n}$, modulo $E_n^{1- \sigma_n}$,
where $({\mathfrak a})_n =: (\alpha_n)$, $\alpha_n \in k_n^\acyc{}^\times$ (see, 
e.g., \cite{Jau1988}). But this is a tautological translation of the property of
capitulation, without any information and computability.

\smallskip
A feeling is that the capitulation kernel of $\BJ_{k_1^\acyc/k}$
(of order bounded by $\order \CH_k$) can become non-trivial for 
$\BJ_{k_n^\acyc/k}$, $n$ large enough; but this is of course a 
non-algebraic property of units, which explains the almost 
hopelessness of proving capitulation in $k^\acyc$, and more 
generally proving analogues of Greenberg's conjectures. 
Moreover, the above injective map does not give any information
on the image, even when the structure of ${\bf H}^1(G_n,E_n)$ is known.

\smallskip
The {\it existence of a link} between units and capitulation 
of $\CH_k$ in $K \ne k^\cyc$ may be better specified by the following result
of which we give a more general statement (it is rather 
obvious but over an imaginary quadratic field one brings 
more precise properties). 

\begin{theorem}\label{degreecap}
For $p \geq 3$, let $K$ be a $p$-cyclic extension or a $\Z_p$-extension 
of $k = \Q(\sqrt{-m})$, $m \ne 3$.
Let $\Ccl({\mathfrak a}) \in \CH_k$ of order $p^e$, $e \geq 1$, and let 
$\alpha_0^{} \in k^\times$ be the generator of ${\mathfrak a}^{p^e}$.

\smallskip
(i) If $\Ccl({\mathfrak a})$ capitulates in the layer $K_n$, then $n \geq e$. 

\smallskip
(ii) For $n \geq e$, $\Ccl({\mathfrak a})$ capitulates in $K_n$ if and only if there 
exists $\varepsilon^\capitul_n \in E_n$ such that:
$$\alpha_0^{} = \varepsilon^\capitul_n \cdot \alpha_n^{p^e},\ \,
 \alpha_n \in K_n^\times. $$

(iii) If $\Ccl({\mathfrak a})$ capitulates in $K_n$, $\varepsilon^\capitul_n 
\notin E_n^p$ and is unique in $E_n/E_n^{p^e}$.
\end{theorem}

\begin{proof}
To simplify, put $({\mathfrak a})_n := \BJ_{K_n/k}({\mathfrak a})$.
Suppose that $\Ccl({\mathfrak a}) \in \CH_k$ capitulates in $K_n$ 
for $n < e$, and set $({\mathfrak a})_n = (\alpha_n)$, $\alpha_n 
\in K_n^\times$; then $({\mathfrak a})_n^{p^e} = (\alpha_n)^{p^e}$, 
whence $(\alpha_0^{})_n = (\alpha_n^{p^e})$; so there exists
$\varepsilon^\capitul_n \in E_n$ such that
$\alpha_0^{} = \varepsilon^\capitul_n \cdot \alpha_n^{p^e}$.
Taking the arithmetic norm in $K_n/k$, this yields 
$\alpha_0^{p^n} = \pm\BN_{K_n/k}(\alpha_n)^{p^e}$; 
then $\alpha_0^{} =  (\pm \BN_{K_n/k}(\alpha_n))^{p^{e-n}}$,
hence ${\mathfrak a}^{p^e} = (\BN_{K_n/k}(\alpha_n)^{p^{e-n}})$
giving ${\mathfrak a}^{p^n} = (\BN_{K_n/k}(\alpha_n))$ in $k$ for
$n < e$ (absurd). Reciprocally, $\alpha_0^{} = \varepsilon^\capitul_n 
\cdot \alpha_n^{p^e}$ implies trivially the capitulation of $\Ccl({\mathfrak a})$
in $K_n$.

\smallskip
Let's prove that $\varepsilon^\capitul_n \notin E_n^p$. Suppose
$\varepsilon^\capitul_n = \eta_n^p$, $\eta_n \in E_n$, thus
$\alpha_0^{} = \eta_n^p \cdot \alpha_n^{p^e} $;
then setting $\beta_n := \eta_n \cdot \alpha_n^{p^{e-1}}$, 
one gets $\alpha_0^{} = \beta_n^p$ in $K_n^\times$.
This cannot define a Kummer $p$-cyclic extension of~$k$
since $\mu_p \not\subset k^\times$; thus $\beta_n = \beta 
\in k^\times$ and ${\mathfrak a}^{p^e} = (\beta)^p$, 
implying ${\mathfrak a}^{p^{e-1}} = (\beta)$ in $k$ (absurd).
\end{proof}

\begin{remarks}\label{cyclocap}
(i) These results do not apply to $k^\cyc$; indeed, since 
in this case $\varepsilon^\capitul_n \in \Q_n^\cyc$,
taking $\BN_{k_n^\cyc/\Q_n^\cyc}$ leads to 
$\BN_{k/\Q}(\alpha_0^{}) = a^{p^e} =
(\varepsilon^\capitul_n)^2 \cdot 
\BN_{k_n^\cyc/\Q_n^\cyc}(\alpha_n^{p^e})$, for   
$a \in \Q^\times$ since $(\alpha_0^{}) = {\mathfrak a}^{p^e}$, 
whence $\varepsilon^\capitul_n \in E_n^{p^e}$ giving 
${\mathfrak a}$ $p$-principal in $k$ (absurd);
a capitulation in $k^\cyc/k$ is impossible.

\smallskip
(ii) Note that if $\tau$ is the complex conjugation, 
$\big (\langle \varepsilon^\capitul_n \rangle E_n^{p^e}/E_n^{p^e}\big)^{1-\tau} 
= 1$ for $k^\cyc$ (because $\varepsilon^\capitul_n$ is real), while 
$\big(\langle \varepsilon^\capitul_n \rangle 
E_n^{p^e}/E_n^{p^e}\big)^{1+\tau} = 1$ for $k^\acyc$ (use
$\alpha_0^{} = \varepsilon^\capitul_n \, \alpha_n^{p^e}$,
and $\alpha_0^{1+\tau} = \BN_{k/\Q}(\alpha_0^{}) =  
\BN_{k/\Q}({\mathfrak a}^{p^e})=a^{p^e}$, $a \in \Q^\times$, giving
$(\varepsilon^\capitul_n)^{1+\tau} = (a \cdot \alpha_n^{-(1+\tau)})^{p^e}
\in E_n^{p^e}$).
This defines plus and minus components, in the meaning
of $p^{e\,{\rm th}}$-powers; in particular, for $k^\acyc$:
$$\langle \varepsilon^\capitul_n \rangle \, E_n^{p^e}/E_n^{p^e}
= \big(\langle \varepsilon^\capitul_n \rangle \, E_n^{p^e}/E_n^{p^e} \big)^-
 \simeq \Z/p^e \Z, $$

\noindent
which is coherent with the 
arithmetical context given by $k^\acyc/k$ where minus components
play the essential role and must be non-trivial (see \S\,\ref{minus}); in other 
words, elements of order $p^e$ of $(E_n/E_n^{p^e})^-$ are necessary to 
obtain capitulations, in $K_n$, of classes of $k$.  See Appendix \ref{capunit} 
for a numerical example giving the unit $\varepsilon^\capitul_n$ for $n=1$.
\end{remarks}

\subsection{Artin group of the first layer of a \texorpdfstring{$\Z_p$}{Lg}-extension}
Let $L$ be any number field, and let $H_L^\pr$ be the 
maximal abelian $p$-ramified pro-$p$-extension of $L$. 
For more references and generalities on the $\Log_p$-function 
defined on $\Gal(H_L^\pr/L)$, giving a canonical description of the 
Galois groups of the subfields of the compositum $\wt L$ of the 
$\Z_p$-extensions of $L$, see our original paper \cite{Gra1985}, taken 
into account in our book for a general presentation. 
See also \cite{Jau1998} and the large bibliography of \cite{Jau2016a} 
about infinitesimal arithmetic.

\smallskip
Recall that this $p$-adic class field theory is based on the isomorphisms: 
\begin{equation}
\Gal(\wt L/L) \simeq \Log_p (I_L \otimes \Z_p)\ \ \& \ \
\Gal(\wt L/\wt L \cap H_L^\nr) \simeq \Log_p (P_L \otimes \Z_p),
\end{equation} 

\noindent
where $H_L^\nr$ is the $p$-Hilbert class field of $L$, $I_L$ (resp. $P_L$) 
the group of prime-to-$p$ ideals (resp. principal ideals) of $L$, with: 
\begin{equation}\label{calcul}
\Log_p({\mathfrak a}) = \ffrac{1}{p^e}\log_p(\alpha) := \ffrac{1}{p^e}
\big(\log_{\mathfrak p}(\alpha) \big)_{{\mathfrak p} \mid p} \in 
\hbox{$\prod_{{\mathfrak p} \mid p}$}\, L_{\mathfrak p}
\pmod {\log_p(E_L \otimes \Q_p)}, 
\end{equation} 
as soon as  ${\mathfrak a} \in I_L \otimes \Z_p$
is such that ${\mathfrak a}^{p^e} =: (\alpha)$, $\alpha \in L^\times \otimes \Z_p$,
where $E_L$ is the group of units of $L$. 
Of course, $\log_p(E_L \otimes \Q_p)$ only depends on Leopoldt's 
conjecture but not of the units strictly speaking.
If $L=k = \Q(\sqrt{-m})$, $\Log_p({\mathfrak a}) = \frac{1}{p^e}\log_p(\alpha)$, 
since $\log_p(E_k \otimes \Q_p) = 0$.

\smallskip
The map $\Log_p : \Gal(H_L^\pr/L) \to \Gal(\wt L/L) 
\simeq \Log_p (I_L \otimes \Z_p)$ must be understood 
as the map $\Log_p \circ \Artmap^{-1}$, where
$\Artmap$ is the Artin symbol defined on $I_L \otimes \Z_p $
with values in $\Gal(H_L^\pr/L)$, giving the exact sequence 
\cite[Theorem III.2.5, \S\,III.6]{Gra2005}:
$$1 \too \CT_L \tooo \Gal(H_L^\pr/L) \mathop{\relbar\lien
\relbar\lien\relbar\lien\relbar\lien\toooo}^{\Log_p \circ \Artmap^{-1}} 
\Log_p (I_L \otimes \Z_p) \too 1. $$

Recall that in a ``ray class group'' viewpoint, one may write 
\cite[Theorem III.2.4]{Gra2005}:
\begin{equation}
\Gal(H_L^\pr/L) \simeq I_L \otimes \Z_p \hbox{$\big/ \bigcap_{n>0}$}
\big(P_{L, (p^n)} \otimes \Z_p \big), 
\end{equation}
with ray-groups  $P_{L, (p^n)}$ of modulus $(p^n)$; it is clear that 
$\Log_p$ is trivial on $\bigcap_{n>0} \big(P_{L, (p^n)} \otimes \Z_p \big)$
and that $\Ker(\Log_p) = \CT_L = \tor_{\Z_p}^{}(\Gal(H_L^\pr/L))$.

\medskip
In the present paper, we restrict ourselves to the case  
$k=\Q(\sqrt{-m})$ and introduce $M := k(\mu_p^{})$, $p \ne 2$; we will 
apply the following result to $L=k$ or $L=M$:

\begin{theorem} \label{artingroup}
Denote by $H_L^\pr$ the maximal abelian $p$-ramified pro-$p$-extension 
of $L$ and let $H_{L,1}^\pr$, of Artin group $\Art(\hbox{\ft$H_{L,1}^\pr$})$, 
be the maximal elementary $p$-sub-extension of $H_L^\pr/L$.

\smallskip
(i) Let ${\mathfrak b}^{}_1, \ldots , {\mathfrak b}_{t_L} \in I_L \otimes \Z_p$, 
where $t_L$ is the $p$-rank of $\CT_L$, be such that:
\begin{equation} \label{artin}
\langle\, {\mathfrak b}^{}_1, \ldots , {\mathfrak b}_{t_L} \,\rangle 
\cdot \Art(\hbox{\ft$H_{L,1}^\pr$}) = \big\{ {\mathfrak a} \in I_L \otimes \Z_p,\ 
\Log_p({\mathfrak a}) \in p \,\Log_p(I_L \otimes \Z_p)\big \}. 
\end{equation}

\smallskip
Then the Artin group of $\wt L_1$ is the left member of \eqref{artin}; 
so $\Gal(\wt L/\wt L_1) \simeq p\,\Log_p (I_L \otimes \Z_p)$. 

\medskip
(ii) Put $M=L(\mu_p^{})$, let $\{w M^{\times p},\ M(\sqrt[p]{w}) \subseteq 
H_{M,1}^\pr \,\rangle\}$ be the Kummer radical of $H_{M,1}^\pr$, and then
$\{w M^{\times p},\ M(\sqrt[p]{w}) \subseteq \wt M_1 \,\rangle\}$
that of $\wt M_1$. 
We consider the analogue of \eqref{artin} for $M$. The Artin group 
of $\wt k_1$, obtained by taking the relative norm in $M/k$, is:
\begin{equation} \label{artink}
\langle\, {\mathfrak b}^{}_1, \ldots , {\mathfrak b}_{t_k} \,\rangle 
\cdot \Art(\hbox{\ft$H_{k,1}^\pr$}) = \big\{ {\mathfrak a} \in I_k \otimes \Z_p,\ 
\Log_p({\mathfrak a}) \in p \,\Log_p(I_k \otimes \Z_p)\big \}. 
\end{equation}
\end{theorem}

In other words, one uses Artin automorphisms of ideals ${\mathfrak b}$ 
fulfilling some $p$-adic congruences from the condition $\Log_p({\mathfrak b}) 
\in p \,\Log_p(I_L \otimes \Z_p)$; this condition depends on 
$\Log_p(I_L \otimes \Z_p)$ and if $I_L = \langle {\mathfrak a}_1, \ldots, 
{\mathfrak a}_r \rangle\,P_L$, then
$\Log_p(I_L \otimes \Z_p) = \big\langle \Log_p({\mathfrak a}_i) \big \rangle_{\Z_p}
+ \Log_p(P_L \otimes \Z_p)$,  
where $\Log_p(P_L \otimes \Z_p)$ is standard regarding $L$ 
and may be computed once for all. The $\Log_p({\mathfrak a}_i)$'s 
are a priori random in the free module 
$p^\Z \cdot \CO_L$, where $\CO_L$ is the ring of integers of the 
local algebra $\prod_{{\mathfrak p} \mid p}\,L_{\mathfrak p}$.

\smallskip
We will proceed differently without precise $p$-adic computations 
from the above and \eqref{calcul}, noting that the index 
$\big [\Log_p(I_L \otimes \Z_p) : \Log_p(P_L \otimes \Z_p) \big]$ 
depends on orders of classical invariants of $L$ 
that {\sc pari/gp} has in its library (see Theorem \ref{Val}).

\subsection{Minus and plus components attached to 
\texorpdfstring{$H_k^\pr/k$}{Lg}}\label{minus}
The totally imaginary field $H_k^\pr$ is not a CM-field, but one may 
define {\it plus}  and {\it minus} components as follows: 

\medskip
Put $g := \Gal(k/\Q)= \{1, \tau\}$. Since $k$ is imaginary, complex
conjugation $\tau$ operates on various objects $X_k$ associated to
the arithmetic of $H_k^\pr$; we will say that $X_k$ is a {\it minus} component if 
$x^\tau = x^{-1}$ for all $x \in X_k$ and a {\it plus} component if $x^\tau = x$ 
for all $x \in X_k$. For $p \ne 2$, any $\Z_p[g]$-module $X_k$ is of the 
form $X_k = X_k^+ \oplus X_k^-$, where $X_k^+ = X_k^{\frac{1+\tau}{2}}$
and $X_k^- = X_k^{\frac{1-\tau}{2}}$. 

\smallskip
For instance, $\CH_k \simeq \Gal(H_k^\nr/k)$, $\CT_k \simeq 
\Gal(H_k^\pr/\wt k)$ and $\Gal(k^\acyc/k)$ are {\it minus} components, 
$\Gal(k^\cyc/k) \simeq \Gal(\Q^\cyc/\Q)$ is a {\it plus} component.
In particular, we will use the relations:
$$\hbox{$\Gal(k^\acyc/k) \simeq \Log_p(I_k \otimes \Z_p)^-$, 
$\CO_k^- \simeq \Z_p\,\sqrt{-m}$, }$$
where $\CO_k = \prod_{{\mathfrak p}\mid p}\CO_{\mathfrak p}$ is the
product of the local rings of integers of the $k_{\mathfrak p}$'s; note that
$\CO_{\mathfrak p} = \Z_p$ (resp. $\CO_{\mathfrak p} = \Z_p \oplus 
\Z_p \,\sqrt{-m}$) if $p$ splits (rep. is inert or ramified) in $k/\Q$.

\begin{remarks}
(i) An element $x \in X$ is in $X^-$ if and only if $x^{1+\tau} = 1$. For instance, 
if $X = k^\times/k^{\times p}$ and $(\alpha) = {\mathfrak a}^p$ in $k$, 
the class $x$ of $\alpha$ in $X$ is in the minus component since 
$\alpha^{1+\tau} = a^p$ for $a = \BN{\mathfrak a}$ (absolute norm);
so, $x \in X^-$. Most often, such remarks will be implicit. More generally, 
for $M = k(\mu_p^{})$, we will work in the $\Gal(M/\Q)$-module 
$\Gal(H_M^\pr/M)$ and the invariants of $H_M^\pr$, decomposed 
by means of the characters of $\Gal(M/\Q)$.

\smallskip
(ii) Let $\chi$ be the quadratic character of $k$. The case of minus components, 
discussed above for $k$, corresponds to objects $X_k$ of $M_\chi = k$ of norm $1$ 
in $k^\times/k^{\times p}$, whence $\BNu_{k/\Q} (x) := x^{1+\tau} = 1$ for all $x \in X_\chi$. 
In other words, $X_\chi = X_k^-$. 
\end{remarks}

\subsection{\texorpdfstring{$p$}{Lg}-ranks of ray class fields}
Since $\CT_k$ is a direct factor of $\Gal(H_k^\pr/k)$, the structure of 
$\Gal(H_k^\pr/k)$ giving $\rkp(\CT_k) = 
\rkp\big (\Gal(H_k^\pr/k)^-\big)-1$ may be analyzed at a finite steps by 
means of $\Gal(H_k(p^\nu)/k) \simeq \CH_k(p^\nu)$ for a $p$-ray class 
field of modulus $(p^\nu)$. Since {\sc pari/gp} gives the structure 
of ray class groups $\CH_k(p^\nu)$, the computations of the 
structure of $\CT_k$, and in particular of its $p$-rank, are effective
(see proof of Theorem \ref{T}). We will recall in Section \ref{calculT},
from \cite[Section 2]{Gra2019c}, that the minimal $\nu_0^{}$ needed 
to obtain the $p$-rank of $\CT_k$ is $\nu_0^{}=2$ (resp.~$3$) for 
$p \ne 2$ (resp. $p=2$); for the structure of $\CT_k$, in our context, 
we will deduce the minimal value of $\nu$ from the knowledge of the 
exponent of $\CH_k$.

\begin{theorem}  \label{thmfond1}
(i) For ${\mathfrak p} \mid p$ in $k$ and $j\geq 1$, let 
$\CU_{\mathfrak p}^j = 1+ {\mathfrak p}^j\CO_{\mathfrak p}$.
For $m \geq n \geq 1$, we have
$0 \leq \rkp(\CH_k(p^m))-\rkp(\CH_k(p^n)) \leq \hbox
{$\sum_{{\mathfrak p} \mid p}$\ } \rkp \big((\CU_{\mathfrak p}^1)^p\,
\CU_{\mathfrak p}^{n \, \cdot\, e_{\mathfrak p}}/
(\CU_{\mathfrak p}^1)^p\,\CU_{\mathfrak p}^{m \,\cdot \, e_{\mathfrak p}} \big)$, 
where $e_{\mathfrak p}$ is the ramification index of ${\mathfrak p}$ in $k/\Q$.

\smallskip
(ii) $\rkp(\CH_k(p^m)) = \rkp(\CH_k(p^n)) = \rkp(\Gal(H_k^\pr/k))$,
for all $m \geq n > \ffrac{p}{p-1}$. 

(iii) $H_{k,1}^\pr \subseteq H_{k,1}(p^2)$, 
whence $k_1^\acyc \subseteq H_{k,1}(p^2)^-$. 
\end{theorem}

\begin{proof} 
(i) From \cite[Theorem I.4.5 \& Corollary I.4.5.4]{Gra2005} 
applied to the set of $p$-places and for the ordinary sense.

\smallskip
(ii) It is sufficient to get, for some $n \geq 1$,
$(\CU_{\mathfrak p}^1)^p\,\CU_{\mathfrak p}^{n  \,\cdot \, e_{\mathfrak p}} 
= (\CU_{\mathfrak p}^1)^p$,
hence $\CU_{\mathfrak p}^{n \,\cdot \, e_{\mathfrak p}} \subseteq 
(\CU_{\mathfrak p}^1)^p$, for all ${\mathfrak p} \mid p$; 
indeed, we then have
$\rkp(\CH_k(p^n)) = \rkp(\CH_k(p^m)) = \rkp(\Gal(\wt k/k))+ \rkp(\CT_k)$, 
as $n \to \infty$, giving $\rkp(\CH_k(p^n)) = 2+ \rkp(\CT_k)$ for such $n$'s. 
The condition $\CU_{\mathfrak p}^{n \,\cdot \, e_{\mathfrak p}} \subseteq 
(\CU_{\mathfrak p}^1)^p$ is fulfilled as soon as
$n \cdot e_{\mathfrak p} > \ffrac{p \cdot e_{\mathfrak p}}{p-1}$, hence
$n > \ffrac{p}{p-1}$ (see \cite[Chap. I, \S\,5.8, Corollary~2]{FV2002} or 
\cite[Proposition 5.7]{Wa1997}), whence the minimal $n$, 
giving $\nu_0^{} = 2$ for $p \ne 2$ and $\nu_0^{} = 3$ for $p = 2$; 
furthermore, for $p \geq 3$, $\CH_k(p^2)$ gives the $p$-rank of $\CT_k$.
\end{proof}

\subsection{Technical conditions for \texorpdfstring{$\Log_p({\mathfrak q}) 
\in p\, \Log_p(I_k \otimes \Z_p)$}{Lg}}\label{mainremarks}

We consider the expression \eqref{artink} in Theorem \ref{artingroup}, where 
the ${\mathfrak b}_i$'s are replaced by enough prime ideals ${\mathfrak q}$ 
as explained. In the practice, we suppose $p = 3$ for programming which is 
specific for an imaginary quadratic field; but the results hold for any prime 
$p \nmid [k : \Q]$, where $k$ is an arbitrary abelian imaginary number field, 
and for its odd characters $\chi$, since only the search of the Kummer 
radicals depends on $p$ via the definition of $\chi^* := \omega_p\,\chi^{-1}$, 
function of the cyclotomic character $\omega_p$. 

\smallskip
By abuse, we will often confuse a class $w M^{\times p}$ with a 
representative $w \in  M^\times$.

\smallskip
Since the sequel depends on the minus part of the schema of 
$H_k^\pr/\wt k/k^\acyc/k$, the following diagram may be useful when 
$p=3$, $m \equiv 3 \pmod 9$ in which case $\CT_k = \CT_k^\bp \oplus 
\CW_k^\bp$ with $\CW_k^\bp \simeq \Z/3\Z$, and when $H_k^\nr$ is 
not linearly disjoint from $k^\acyc$ (whence $k_1^\acyc/k$ unramified). Only 
$H_k^\nr$, $H_k'^\bp := k^\acyc H_k^\nr$ and $H_k^\bp = \wt k H_k^\nr$ 
are canonical.
Since $\Gal(\wt k/k^\acyc)$ and $\Gal(k^\acyc/k)$ are free, the existence 
of three direct compositum $H_k'^\pr = H_k'^\bp F''_0$, $F'=H_k^\nr F'_0$,
and $F=H_0 F_0$ is clear.
The extension $H_k'^\bp/k^\acyc$ is unramified ({\sc nra}), and the extensions 
$H_k'^\pr/H_k'^\bp$ and $F'/H_k^\nr$ are ramified ({\sc ram}), of degree $3$.

\smallskip  
Some notations have been simplified, as $\Log(I)^- := \Log_p(I_k \otimes \Z_p)^-$:

\unitlength=1.55cm 
\begin{equation}\label{schemalog}
\begin{aligned}
\vbox{\hbox{\hspace{-0.5cm}
\begin{picture}(9.2,7.2)
\put(6.05,4.5){\line(3,1){1.98}}
\put(6.05,6.45){\line(3,1){1.98}}
\put(4.05,4.55){\line(3,1){1.95}}
\put(3.85,6.5){\line(3,1){1.98}}
\put(8.1,5.1){$F'$}
\put(6.0,5.1){$F'_0$}
\put(6.0,0.7){$F_0$}
\put(4.34,0.42){\tiny {\sc ram}}
\put(4.34,0.42){\tiny {\sc ram}}
\put(6.95,0.86){\tiny {\sc nra}}
\put(6.95,5.26){\tiny {\sc nra}}
\put(8.26,6.2){\tiny {\sc ram}}
\put(3.8,6.45){\line(1,0){1.8}}
\put(4.05,4.50){\line(1,0){1.5}}
\put(6.25,7.2){\line(1,0){1.8}}
\bezier{300}(3.65,6.36)(6.8,5.2)(8.0,7.0)
\put(4.8,6.55){\tiny$\simeq \! \CT_k^\bp$}
\put(4.7,6.3){\tiny {\sc nra}}
\put(6.45,5.78){\ft$\CT_k$}
\put(6.55,6.82){\tiny$\simeq \! \CW_k^\bp$}
\put(6.76,6.56){\tiny {\sc ram}}
\put(4.9,4.55){\tiny {\sc nra}}
\put(4.9,4.98){\tiny {\sc ram}}
\put(7.0,4.65){\tiny${\mathcal W}'^\bp_k$}
\put(6.5,4.8){\tiny {\sc ram}}
\put(8.2,5.4){\line(0,1){1.6}}
\put(8.2,1.0){\line(0,1){4.0}}
\put(5.8,0.2){\line(0,1){4.0}}
\put(8.1,0.7){$F$}
\put(5.63,-0.04){$H_0$}
\put(4.6,0.08){\tiny {\sc nra}}
\put(6.5,0.4){\tiny {\sc ram}}
\put(5.8,4.66){\line(0,1){1.6}}
\put(5.45,5.5){\tiny {\sc ram}}
\put(3.50,4.7){\line(0,1){1.6}}
\bezier{40}(6.15,0.9)(6.15,3.0)(6.15,5.1)
\bezier{40}(6.15,5.1)(6.15,3.0)(6.15,0.9)
\bezier{15}(6.15,5.4)(6.15,6.2)(6.15,7.0)
\bezier{15}(6.15,7.0)(6.15,6.2)(6.15,5.4)
\bezier{40}(6.15,0.9)(6.152,3.022)(6.15,5.1)
\bezier{40}(6.15,5.1)(6.152,3.022)(6.15,0.9)
\bezier{15}(6.15,5.4)(6.152,6.22)(6.15,7.0)
\bezier{15}(6.15,7.0)(6.152,6.22)(6.15,5.4)
\bezier{100}(3.8,0.1)(4.9,0.45)(6.0,0.8)
\put(3.50,1.5){\line(0,1){2.8}}
\put(3.50,0.3){\line(0,1){0.8}}
\put(3.35,1.25){$k_1^\acyc$}
\put(3.44,0.04){$k$}
\put(3.55,0.7){\ft$p$}
\bezier{400}(3.3,1.3)(2.2,3.4)(3.4,6.3)
\put(1.9,2.6){\ft$p\Log(I)^{\!-}$}
\bezier{500}(3.3,0.2)(0.0,3.2)(3.3,6.4)
\bezier{100}(3.6,4.7)(3.9,5.1)(3.6,6.3)
\put(3.75,5.0){\ft$\Log(P)^{\!-}$}
\put(0.88,3.1){\ft$\Log(I)^{\!-}$}
\bezier{400}(3.7,0.15)(5.5,1.8)(5.65,4.3)
\put(4.3,3.5){\ft$\big\langle \!\Log({\mathfrak q})^{\!-}\big \rangle$}
\put(3.12,5.34){\tiny {\sc ram}}
\put(3.12,3.55){\tiny {\sc nra}}
\put(3.12,2.0){\tiny {\sc nra}}
\put(3.12,0.7){\tiny {\sc nra}}
\put(4.45,6.86){\tiny {\sc ram}}
\put(6.95,7.25){\tiny {\sc nra}}
\put(8.26,3.0){\tiny {\sc nra}}
\bezier{300}(3.7,2.78)(5.7,5.2)(3.65,6.3)
\put(3.4,2.7){$-$}
\put(4.9,1.55){\ft$\CH_k$}
\put(5.45,2.1){\tiny {\sc nra}}
\put(5.6,6.4){$H_k'^\bp$}
\put(8.1,7.12){$H_k'^\pr$}
\put(5.92,7.15){$F''_0$}
\put(3.45,6.4){$k^\acyc$}
\put(5.55,4.4){$H_k^\nr$}
\put(4.85,4.32){\tiny$\CH'_k$}
\put(2.95,4.42){$k^\acyc \!\cap\! H_k^\nr$}
\put(6.1,0.1){\line(3,1){1.95}}
\bezier{75}(6.2,0.8)(7.15,0.8)(8.1,0.8)
\bezier{75}(6.2,5.2)(7.15,5.2)(8.1,5.2)
\put(3.8,0.02){\line(1,0){1.75}}
\end{picture} }} 
\end{aligned}
\end{equation}
\unitlength=1.0cm 

\medskip
Prime ideals ${\mathfrak q} \mid q$, split in $k$, will be auxiliary ideals of 
which we will calculate the logarithm to obtain primes ${\mathfrak q}$ 
such that their Artin symbols $\Big (\ffrac {H_k'^\pr/k}{{\mathfrak q}}\Big)$   
be in $\Gal(H_k'^\pr/k_1^\acyc)$, in other words such that ${\mathfrak q}$ 
splits in $k_1^\acyc/k$ and be inert in some of the other cubic extensions
that are eliminated; this 
method simplifies and generalizes in a systematic way many computations 
in the literature as in Brink's Examples \cite[Section IV]{Br2007} where 
similar estimations of $\hbox{\sc Val}$ are obtained in particular cases ($m=21$, 
$107$ for $p=3$).

\smallskip
Let $\hbox{\ft{\sc exp}}(k)$ be the exponent of the $p$-class group $\BH_k$ 
and let $\hbox{\ft{\sc expta}}(k)$ be the exponent of the tame part of $\BH_k$.
Let $q$ be a prime number, split in $k$, and let ${\mathfrak q} \mid q$; put:
$${\mathfrak q}^{\hbox{\tiny{\sc exp}}(k) \cdot \hbox{\tiny{\sc expta}}(k)} =: 
(\beta), \ \, \beta \in k^\times,\  {\rm integer}. $$

By definition, $\Log_p({\mathfrak q}) = 
\ffrac{1}{\hbox{\tiny{\sc exp}}(k) \cdot \hbox{\tiny{\sc expta}}(k)}\,
\log _p(\beta)$ and $\Log_p({\mathfrak q})^- \in p \,\Log_p(I_k \otimes \Z_p)^-$
is equivalent to
$\log_p(\beta)^- \in p \cdot \hbox{\ft{\sc exp}}(k) \cdot \Log_p(I_k \otimes \Z_p)^-$. 
So, it suffices to  compute $\Log_p(P_k \otimes \Z_p)^-$ 
and to obtain a relation between $\Log_p(I_k \otimes \Z_p)^-$ 
and $\Log_p(P_k \otimes \Z_p)^-$; this relation does exist 
from class field theory interpretation of the cyclic quotient $\Log_p
(I_k \otimes \Z_p)^-/\Log_p(P_k \otimes \Z_p)^-$ (Theorem \ref{Val}\,(i)). 

\begin{lemma}\label{logP}
Let $\CO_k$ be the ring of integers of the algebra 
$\prod_{{\mathfrak p}\mid p} k_{\mathfrak p}$. Then $\CO_k^- = 
\Z_p \, \sqrt{-m}$ and $\Log_p(P_k \otimes \Z_p)^-$ is given by:

\smallskip
$\bullet$ $p \CO_k^- = p\, \Z_p \, \sqrt{-m}$ if $p \nmid m$, 
where $\sqrt{-m}$ is a local unit in that case;

\smallskip
$\bullet$ $\CO_k^- = \Z_p \, \sqrt{-m}$ if $p \mid m$,
$m \not \equiv 3 \!\pmod 9$ when $p=3$, where $\sqrt{-m}$ 
is then an uniformizing parameter; 

\smallskip
$\bullet$ $3 \CO_k^- = 3\, \Z_3 \, \sqrt{-m}$ if $p=3$ and
$m \equiv 3 \!\pmod 9$.
\end{lemma}

\begin{proof}
The computation is immediate since $\Log_p(P_k \otimes \Z_p) =
\log_p(\CU_k^1) = \pi\, \CO_k$, $\pi \in \{p, \sqrt{-m}\}$ when
$p \nmid m$ (resp. when $p \mid m$) with the relation $\CU_k^1 = 
\mu_3 \oplus \CU_k^2$ if $p=3$ and $m \equiv 3 \!\pmod 9$ which implies 
$\log_3(\CU_k^1)^- = \log_3(\CU_k^2)^- = (\log_3(1+3\CO_k))^- =
3 \CO_k^- = 3\, \Z_3 \, \sqrt{-m}$.
\end{proof}

\subsection{Description of the method \texorpdfstring{($p = 3$)}{Lg}}
We assume $p=3$ since Kummer radicals are needed.
One proceeds as follows to find prime ideals ${\mathfrak q}$
such that $\Log_3({\mathfrak q})^- \in 3 \,\Log_3(I_k \otimes \Z_3)^-$:

\begin{algorithm}\label{method}{\rm
(i) This algorithm uses a large enough modulus ${3^{\hbox{\tiny \sc Val}}}$, 
computed with classical invariants of $k$, to ensure that in the writing
$\log_3(\beta) = C_0+C_1\,\sqrt{-m} \in \Z_3[\sqrt{-m}]$, hence
$\log_3(\beta)^- = C_1\,\sqrt{-m}$, the condition ${\bf v}_3^{}(C_1) 
\geq \hbox{\sc Val}$, in terms of $3$-adic valuations ${\bf v}_3^{}$, 
is sufficient to imply the condition $\Log_3({\mathfrak q})^- \in 3 \times 
\Log_3(I_k \otimes \Z_3)^-$ (see Schema \eqref{schemalog});  
whence giving Artin symbols $\Big (\ffrac {H_k'^\pr/k}{{\mathfrak q}}\Big)
\in \Gal(H_k'^\pr/k_1^\acyc)$ for ${\mathfrak q} \mid q$, $q \in \GI$, 
where $\GI$ is any interval of primes split in $k$. In other words, all these
${\mathfrak q}$'s split in $k_1^\acyc$.

\smallskip
(ii) One hopes that some ${\mathfrak q}$'s are inert in another tested cubic 
field $K_1 \subset H_k^\pr$, thus eliminating the corresponding radical 
$w \in W_{\chi^*}$ (inertia equivalent to the irreducibility in $\Q_q[x]$ 
of the polynomial defining $K_1$ and denoted ${\sf Qq}$ in the 
programs). Then (if any) the {\it unique} solution ${w^\acyc}$ 
defines ${k_1^\acyc}$, because all ideals ${\mathfrak q}$, 
such that $\Log_3({\mathfrak q})^- \in 3 \,\Log_3(I_k \otimes \Z_3)^-$, 
are split in this selected cubic field, but not all split in the other cubic fields.

\smallskip
Therefore, ${w^\acyc}$ passed the decomposition test for 
all ${\mathfrak q}$, while each of the others ${w}$'s failed for at
least an ideal ${\mathfrak q}$. This can be seen 
when the program is running, because when $w = w^\acyc$, all 
primes $q \in \GI$ are tested, which takes a noticeably longer time. 

\smallskip
(iii) It is essential to prove that this may define a unique solution which is 
not some $w \ne w^\acyc$ defining $K_1 \ne k_1^\acyc$. If so, this means that 
there exists $q \in \GI$, such that $\Big (\ffrac {H_k'^\pr/k}{{\mathfrak q}}\Big)$ 
fixes only $K_1$ and then is such that ${\mathfrak q}$ is inert in $k_1^\acyc/k$ (so 
wrongly eliminating $w^\acyc$).
Indeed, assume that this occurs; since $\Gal(k^\acyc/k) \simeq \Z_3$, it 
follows that ${\mathfrak q}$ is totally inert in $k^\acyc/k$,  
that $\Log_3({\mathfrak q})^-$ is a generator of $\Log_3(I_k \otimes \Z_3)^-$
and does not belong to $3 \, \Log_3(I_k \otimes \Z_3)^-$ (absurd). Whence 
the significance of the parameter $\hbox{\sc Val}$ (Theorem \ref{Val}\,(ii)). 
Experiments show that (due to Chebotarev's density Theorem) such primes 
$q$ are small and numerous.

\smallskip
(iv) At the end of a running, the program gives the list ${\sf List\Sigma}$
of the $m$'s for which there is several (wrong) solutions, because the 
interval $\GI$ of auxiliary primes $q$ is not sufficient, so that point (iii) 
does not apply; in practice, this list is empty, otherwise $\GI$ must be 
enlarged.
The programs may be used for any selected list or any interval 
of integers $m$ defining $k$.}
\end{algorithm}

\begin{algorithm}\label{abcd}{\rm
To test capitulations in $k_1^\acyc$ of some classes of $k$,
the principle is to use the relation given by the 
algebraic norm on the set of finite places of a number field:
\begin{equation}\label{JN}
\BNu_{k_1^\acyc/k} = \BJ_{k_1^\acyc/k} \circ \BN_{k_1^\acyc/k},
\end{equation}
where $\BJ$ is the transfer map (or extension of ideals) and $\BN$ the 
usual arithmetic norm. In the totally ramified case, $\BN_{k_1^\acyc/k}$ 
is surjective by class field theory, and $\BNu_{k_1^\acyc/k}
(\CH_{k_1^\acyc}) = \BJ_{k_1^\acyc/k}(\CH_k)$; so, partial capitulation 
depends on the data obtained with the algebraic norm.

\smallskip
In the case $k_1^\acyc/k$ unramified, $\BN_{k_1^\acyc/k}(\CH_{k_1^\acyc}) 
=: \CH'_k \simeq \Gal(H_k^\nr/k^\acyc \cap H_k^\nr)$ is a subgroup of $\CH_k$
of index $3$, and partial capitulation (from $\CH'_k $) may exist.

\smallskip
The program computes 
$\BNu_{k_1^\acyc/k}(\CH_{k_1^\acyc})$, using the fact that {\sc pari/gp} 
gives representative ideals of the generating classes from the instruction 
${\sf kacyc.clgp}$ and allows to conjugate these ideals ${\sf X}$ via 
${\sf nfgaloisconj(kacyc)}$ describing $G_1 = \Gal(k_1^\acyc/k)$ and 
${\sf nfgaloisapply(kacyc,s,X)}$ giving ${\sf X^s}$; a line, in the output of the form (for 
example with $4$ generators $h_i^\acyc$):

\smallskip
\begin{verbatim}
Norm of the component i of H_kacyc: [a,b,c,d]
\end{verbatim}

\smallskip\noindent
denotes the class $(h_1^\acyc)^a\, (h_2^\acyc)^b\, 
(h_3^\acyc)^c\, (h_4^\acyc)^d$ written on the generators 
$h_i^\acyc$ defining the class group $\BH_{k_1^\acyc}$ of 
$k_1^\acyc$, from the instruction ${\sf kacyc.cyc}$ giving
the orders ${\sf [a,b,c,d]}$ of the $h_i^\acyc$'s,
and ${\sf bnfisprincipal(kacyc,Y)}$, ${Y = \BNu_{k_1^\acyc/k}(h_i^\acyc)}$,
which gives the structure of $\BNu_{k_1^\acyc/k}(\BH_{k_1^\acyc})$ 
to be compared with that of $\BH_k$, at the level of the $3$-Sylow's.

\smallskip
For this, the obtained matrix may be simplified by linear combinations
on the lines to get the structure, then the order, of the image. 

\smallskip
An important remark is that {\sc pari/gp} uses random prime numbers 
in some delicate computations, so that numerical results may vary 
during several runs of the program, but generators, matrices, etc. are 
equivalent; of course this also depends on the {\sc pari/gp} version}.
\end{algorithm}

For practical computations, we have the following results:

\begin{proposition}\label{Q}
Let $w \in W_{\chi^*}$ (of conjugate $w'$) be the radical associated to 
a cyclic cubic field $K_1$, contained in $H_k^\pr$, and decomposed 
over $k$. Put $w = \frac{1}{2}(u+v \sqrt{3m})$ if $3 \nmid m$ (resp. 
$w = \frac{1}{2}(u+v \sqrt{m/3})$ if $3 \mid m$), $u,v \in \Z$.
The field $K_1$ is given, over $k$, by the irreducible polynomial 
of $\sqrt[3]{w}+\sqrt[3]{w'}$, $Q = x^3-3a x-t \in \Q[x]$, where $a^3 = w w'$ 
and $t = w+w'$. 
\end{proposition}

Nevertheless, some circumstances are tricky for the {\sc pari/gp}
determination of the $w$'s and we will define the following four 
cases regarding $k^*$: 

\begin{definition}\label{fourcases}
We distinguish the following arithmetic properties of the
reflection $k^*$ of $k$ in the mirror involution, where 
${\mathfrak p}^* \mid 3$ if $3$ splits in $k^*$:

\smallskip
(i) $m \not\equiv 3 \pmod 9$, $\CH_{k^*} \ne 1$ (Non Split case:
$3$ is inert or ramified in $k^*/\Q$);

\smallskip
(ii) $m \equiv 3 \pmod 9$, $\CH_{k^*} \ne 1$,
$\Ccl ({\mathfrak p}^*) \not\in \CH_{k^*}^3$ (Normal Split case:
$3$ splits in $k^*$ and $\Ccl ({\mathfrak p}^*)$ is ``direct factor'');

\smallskip
(iii) $m \equiv 3 \pmod 9$, $\CH_{k^*} \ne1$,
$\Ccl ({\mathfrak p}^*) \in \CH_{k^*}^3$ (Special Split case:
$3$ splits in $k^*$ and $\Ccl ({\mathfrak p}^*)$ is ``non direct factor'', 
including the case where ${\mathfrak p}^*$ is $3$-principal);

\smallskip
(iv) $\CH_{k^*}=1$, whatever the decomposition of $3$ in $k^*$ 
(Trivial case).
\end{definition} 

We see that every square-free $m \ne 3$ belongs to one and only one 
of the $4$ categories. To be more efficient, we will use four programs, 
all the more that Definition \ref{fourcases} will be characterized by means 
of relations between the $3$-ranks of the classical invariants of $k$ and 
will lead to some simplifications, specific of each case (e.g., Theorem 
\ref{disjunction}).

\section{Kummer radicals and characters}\label{characters}
We refer to any book of number theory for Kummer theory, reflection 
principles and classical inequalities of the Spiegelungssatz; some of these 
results are used in the recent literature, as for instance in \cite{HW2010, 
HW2018}.

\smallskip
But some interpretations of these inequalities, giving rise to equalities, 
are less classical; we gave in \cite[I.6, II.(b), II.5.4.9.2]{Gra2005} these 
complements using the notion of $p$-primarity. We recall these aspects 
in the case $p = 3$, $k = \Q(\sqrt{-m})$ and $M = k(\mu_3^{})$.

\subsection{Characters of \texorpdfstring{$M$}{Lg}}
If $\psi$ is an abelian character of $M$, $M_\psi$ denotes the cyclic subfield 
of $M$ fixed by $\Ker(\psi)$. 

\smallskip
Let $\chi$ be the odd character defining $k = M_\chi$, let
$\omega_3$ be the character of $\Q(\mu_3)$ and let $\chi_0^{}$ be the unit 
character defining $\Q$. Then $M$ contains the field $k^* = 
\Q(\sqrt{3m})$ (or $\Q(\sqrt{m/3})$), of even character $\chi^* = \omega_3 \chi^{-1}$.
If $X_M$ is an arithmetic $\Z_3[\Gal(M/\Q)]$-module attached to $M$, 
then $X_M=X_{\chi_0^{}} \oplus X_{\chi} \oplus X_{\chi^*}\oplus X_{\omega_3}$
where each $\psi$-component
$X_{\psi}$ is attached to $M_\psi$. For $\psi \ne \chi_0^{}$,
$X_{\psi} = \{x \in X_{M_\psi},\ \, \BNu_{M_\psi/\Q} (x) = 1\}$ and 
$X_{M_\psi} = X_{\psi} \oplus X_{\chi_0}$ .

\subsection{Units, pseudo-units and \texorpdfstring{$S$}{Lg}-units -- Radicals}

Let $M$ be a number field containing $\zeta_p$, a root of unity of order $p$. 
By definition of the $p$-ramification, the corresponding Kummer radical $W_M$ 
is the sub-module of $M^\times/M^{\times p}$, whose representative elements 
$w$ are such that:
$$(w) = {\mathfrak a}^p \cdot {\mathfrak a}_p, $$
with a prime-to $p$ ideal ${\mathfrak a}$ and an ideal ${\mathfrak a}_p$ in 
$\langle S_M \rangle$, where $S_M := \{ {\mathfrak p} \mid p, \  {\rm in}\  M \}$;
${\mathfrak a}$ and ${\mathfrak a}_p$ are unique. So, $M(\sqrt[p]{w})/M$
is a $p$-ramified cyclic extension of degree $1$ or $p$; $M(\sqrt[p]{W_M})/M$
defines the first layer $H_{M,1}^\pr/M$ of $H_M^\pr/M$. 

\smallskip
The $S_M$-units are the elements $\eta \in M^\times$ such that $(\eta) \in 
\langle S_M \rangle$; the group of $S_M$-units is of the form 
$E_M^S = \langle \eta_1, \ldots, \eta_{\hbox{\tiny ${\order S_M}$}}^{} \rangle 
\oplus E_M$ and its quotient $E_M^S/(E_M^S)^p$ is of $\F_p$-dimension 
$\order S_M+r_1+r_2$ (with $r_1+ 2r_2 = [M : \Q]$) since $\mu_p \in E_M$.

\begin{definition}\label{pseudo-units}
Let $\BH_M$ be the class group of a number field $M$ containing $\zeta_p$, 
and put:
$$\BH_M =: \big\langle \Ccl ({\mathfrak a}_1)\big\rangle \oplus 
\cdots \oplus \big\langle \Ccl ({\mathfrak a}_r)\big \rangle, $$ 
with representatives ${\mathfrak a}_i$, taken prime to $p$; for 
each $i$, put ${\mathfrak a}_i^{{e_i}} =: (\alpha_i)$, 
$\alpha_i \in M^\times$, where ${e_i}$ is the order of 
$\Ccl ({\mathfrak a}_i)$. Let $E_M := \mu_p^{} \oplus \big \langle 
\varepsilon_1, \ldots, \varepsilon_{r_1+r_2-1} \big\rangle$ be the 
group of units given by a system of fundamental units of $M$.
The group:
$$Y_M := \mu_p^{} \oplus \big\langle \varepsilon_1, \ldots, 
\varepsilon_{r_1+r_2-1}, \ \alpha_1, \ldots, \alpha_r \big\rangle $$ 
(defined modulo $p^{\rm th}$-powers, whence to be read as the 
$\F_p$-vector space $Y_M M^{\times p}/M^{\times p}$, whose 
$\F_p$-dimension is $r_1+r_2+r$) is called the group of pseudo-units 
of $M$.
\end{definition}

When $p=3$ and $M=k(\mu_3)$, the sub-radical of 
$W_M$ giving by descent over $k$ the $\chi$-component of the first 
layer $H_{k,1}^\pr$ (excluding $H_{\Q,1}^\pr$) is 
$W_{\chi^*} := (W_M)_{\chi^*} = (W_{k^*})_{\chi^*} $.

\smallskip
The elements $w \in W_{k^*}$ are such that there exist unique ideals 
${\mathfrak a}^*$, prime to $3$, and ${\mathfrak a}^*_3 \in 
\langle S^* \rangle$, where $S^* := S_{k^*}$, the set of $3$-places 
of $k^*$, such that $(w) = {\mathfrak a}^{* 3}\, {\mathfrak a}^*_3$.
When ${\mathfrak a}^*_3 = 1$, $w$ is a pseudo-unit, when 
${\mathfrak a}^* = 1$, $w$ is a $S^*$-unit, and if ${\mathfrak a}^*_3 
= {\mathfrak a}^* = 1$, then $w$ is a unit of $k^*$.  

\begin{definition}
An element $w$ of $W_{k^*}$ is said to be ``{\it fundamental}'' if 
$\langle w \rangle$ is of character $\chi^*$, whence if
$\BNu_{k^*/\Q}(w) = 1$ in $k^{* \times}/k^{* \times 3}$;
if so, we denote it by $w^*$ instead of $w$.
This defines the Kummer radical $W_{\chi^*}$ which is
essentially made up of $Y_{k^*} = \big\langle\,\varepsilon^*, 
\alpha_1^*, \ldots, \alpha_{r^*}^* \, \big\rangle$, where 
$\varepsilon^*$ is the fundamental unit of $k^*$, and,
possibly, of the two following elements under the assumption
that $3$ splits in $k^*$ ($m \equiv 3 \pmod 9$):

\smallskip
(i) the fundamental $S^*$-unit, $\eta^*$, if ${\mathfrak p}^*{}^{3^e \cdot u}
= (\eta^*)$, $e \geq 1$, $3 \nmid u$, and 
$\eta^* \notin k^{* \times 3}$; if ${\mathfrak p}^{* u} = (\eta)$
($e=0$), ${\mathfrak p}^*$ is $3$-principal, $\eta$ is not 
fundamental and will be a particular case of (ii);

\smallskip
(ii) the tricky element
$w_{{\mathfrak p}^*}^* := w_{{\mathfrak p}^*} \cdot
\BNu_{k^*/\Q}(w_{{\mathfrak p}^*})$, 
where $w_{{\mathfrak p}^*}$ is characterized by the relation 
(equivalent to $\Ccl({\mathfrak p}^*) \in \CH_{k^*}^3$):
\begin{equation} \label{tricky}
{\mathfrak p}^* = (w_{{\mathfrak p}^*}) \cdot {\mathfrak a}^{* 3}.
\end{equation}

Indeed, from this relation, we note that $\BNu_{k^*/\Q}(w_{{\mathfrak p}^*})$ 
is not in $k^{* \times 3}$; so we must instead consider $w_{{\mathfrak p}^*}^*$; 
this is crucial only for Program III. See the discussion in proof of Theorem \ref{wp}.
\end{definition}

\begin{example}
The following example shows how to avoid a possible trap in 
very rare cases, due to the fact that, if the theoretical reasonnings
work in $\CH_k = \BH_k \otimes \Z_p$, {\sc pari/gp} works only
in $\BH_k$, especially by means of the instruction
${\sf bnfisprincipal}$ testing {\it global principalities} and, if so,
giving a generator.
Let $m=146766$; in Program III, ${\sf Sideal}$ is 
$({\mathfrak p}^*)^{\hbox{\tiny{\sc expta}}(k^*)}$, where
$\hbox{\ft{\sc expta}}(k^*)$ (denoted ${\sf exptakstar}$ 
in the programs), is the tame part of the exponent of $\BH_{k^*}$: 

\smallskip
\ft\begin{verbatim}
m=146766 H_k=[60,2,2] T_k=[3,3] H_kstar=[6,2] k_1^ac/k is Ramified
Sideal=component(component(idealfactor(kstar,3),1),1);
Sideal=idealpow(kstar,Sideal,exptakstar);
Z=kstar.clgp[3]; \\=[[23,22;0,1],[29,12;0,1]]=generators [h1,h2] of H_kstar
Y=bnfisprincipal(kstar,Sideal)[1]; \\=[0,0] means Sideal principal
d=matsize(Y)[1];Ideal=1;for(i=1,d,Ideal=
idealmul(kstar,Ideal,idealpow(kstar,Z[i],Y[i])));\\=[1,0;0,1]
X=idealdiv(kstar,Sideal,Ideal); \\[9,5;0,1]
c=bnfisprincipal(kstar,X); \\=[[0,0]~,[-560257,2533]~]
Gamma=Mod(kstar.zk[1]*c[2][1]+kstar.zk[2]*c[2][2],Pstar);
wstar=Gamma*norm(Gamma);listput(Lw0,wstar);
\end{verbatim}\ns

Ones obtains $w_{{\mathfrak p}^*} = Mod( 2533*x - 560257,x^2 - 48922)$ 
in ${\sf Gamma}$ such that ${\mathfrak p}^{* 2}  = (w_{{\mathfrak p}^*})$.
The program finds the solution for the radical corresponding to $k_1^\acyc$:
$${\sf w =Mod(156253399633432764*x - 34560624129365642229, x^2 - 48922)} 
= w_{{\mathfrak p}^*}^* \cdot \varepsilon^*. $$

If the instruction ${\sf Sideal=idealpow(kstar,Sideal,exptakstar)}$
is omitted, the program does not find any solution among the $13$
radicals and gives the data:

\ft\begin{verbatim}
Y=bnfisprincipal(kstar,Sideal)[1]=[3,1]~
Ideal=[352843,295319;0,1]
X=[3,57524/352843;0,1/352843]
c=[[0,0]~,[-24780479/352843,-112036/352843]~]
Gamma=Mod(-112036/352843*x-24780479/352843,x^2-48922)
wstar=Mod(-336108/124498182649*x-74341437/124498182649,x^2-48922)
\end{verbatim}\ns

${\sf Y=[3,1]}$ means ${\mathfrak p}^* = h_1^3 \cdot h_2$
where the generators $h_i$ of $\BH_{k^*}$ are of order $6$
and $2$, respectively; so this confirms that  ${\mathfrak p}^*$
is $3$-principal, but of order $2$ in the whole class group. This gives 
${\sf w_{{\mathfrak p}^*}^* = Mod(-112036/352843*x-24780479/352843,
x^2-48922)}$ of norm $3 \times 23^{-3} \times 29^{-1}$, in other words,
${\mathfrak p}^* = (w_{{\mathfrak p}^*}^*) \,{\mathfrak l}_{23}^3 \,
{\mathfrak l}_{29}$, which does not define $w_{{\mathfrak p}^*}$
correctly.
\end{example}

\subsection{Classification of the Kummer radicals}\
The nature of the radicals may be characterized as 
follows with the arithmetic properties of $k$ and $k^*$
(cf. Definition \ref{fourcases}):

\smallskip
(i) The Non Split case ($3$ non-split in $k^*$, equivalent to 
$m \not\equiv 3 \!\!\pmod 9$, and $\CH_{k^*}\ne 1$). 
If $r^* \!= \rk(\CH_{k^*})$, $W_{\chi^*}$ is the group of 
pseudo-units $Y_{k^*}$ (Definition \ref{pseudo-units}): 
$$W_{\chi^*} = Y_{k^*} = \big\langle\,\varepsilon^*, \alpha_1^*, 
\ldots, \alpha_{r^*}^* \, \big\rangle. $$ 
Note that $Y_{k^*}$ (as many other modules) is equal to its 
$\chi^*$-component (fundamental elements). Since $3$ does 
not split in $k^*$, $S^*$-units, $S^*$-ideals do not intervene.

\smallskip
(ii) The Normal Split case ($m \equiv 3 \pmod 9$ with $\CH_{k^*}\ne 1$ 
and $\Ccl ({\mathfrak p}^*) \not \in \CH_{k^*}^3$). Then, $W_{\chi^*}$ is 
generated by $Y_{k^*}$ ({\sc pari/gp} decomposes $\BH_{k^*}$ into cyclic 
factors $\Ccl({\mathfrak a}_i^*)$ with random ${\mathfrak a}_i^*$'s, but one 
can take prime-to-$p$ representatives of the $\Ccl({\mathfrak a}_i^*)$'s).  
We will see that the fundamental $S^*$-unit is equivalent to a pseudo-unit,
which still yields $W_{\chi^*} = Y_{k^*}$ (see Theorem \ref{wp}\,(i)).

\smallskip
(iii) The Special Split case, $m \equiv 3 \!\!\pmod 9$, $\CH_{k^*} \ne 1$,
and $\Ccl ({\mathfrak p}^*) \in \CH_{k^*}^3$, which modifies the 
Kummer radicals by introducing $w_{{\mathfrak p}^*}^* :=
w_{{\mathfrak p}^*}^* \cdot \BN_{k */\Q}(w_{{\mathfrak p}^*})$, where 
${\mathfrak p}^* = (w_{{\mathfrak p}^*}) \cdot {\mathfrak a}^{* 3}$, so that
$w_{{\mathfrak p}^*}^*$ replaces the fundamental $S^*$-unit $\eta^*$ 
(cf. Theorem \ref{wp}\,(ii)). 

\smallskip
(iv) The Trivial case $\CH_{k^*}=1$; if so, there are two possibilities: 

\smallskip
\quad $\bullet$  $3$ does not split in $k^*$; thus the radical 
solution is the fundamental unit $\varepsilon^*$ of $k^*$.

\smallskip
\quad $\bullet$  $3$ splits in $k^*$ ($m \equiv 3 \pmod 9$); then the 
Kummer radical is generated by the fundamental unit $\varepsilon^*$ and 
the fundamental $S^*$-unit $\eta^*$ (the element $w_{{\mathfrak p}^*}^*$ 
does not intervene since the defining relation becomes ${\mathfrak p}^* 
= (w_{{\mathfrak p}^*}) \cdot (\alpha^{* 3}) = (\eta^*)$, hence $w_{{\mathfrak p}^*} 
= \eta^*$ in $k^{* \times}/k^{* \times 3}$).

\smallskip
In the following subsections, we describe the simplifications arising,
in the abelian $p$-ramification theory, for the particular context
of an imaginary quadratic field with $p \geq 3$.

\subsection{Abelian \texorpdfstring{$p$}{Lg}-ramification over
\texorpdfstring{$k$, $p \geq 3$}{Lg}} \label{3ramif}

Our context leads to the following schema of the maximal 
abelian pro-$p$-extension $k^\ab$ of $k$, where 
$H_k^\nr$ is the $p$-Hilbert class field and $H_k^\ta$ 
the maximal abelian tamely ramified pro-$p$-extension.
In \cite[III.4.4.1, III.2.6.1\,(Fig.\,2)]{Gra2005}, it is proved that 
$\Gal(k^\ab/H_k^\pr H_k^\ta) \simeq E_k \otimes \Z_p$. 
In the imaginary quadratic case, $E_k \otimes \Z_p \simeq \mu_p(k)$;
for $p \ne 2$, $E_k \otimes \Z_p = 1$ except for $p = 3$ and 
$k = \Q(\mu_3)$ that we exclude once for all:

\unitlength=1.2cm
\begin{equation}\label{schema1}
\begin{aligned}
\vbox{\hbox{\hspace{1.0cm}
\begin{picture}(9.0,2.0)
\put(1.6,2.0){\line(1,0){3.5}}
\put(2.6,1.75){\tiny $p$ unramified}
\put(1.6,0.18){\line(1,0){3.9}}
\put(-0.7,0.18){\line(1,0){1.4}}
\put(1.0,0.5){\line(0,1){1.20}}
\put(6.00,0.5){\line(0,1){1.20}}
\put(4.75,1.42){\tiny $p$ split in $k$}
\put(6.1,1.42){\tiny$\CU_k^1 = \CU^1_{\mathfrak p}
  \oplus \CU^1_{{\mathfrak p}'} \simeq (1+p\Z_p)^2$}
\put(4.7,1.05){\tiny $p$ inert in $k$}
\put(6.1,1.05){\tiny$\CU^1_k = 1+ p (\Z_p + \Z_p \sqrt {-m})$}
\put(4.35,0.7){\tiny $p$ ramified in $k$}
\put(6.1,0.7){\tiny$\CU^1_k = 1+ \sqrt{-m} (\Z_p + \Z_p \sqrt {-m})$}
\put(-0.1,1.45){\tiny$\CI_{\mathfrak p}
  \oplus \CI_{{\mathfrak p}'}$}
\put(-0.3,1.1){\tiny$\CI_p(H_k^\pr/k)$}
\put(-0.3,0.7){\tiny$\CI_{\mathfrak p}(H_k^\pr/k)$}
\put(5.2,1.9){$H_k^\pr H_k^\ta=k^\ab$}
\put(0.8,1.9){$H_k^\pr$} 
\put(5.8,0.1){$H_k^\ta$}
\put(0.8,0.1){$H_k^\nr$}
\put(-1.0,0.1){$k$}
\end{picture} }} 
\end{aligned}
\end{equation}
\unitlength=1.0cm

\smallskip
$\bullet$\ \, If $p$ splits in $k$, set $(p) = {\mathfrak p}\,{\mathfrak p}'$;
the inertia groups, $\CI_{\mathfrak p}(H_k^\pr/k)$ and $\CI_{{\mathfrak p}'}
(H_k^\pr/k)$, of ${\mathfrak p}$ and ${{\mathfrak p}'}$ in $H_k^\pr/k$, are 
images of the groups of principal local units $\CU^1_{\mathfrak p}$ and 
$\CU^1_{{\mathfrak p}'}$ of the completions $k_{\mathfrak p}$ and 
$k_{{\mathfrak p}'}$ (isomorphic to $\Q_p$), respectively, by the local 
reciprocity maps.
Note that $\CI_{\mathfrak p}(H_k^\pr/k)$ and 
$\CI_{{\mathfrak p}'}(H_k^\pr/k)$, isomorphic to 
$1+p \Z_p$, are free since $\Q_p$ does not contain $\mu_p^{}$.

\smallskip
$\bullet$\ \, If $p > 3$ is inert (resp. ramified) in $k$, or
if $p = 3$ is inert, the inertia group $\CI_p(H_k^\pr/k)$ 
of ${\mathfrak p}$ in $H_k^\pr/k$ is isomorphic to 
$\CU^1_k = 1+p (\Z_p + \Z_p \sqrt {-m})$ 
(resp. $1+\sqrt{-m} (\Z_p + \Z_p \sqrt {-m})$); 
it is isomorphic to $\Z_p^2$ and fixes $H_k^\nr$. 

\smallskip
$\bullet$\ \, If $p = 3$ is ramified, in the case $k_{\mathfrak p} \ne \Q_3(\sqrt{-3})$, 
$\CU^1_k$ is free; in the case $k_{\mathfrak p} = \Q_3(\sqrt{-3})$ (equivalent to 
$m \equiv 3 \pmod 9$), $\CU^1_k \simeq \Z/3\Z \times \Z_3^2$ and $3$ ramifies 
in $H_k^\pr/\wt k$ with ramification index $3$, giving $[H_k^\pr : \wt k H_k^\nr] = 3$
where $\Gal(H_k^\pr/\wt k H_k^\nr) \simeq \CW_k^\bp$.

\subsection{Schema of \texorpdfstring{$H_k^\pr/k$}{Lg}
for \texorpdfstring{$k=\Q(\sqrt{-m})$}{Lg}} \label{sch2}

We refer also to Schema \ref{schemalog}. Let's consider the sub-schema 
of \ref{schema1} about $H_k^\pr/k$, in which some simplifications hold since 
$E_k \otimes \Z_p = 1$ and $k \ne \Q(\mu_3)$ when $p=3$, as the fact that 
whatever the decomposition of $p$ in $k/\Q$, $\Gal(H_k^\pr/H_k^\nr) \simeq 
\CU_k^1$, $\Gal(H_k^\pr/\wt k\,H_k^\nr) \simeq \tor_{\Z_p}^{}(\CU^1_k) 
\in \{1, \mu_3\}$, the normalized $p$-adic regulator $\CR_k$ is trivial; 
then the sub-module $\CW_k^\bp$, fixing the Bertrandias--Payan field 
$H_k^\bp$, is $\tor_{\Z_p}^{}(\CU^1_k)$.
Recall that $\wt k \cap H_k^\nr = k^\acyc \cap H_k^\nr$ and that 
$\CT_k^\bp = \Gal(H_k^\bp/\wt k)$ is isomorphic to a subgroup 
$\CH'_k$ of $\CH_k$.

\smallskip
$\bullet$ \, If $(p) = {\mathfrak p}{{\mathfrak p}'}$ in $k$, it totally splits
in $H_k^\bp/\wt k$; indeed, since the Gross--Kuz'min conjecture holds 
for $k$, the decomposition groups (pro-cyclic) of ${\mathfrak p}$ and 
${{\mathfrak p}'}$ in $H_k^\pr/k$, are isomorphic to $\Z_p$ and 
do not intersect $\CT_k$ non-trivially (see \cite[Lemme 4]
{Jau2024a} or \cite{Jau2024b} about totally $p$-adic number fields). 
The extension $\wt k/k^\cyc$ is unramified and the decomposition 
groups of ${\mathfrak p}$ and ${\mathfrak p}'$ in this extension 
are of finite indices in $\Gal(\wt k/k^\cyc)$; similarly, $\wt k/k^\acyc$
is unramified.

\smallskip
$\bullet$ \, If $p$ does not split in $k$, 
$\CW_k^\bp \simeq \tor_{\Z_p}(\CU^1_p) \simeq \mu_3$ if and only if $p=3$,
$m \equiv 3 \pmod 9$, otherwise $\CW_k^\bp=1$; if $m \equiv 3 \pmod 9$, 
$H_k^\pr/H_k^\bp$ is ramified as seen above. Since ${\mathfrak p} \mid (p)$ 
is $p$-principal in $k$, it totally splits in $H_k^\nr/k$, hence totally splits in 
$\wt k H_k^\nr/\wt k$.

\unitlength=1.0cm 
\begin{equation}\label{schema2}
\begin{aligned}
\vbox{\hbox{\hspace{-3.0cm} 
\begin{picture}(10.0,3.0)
\put(6.6,2.46){\line(1,0){1.45}}
\put(6.6,2.52){\line(1,0){1.45}}
\put(8.8,2.50){\line(1,0){1.9}}
\put(3.8,2.50){\line(1,0){1.6}}
\put(4.15,0.50){\line(1,0){1.3}}
\put(9.2,2.15){\ft$\CW_k^\bp$}
\put(1.55,0.50){\line(1,0){1.45}}
\put(1.3,0.40){$k$}
\bezier{300}(1.38,0.35)(3.5,0.0)(5.6,0.35)
\put(3.2,-0.1){\ft$\CH_k$}
\put(4.6,0.65){\ft$\CH'_k$}
\bezier{400}(3.7,2.7)(7.25,3.3)(10.8,2.7)
\bezier{300}(3.7,2.2)(6.0,1.6)(8.3,2.2)
\put(6.0,1.55){\ft$\CT_k^\bp$}
\put(7.0,3.14){\ft$\CT_k$}
\put(3.50,0.9){\line(0,1){1.25}}
\put(5.7,0.9){\line(0,1){1.25}}
\bezier{350}(6.2,0.5)(8.5,0.6)(10.9,2.3)
\put(8.6,0.75){\ft$\CU^1_k$}
\put(10.85,2.4){$H_k^\pr$}
\put(5.6,2.4){$\wt k H_k^\nr$}
\put(8.15,2.4){${H_k^\bp}$}
\put(6.8,2.15){\ft$\CR_k\! = \!1$}

\put(4.3,2.25){\tiny${\rm split}$}

\put(3.35,2.4){$\wt k$}
\put(5.5,0.4){$H_k^\nr$}
\put(3.05,0.4){$\wt k \!\cap\! H_k^\nr$}
\put(2.35,1.5){\ft $\Z_p \!\times\! \Z_p$}
\end{picture} }} 
\end{aligned}
\end{equation}
\unitlength=1.0cm

\subsection{Equality \texorpdfstring{$\big[\Log_p
(I_k \otimes \Z_p)^- \!: \Log_p(P_k \otimes \Z_p)^- \big] = 
\ffrac{\order \CH_k}{\order \CT_k^\bp}$.
Definition of $\hbox{\sc Val}$}{Lg}}
 
Whatever the decomposition of $p$ in $k/\Q$, one has
linear disjunction of $\wt k$ and $H_k^\nr$ over $k$
(whence that of $k^\acyc$ and $H_k^\nr$), if and
only if $\CH_k \simeq \CT_k^\bp$. More precisely:

\begin{theorem}\label{Val}
Let $k = \Q(\sqrt{-m})$ and let $\hbox{{\sc exp}(k)}$ be the exponent of  
$\CH_k$ for $p \geq 3$.

\smallskip
(i) We have
$\big[\Log_p(I_k \otimes \Z_p)^- : \Log_p(P_k \otimes \Z_p)^- \big] =
\big[k^\acyc \cap H_k^\nr : k \big] = \ffrac{\order \CH_k}
{\order \CT_k^\bp}$. 

(ii)  Let ${\mathfrak q}$ be a prime ideal dividing 
a prime $q$ split in $k$.
Set $\log_p(\beta) = C_0 + C_1 \sqrt{-m}$ in  $\Z_p \oplus \Z_p \sqrt{-m}$, 
hence $\log_p(\beta)^- = C_1 \sqrt{-m}$, where $(\beta)$ is the principal ideal 
${\mathfrak q}^{\hbox{\tiny{\sc exp}}(k) \cdot \hbox{\tiny{\sc expta}}(k)}$.

\smallskip
Then the condition $\Log_p({\mathfrak q})^- \in p\,\Log_p(I_k \otimes \Z_p)^-$ 
holds as soon as (cf. Remark \ref{mainremarks}):
$${\bf v}_p^{}(C_1) \geq \hbox{\sc Val} := {\bf v}_p^{}(\hbox{\ft{\sc exp}}(k))
+{\bf v}_p^{}(\order \CT_k) - {\bf v}_p^{}({\rm h}(k))+2-{\bf v}_p^{}(m). $$
\end{theorem}

\begin{proof}
Using properties of \S\,\ref{sch2} and Schema \ref{schemalog} leads to (i).

(ii) Since $\Gal(k^\acyc \cap H_k^\nr/k)$ is cyclic, 
$\Log_p(I_k \otimes \Z_p)^- = \frac{\order \CT_k^\bp}
{\order \CH_k} \times \Log_p(P_k \otimes \Z_p)^-$, 
whence the condition $\Log_p({\mathfrak q})^- 
\in p \, \Log_p(I_k \otimes \Z_p)^-$, as soon as:
$${\bf v}_p^{}(C_1 \sqrt{-m}) \geq {\bf v}_p^{}(\hbox{\ft{\sc exp}}(k)) +
{\bf v}_p^{}(\order \CT_k^\bp)-{\bf v}_p^{}({\rm h}(k))+1 
+ \min \big[{\bf v}_p^{}(\Log_p(P_k \otimes \Z_p)^-)\big].$$

The computation of $\min \big[{\bf v}_p^{}(\Log_p(P_k \otimes \Z_p)^-)\big]
= \min \big[{\bf v}_p^{}(\log_p(\CU_k^1)^-)\big]$ (Lemma \ref{logP}) gives:
\begin{equation*}
\left \{\begin{aligned}
&{\bf v}_p^{}(p \sqrt{-m}) = 1, \ \hbox{for $p \geq 3$ unramified in $k$},\\
&{\bf v}_3^{}(\sqrt{-m}) = \ffrac{1}{2}, \ \hbox{for $p=3 \mid m$ with
$m \not \equiv 3 \!\pmod 9$}, \\
&{\bf v}_3^{}(3 \sqrt{-m}) = \ffrac{3}{2}, \ \hbox{for $m \equiv 3 \!\pmod 9$} .
\end{aligned}\right.
\end{equation*}

The above valuations of $C_1$ are respectively 
${\bf v}_p^{}(\hbox{\ft{\sc exp}}(k))+{\bf v}_p^{}(\order \CT_k^\bp)-{\bf v}_p^{}({\rm h}(k)) 
+ 1 + (1,0,1)$; since that of $\order \CW_k^\bp$ are $(0,0,1)$ and that of $m$ are 
$(0,1,1)$, the claim follows.
\end{proof}

\begin{remark}
The number $\hbox{\sc Val}$ applies to any prime ideal ${\mathfrak q}$
to test the condition $\Log_p({\mathfrak q})^- \in p\,\Log_p
(I_k \otimes \Z_p)^-$; its definition is optimal regarding this 
condition. For $p=3$, the successive maxima, up to $m=10^7$,
are given by the following list:

\ft\begin{verbatim}
m=23,Val=2   m=107,Val=3   m=586,Val=4   m=3647,Val=5   m=124847,Val=7   m=3238567,Val=8
\end{verbatim}\ns
\end{remark}

\section{Reflection properties between \texorpdfstring{$\CT_k$}{Lg} 
and the radical \texorpdfstring{$W_{\chi^*}$}{Lg}} 
As soon as Kummer theory intervenes, we suppose that $p=3$.

\subsection{Analysis of \texorpdfstring{$\Gal(H_{k,1}^\pr/k)_\chi\,$}{Lg}}

The corresponding Kummer extension
$M(\sqrt[3]{W_{\chi^*}})/M$ is decomposed over $k$ into 
the sub-extension of $H_{k,1}^\pr/k$ of character $\chi$, that is
to say its minus part. 
By definition, $\Gal(M(\sqrt[3]{W_{\chi^*}})/M)$ is isomorphic to 
$\Gal(k_1^\acyc/k) \oplus (\CT_k/\CT_k^3)$; since $\Gal(k_1^\acyc/k) 
\simeq \Z/3\Z$, we have:
\begin{equation}\label{eq1}
\rk(\CT_k) = \rk(W_{\chi^*})-1.
\end{equation}

If $3$ splits in $k^*$ (i.e., $m \equiv 3 \pmod 9$), $\CT_k$ contains the 
sub-module $\CW_k^\bp \simeq \langle \zeta \rangle$ fixing $H_k^\bp$. 
From $3$-adic class field theory \cite{Jau1998, GJN2016}, the ideal or 
id\'elique interpretation of $\zeta$ is well defined (cf. \eqref{zeta}).
Two possibilities arise in the case $m \equiv 3 \pmod 9$, 
$\CT_k = \CT_k^\bp \oplus \CW_k^\bp$ or $\CW_k^\bp \subseteq \CT_k^3$ 
(implying $\CT_k^\bp \ne 1$); then $\rk(\CT_k) = \rk(\CT_k^\bp)+ 1$, or else
$\rk(\CT_k) = \rk(\CT_k^\bp)$.

\smallskip
We intend to decide between $\CW_k^\bp$ direct factor or not of $\CT_k$, 
only using the {\sc pari/gp} computable invariants $\CT_k$, $\CH_k$ 
and $\CH_{k^*}$ (refer to Schema \eqref{schemalog} in which only 
minus components are involved, hence depend on the sub-schema
defined by $k^\acyc/k$ and the subfields $H_k'^\bp$, $H_k'^\pr$):

\begin{theorem}\label{DF}
Let $k = \Q(\sqrt{-m})$, $m \equiv 3 \pmod 9$.
Let's consider two cases:

\smallskip
(i) $k^\acyc$ is linearly disjoint from $H_k^\nr$ over $k$
($k^\acyc/k$ is totally ramified). 
Then $\CW_k^\bp$ is a direct factor of $\CT_k$ 
if and only if $\rk(\CT_k) = \rk(\CH_k)+1$.

\smallskip
(ii) $k^\acyc$ is not disjoint from $H_k^\nr$ over $k$ ($k_1^\acyc/k$ 
is unramified). Then $\CW_k^\bp$ is a direct factor of $\CT_k$ if and 
only if there exists a radical $w \in W_{\chi^*}$, of conjugate $w'$, 
defining a ramified cubique extension $K_1$ of $k$ in $H_k^\pr$.
Put, using {\sc pari/gp} notations, ${\sf R = polcompositum(Q,x^2+m)[1]}$
where ${\sf Q = x^3-3*a*x-t}$ (${\sf a^3 = w*w'}$, ${\sf t = w+w'}$); then 
$\CW_k^\bp$ is a direct factor of $\CT_k$ if and only if ${\bf v}_3^{}
({\sf nfdisc(R)})$  is strictly larger than $3$.
\end{theorem}

\begin{proof}
In case (i), we have $\CH_k \simeq \CT_k^\bp$
(see Schema \eqref{schemalog}), whence the claim.

\smallskip
In case (ii), $k^\acyc$ being non-disjoint from $H_k^\nr$, one
sees, using the Schema \ref{schemalog}, that $\CW_k^\bp$ is a direct factor of $\CT_k$ 
if and only if there exists a {\it ramified} cubic extension $K_1=F_0$ of $k$ contained
in $H_k'^\pr$; indeed, since $\Gal(H_k'^\bp/H_k^\nr)$ is free and $H_k^\nr/k$
unramified, any such cubic extension will give by composition with 
$H_k'^\bp$ a direct compositum $H_k'^\bp F_0 = H_k'^\bp\, k^\acyc \,F_0
= H_k'^\bp \,F''_0 = H_k'^\pr$ and reciprocally.
The ``discriminant formula'' 
${\mathcal D}_{K_1/\Q} = \BN_{k/\Q}({\mathcal D}_{K_1/k}) \cdot 
{\mathcal D}_{k/\Q}^3$ shows that $K_1/k$ is ramified if and only if 
${\bf v}_3^{}({\mathcal D}_{K_1/\Q}) = e$, $e > 3$, since 
${\bf v}_3^{}({\mathcal D}_{k/\Q}) = 1$.
\end{proof}

In the case (ii), if $\CW_k^\bp$ is not direct factor of $\CT_k$ ($\CW_k^\bp 
\subseteq \CT_k^3$), $F_0$ does not exist as ramified extension in the 
schema and $F'_0 = H_k^\nr$; then the first layer $H_{k,1}^\pr$ of 
$H_k^\pr$ is contained in $H_k^\nr$. These cases occur when all the 
cubic extensions of $k$ in $H_k^\pr$ are unramified ($m=5142$, $6690$, 
$6789$, $7977$, $8751$, $10173,\,\ldots$).

\subsection{Spiegelungssatz from Kummer theory}

Recall that the Scholz--Leopoldt Spiegel\-ungssatz is given by the 
inequalities $\rk(\CH_k)-1 \leq  \rk(\CH_{k^*}) \leq \rk(\CH_k)$, that 
may be specified as follows, by introducing the $3$-primary radicals
giving rise to suitable equalities:

\begin{theorem} \label{primarity}
Let $Y_{k^*} := 
\langle \varepsilon^*, \alpha_1^*, \ldots, \alpha_{r^*}^* \rangle 
\subseteq W_{\chi^*}$ be the sub-group of pseudo-units of $k^*$ 
(Definition \ref{pseudo-units}), and let $Y_{k^*}^\prim := 
\big \{\alpha \in Y_{k^*},\ \alpha \equiv u^3\! \pmod {3 \sqrt{-3}},
\hbox{in $M$}  \big\}$, be the subgroup of $3$-primary elements 
(i.e., congruent to a cube modulo $3 \sqrt{-3}$ when seen in $M$ 
and giving the unramified cubic extensions of $k$). 
Then we have the general formula \cite[II.5.4.9.2]{Gra2005}:
\begin{equation} \label{reflection}
\left \{ \begin{aligned}
\rk(\CH_k)-\rk(\CH_{k^*}) & = 1- \rk(Y_{k^*}/Y_{k^*}^\prim) \\  
& =  \rk(Y_k/Y_k^\prim) \ \in\  \{0, 1\}.
\end{aligned}\right.
\end{equation}
\end{theorem}

Note that if for instance $Y_{k^*}^\prim \ne Y_{k^*}$, there exists $\beta \in Y_{k^*}$,
non-congruent to a cube modulo $3 \sqrt{-3}$ in $M$, such that
$Y_{k^*} = \langle \varepsilon^*, \alpha_1^*, \ldots, \alpha_{r^*}^* \rangle =
Y_{k^*}^\prim \oplus \langle \beta \rangle$.

\subsection{Computation of \texorpdfstring{$\rk(W_{\chi^*})$}{Lg} from 
\texorpdfstring{$S^*$-class groups}{Lg} and \texorpdfstring{$S^*$-units}{Lg}}

We have the following exact sequences of $\chi^*$-modules, 
where $E_{k^*}^{S^*} = \langle \varepsilon^*, \eta^* \rangle$ 
is the group of fundamental $S^*$-units of~$k^*$ (i.e., of character 
$\chi^*$; so $\eta^*=1$ if $3$ does not split in $k^*$), $\CH_{k^*}^{S^*} 
:= \CH_{k^*}/\Ccl (S^*)$ the $S^*$-class group of $k^*$, and 
$\CH_{k^*}^{S^*}{}^{\sst[3]} := \{c \in \CH_{k^*}^{S^*},\ c^3 = 1\}$
whose $3$-rank is that of $\CH_{k^*}/\Ccl (S^*)$:
\begin{equation}
\left\{\begin{aligned}\label{exact}
(1) \ \ \   &1 \too E_{k^*}^{S^*}/(E_{k^*}^{S^*})^3 \tooo W_{\chi^*} 
\tooo \CH_{k^*}^{S^*}{}^{\sst[3]} \too 1,\\
(2) \ \ \   &1 \too \Ccl (S^*)/\Ccl (S^*) \cap \CH_{k^*}^{3} 
\tooo \CH_{k^*}/\CH_{k^*}^3
\tooo \CH_{k^*}/\CH_{k^*}^3 \cdot \Ccl (S^*) \too 1;
\end{aligned}\right.
\end{equation}

\noindent
in (1), one associates with $w^* \in W_{\chi^*}$, $\Ccl ({\mathfrak a}^*) 
\in \CH_{k^*}$, in the relation $(w^*) = {\mathfrak a}^*_3 \cdot 
{\mathfrak a}^{* 3}$; it is annihilated by $3$ in $\CH_{k^*}^{S^*}$. 
All these Galois modules are equal to their $\chi^*$-components.

\medskip
Let $\rho^{}_{S^*}$ be the $3$-rank of $\Ccl (S^*)/\Ccl (S^*) \cap \CH_{k^*}^{3}$
in (2); if $3$ does not split in $k$, $\rho^{}_{S^*} = 0$.
In Definition \ref{fourcases} we have called Normal Split case the case where 
$m \equiv 3 \pmod 9$ and $\Ccl ({\mathfrak p}^*) \notin \CH_{k^*}^3$ (whence 
$\rho^{}_{S^*} = 1$) and Special Split case when $m \equiv 3 \pmod 9$ and 
$\Ccl ({\mathfrak p}^*) \in \CH_{k^*}^3$ (whence $\rho^{}_{S^*} = 0$). 
We obtain:
\begin{equation}\label{Wclass}
\left\{\begin{aligned}
\rk(W_{\chi^*}) &= \rk \big (E_{k^*}^{S^*}/(E_{k^*}^{S^*})^3 \big) +
\rk(\CH_{k^*}/\CH_{k^*}^3 \cdot \Ccl (S^*)) \\
& = \rk \big (E_{k^*}^{S^*}/(E_{k^*}^{S^*})^3 \big)+\rk(\CH_{k^*})-\rho^{}_{S^*}.
\end{aligned} \right.
\end{equation}

The rank of the $\chi^*$-component of the group of fundamental $S^*$-units is the sum 
of the rank of the unit group of $k^*$ (equal to $1$, being of character $\chi^*$) 
and of the $\chi^*$-component of $\langle S^* \rangle$, equal to $1$ if $3$ 
does not split in $k^*$, equal to $2$ if $3$ splits. 

\begin{theorem}\label{TH}
We have, from \eqref{eq1} and \eqref{Wclass}, 
the general formulas for $p=3$:
\begin{equation*}
\left\{\begin{aligned}
(1) \ \ \  \rk(\CT_k)& = \rk(\CH_{k^*}), 
&&\hbox{if $m \not\equiv 3 \pmod 9$}; \\
(2) \ \ \  \rk(\CT_k) &= \rk(\CH_{k^*}), 
&&\hbox{if $m \equiv 3 \pmod 9$ and $\Ccl ({\mathfrak p}^*) \notin \CH_{k^*}^3$};\\
(3) \ \ \ \rk(\CT_k) &= \rk(\CH_{k^*})+1, 
&&\hbox{if $m \equiv 3 \pmod 9$ and $\Ccl ({\mathfrak p}^*) \in \CH_{k^*}^3$}.
\end{aligned} \right.
\end{equation*}
\end{theorem}

These rank formulas have some importance for computations, but 
the structures of $\CH_{k^*}$ and $\CH_k$ are easily given by 
{\sc pari/gp}; moreover, we will need to check the two equalities 
$\rk(\CH_{k^*})+1 = \rk(\CH_k)$ or 
$\rk(\CH_{k^*}) = \rk(\CH_k)$ of the Spiegelungssatz:
\begin{equation}\label{spiegel}
\left\{\begin{aligned}
&\hbox{(1)\ \ \  $\rk(\CH_{k^*})+1 = \rk(\CH_k)$ holds if 
and only if $Y_{k^*}^\prim = Y_{k^*}$}. \\
&\hbox{(2)\ \ \  $\rk(\CH_{k^*}) = \rk(\CH_k)$ holds if and only if
$\rk(Y_{k^*}/Y_{k^*}^\prim) = 1$.}
\end{aligned} \right.
\end{equation}

\medskip 
Let's give some consequences for the Kummer radical of $k_1^\acyc/k$,
using Theorem \ref{TH}:

\smallskip
{\bf a)} Non linear disjunction (i.e., $k_1^\acyc/k$ unramified) and 
$m \equiv 3 \pmod 9$. 

\smallskip
\qquad $\bullet$\ If $\rk(\CT_k) = \rk(\CT_k^\bp)$, Theorem \ref{DF} 
implies that all the cubic extensions of $k$ in $H_k^\pr$
are unramified and all the cubic polynomials $Q \in \Q[x]$ 
tested in the programs for the search of $k_1^\acyc$, give a 
field discriminant equal to that of $\Q(\sqrt{-m})$, hence exactly of
$3$-valuation~$1$. In particular, $\rk(\CH_{k^*})+1 = \rk(\CH_k)$
from \eqref{spiegel}\,(1); whence
$\rk(W_{\chi^*}) = \rk(\CH_k)+ 1- \rho^{}_{S^*}$. 

\smallskip
\qquad $\bullet$\ If $\CT_k = \CT_k^\bp \oplus \CW_k^\bp$, there exists a cubic 
extension of $k$ ramified and contained in $H_k^\pr$; thus 
we have the equality $\rk(\CH_{k^*}) = \rk(\CH_k)$, and then
$\rk(W_{\chi^*}) = \rk(\CH_k)+2-\rho^{}_{S^*}$. 

{\bf b)} Linear disjunction (i.e., $k_1^\acyc/k$ ramified) and $m \equiv 3 \pmod 9$. 
Then $\CT_k^\bp \simeq \CH_k$.

\smallskip
\qquad $\bullet$\ If $\rk(\CT_k) = \rk(\CT_k^\bp)$, then
$\rk(W_{\chi^*}) = \rk(\CH_k)+1$.

\smallskip
\qquad $\bullet$\ If $\CT_k = \CT_k^\bp \oplus \CW_k^\bp$,
then $\rk(W_{\chi^*}) = \rk(\CH_k)+2$. 

\medskip
{\bf c)} Let $W_{\chi^*}^\prim$ be the subgroup of $3$-primary elements 
of $W_{\chi^*}$ (i.e., giving unramified cyclic cubic fields). By definition
$\rk(W_{\chi^*}^\prim) = \rk(\CH_k)$.

\smallskip
Then, in case {\bf a)}, one verifies that $W_{\chi^*}^\prim = 
W_{\chi^*}$ if and only if $\rk(\CT_k) = \rk(\CT_k^\bp)$, equivalent to  
$\rho^{}_{S^*}=1$ (i.e., $\Ccl({\mathfrak p}^*) \notin \CH_{\chi^*}^3$).

\smallskip
In case {\bf b)}, since there is ramification in $k_1^\acyc/k$, hence 
non-$3$-primary elements, one gets $\rk(W_{\chi^*}^\prim) = 
\rk(W_{\chi^*})-1 = \rk(\CH_k)$, implying $\CW_k^\bp$ non 
direct factor of $\CT_k$.

\begin{theorem}\label{wp}
Assume that $m \equiv 3 \pmod 9$. Denote by $\varepsilon^*$ the 
fundamental unit of $k^*$, by $\eta^*$ its fundamental $S^*$-unit,
and, when $\Ccl ({\mathfrak p}^*) \in \CH_{k^*}^3$ 
(Special Split case), by $w_{{\mathfrak p}^*}^*$ 
the fundamental tricky element (see \eqref{tricky}).
Let $Y_{k^*}$ be the group of pseudo-units of $k^*$ 
(Definition \ref{pseudo-units}).

\smallskip
(i) The Normal Split case ($m \equiv 3 \pmod 9$, $\CH_{k^*}\ne 1$, and 
$\Ccl ({\mathfrak p}^*)\notin \CH_{k^*}^3$ equivalent to $\rho^{}_{S^*} = 1$) 
holds if and only if $\rk(\CT_k) = \rk(\CH_{k^*})$. Then:
$$W_{\chi^*} = \langle \varepsilon^*, \alpha_1^* , \ldots, \alpha_{r^*}^* 
\rangle = Y_{k^*} .$$

(ii) The Special Split case ($m \equiv 3 \pmod 9$, $\CH_{k^*}\ne 1$, 
and $\Ccl ({\mathfrak p}^*) \in \CH_{k^*}^3$ equivalent to $\rho^{}_{S^*} = 0$) 
holds if and only if $\rk(\CT_k) = \rk(\CH_{k^*})+1$. Then:
$$W_{\chi^*} = \langle \varepsilon^*, w_{{\mathfrak p}^*}^*,
\alpha_1^* , \ldots, \alpha_{r^*}^*\rangle  = Y_{k^*} \oplus \langle 
w_{{\mathfrak p}^*}^* \rangle.$$
\end{theorem}

\begin{proof}
In the split case, $\rk(E_{k^*}^{S^*}/(E_{k^*}^{S^*})^3) = 2$.

\smallskip
Case (i). The relation $\Ccl ({\mathfrak p}^*) \not\in \CH^{* 3}$ is 
equivalent to $\langle \Ccl ({\mathfrak p}^*) \rangle$ direct factor in 
$\CH_{k^*}/\CH_{k^*}^3$ (Lemma \ref{eta}); whence $\rk(W_{\chi^*}) 
= \rk(Y_{k^*})$. From \eqref{eq1} and \eqref{Wclass}, $W_{\chi^*} = 
\langle \varepsilon^*, \alpha_1^*, \ldots, \alpha_{r^*}^* \rangle$, whence 
$\rk(W_{\chi^*}) = 2+\rk(\CH_{k^*})-1 = \rk(\CH_{k^*})+1$, and 
$\rk(\CT_k) = \rk(\CH_{k^*})$.

\medskip 
Case (ii). Consider the relation ${\mathfrak p}^* = (w_{{\mathfrak p}^*}) \cdot 
{\mathfrak a}^{* 3}$, with ${\mathfrak a}^*$ prime to $3$:

\smallskip
\quad $\bullet$\ If  ${\mathfrak p}^*$ is non $3$-principal, ${\mathfrak a}^{* 3}$
is necessarily non-$3$-principal. Then there exists a minimal power $3^e$,
$e \geq 1$, and $u$ prime to $3$, such that ${\mathfrak p}^{* 3^e \cdot u} =
(\eta) = (w_{{\mathfrak p}^*}^{3^e \cdot u})\cdot {\mathfrak a}^{* 3^{e+1} \cdot u}$, 
hence $\eta^* = w_{{\mathfrak p}^*}^{* 3^e \cdot u} \cdot \alpha^*$, where the class 
of ${\mathfrak a}^*$ in $\CH_k$ leads to $\alpha^* \in Y_{k^*}$; thus, $\eta^*$,
equivalent to $\alpha^*$, does not intervene in $W_{\chi^*} 
= \langle \varepsilon^*, w_{{\mathfrak p}^*}^*, \alpha_1^* , \ldots, 
\alpha_{r^*}^*\rangle$, and $\rk(\CT_k) = \rk(\CH_{k^*})+1$.

\smallskip
\quad $\bullet$\ 
If ${\mathfrak p}^{* u} = (\eta)$ is principal because the $3$-class of 
${\mathfrak a}^*$ is of order $3$, then $\Ccl({\mathfrak a}^{* u})$ 
contributes to $Y_{k^*}$ because of the relation 
${\mathfrak a}^{* 3 \cdot u} = (\alpha^*)$ and we have a relation of the 
form $\eta= w_{{\mathfrak p}^*}^{u}\, \alpha^*$, thus giving
$\eta^*= w_{{\mathfrak p}^*}^{* u}\, \alpha^*$, 
where the two elements $\eta^*$ and $w_{{\mathfrak p}^*}^{* u}$
are in the Kummer radical and non-trivial modulo $k^{* 3}$; 
so one can eliminate $\eta^*$, whence $W_{\chi^*} 
= \langle \varepsilon^*, w_{{\mathfrak p}^*}^*, \alpha_1^*, \ldots, 
\alpha_{r^*}^*\rangle$ and $\rk(\CT_k) = \rk(\CH_{k^*})+1$.

\smallskip
\quad $\bullet$\ 
If ${\mathfrak p}^{* u} = (\eta)$ is principal, because ${\mathfrak a}^*$ is
$3$-principal, we obtain $\eta = w_{{\mathfrak p}^*}^{u}\, \alpha^{* 3}$;
in particular, $\eta^*$ and $w_{{\mathfrak p}^*}^{* u}$ are equivalent 
modulo $k^{* 3}$ and one finds $\rk(\CT_k) = \rk(\CH_{k^*})+1$
for a Kummer radical of the same form $W_{\chi^*} = \langle \varepsilon^*, 
w_{{\mathfrak p}^*}^*, \alpha_1^* , \ldots, \alpha_{r^*}^* \rangle$.

\smallskip
One verifies, from the relation ${\mathfrak p}^* = (w_{{\mathfrak p}^*}) \cdot 
{\mathfrak a}^{* 3}$, that $w_{{\mathfrak p}^*}^* \not\in Y_{k^*}\!\cdot\! k^{* 3}$.
\end{proof} 

\begin{remark}\label{inutile}
(i) In the Normal Split case, $W_{\chi^*} = Y_{k^*}$ and we may suppose 
that the fundamental $S^*$-unit $\eta^*$ does not intervene since $(\eta^*) 
= {\mathfrak p}^{* 3^e \cdot u}$ ($e \geq 1$ because $\Ccl({\mathfrak p}^*) \ne 1$,
being not a cube), and it suffices to represent $\Ccl({\mathfrak p}^*)$
by a prime-to-$3$ ideal (cf. Definition \ref{pseudo-units}). 
In many cases {\sc pari/gp} gives a generator $h_i^* = \Ccl({\mathfrak a}_i^*)$ 
with ${\mathfrak p}^* \mid {\mathfrak a}_i^*$; one verifies that this does not 
matter for the computation of $Q^\acyc$. For instance, let's give an excerpt 
of the data for $m=9498$, whence $k^* = \Q(\sqrt{3166})$:

\ft\begin{verbatim}
m=9498 Disc=37992 H_k=[24,2] T_k=[3] H_kstar=[3] k_1^ac/k is Unramified
Lw=List([
Mod(2773*x+156029,x^2-3166),
Mod(10913716410016992*x+614084477747659775,x^2-3166),
Mod(3405712514532801800843*x+191629973903847165204931,x^2-3166),
Mod(1062779835565406721297006110*x+59799666379906457780333544073,x^2-3166)])
SOLUTION:J=1 w=Mod(2773*x+156029,x^2-3166) 
Q^ac=x^3-x^2-86*x+390    H_kacyc=[8,2]    Val=2 a=1
J=2 Q=x^3+21*x-222 does not define k_1^ac
J=3 Q=x^3-210*x-1352 does not define k_1^ac
J=4 Q=x^3-54*x-272 does not define k_1^ac
\end{verbatim}\ns

The solution giving $k_1^\acyc$ is $w=\eta^*$ of the {\sc pari/gp} 
form ${\sf Mod(2773*x+156029,x^2-3166)}$, of norm $27$, so that 
${\mathfrak p}^{* 3} =: (\eta^*)$ and $Q=x^3-9x+312058$ (note that 
${\sf polredbest(Q)}$ yields ${\sf Q^{ac}=x^3-x^2-86*x+390}$); 
nevertheless, the field-discriminant is $2^3 \cdot 3^1 \cdot1583^1$ 
implying the non-ramification. This confirms that the Kummer extension
$M(\sqrt[3]{\eta^*})/M$ may be unramified, even in the case $(\eta^*) = 
{\mathfrak p}^{* 3}$, as explained in \cite[Theorem I.6.3\,(ii), Proposition II.1.6.3]
{Gra2005} on the computation of the conductor of such an extension.

\smallskip
(ii) In the Special Split case, $w_{{\mathfrak p}^*}^*$ 
intervenes non-trivially in $W_{\chi^*}$ and the fundamental
$S^*$-unit $\eta^*$ never intervene. Special Split cases are:

\smallskip\noindent
$m \in$ \ft
\{$3387$, $4962$, $5862$, $7257$, $8139$, $8913$, $11217$, $13413$, 
$13953$, $14862$, $17814$, $18714$, $19803$, $20703$, \,\ldots$\}$. \ns
\end{remark}

\section{Properties of the fields \texorpdfstring{$k, k_1^\acyc, k^*$}{Lg}
in the Normal Split case}

\subsection{Non-ramification of \texorpdfstring{$k_1^\acyc/k$}{Lg}}

The Normal Split case leads to the following 
property of the first layer $k_1^\acyc/k$:

\begin{theorem}\label{disjunction}
Let $k := \Q(\sqrt{-m})$, $m \equiv 3 \pmod 9$, and let 
${\mathfrak p}^*$ be a prime ideal dividing $3$ in 
$k^*:= \Q(\sqrt{m/3})$. In the Normal Split case 
where, by definition, $\CH_{k^*} \ne 1$, $\Ccl ({\mathfrak p}^*) 
\not\in \CH_{k^*}^3$ (equivalent to $\rk(\CT_k) = \rk(\CH_{k^*})$), 
the first layer $k_1^\acyc/k$ is unramified.
\end{theorem}

\begin{proof}
Recall that the $3$-principality of ${\mathfrak p}^*$ would mean
$\Ccl ({\mathfrak p}^*) \in \CH_{k^*}^3$ (Special Split case for which 
$\rk(\CT_k) = \rk(\CH_{k^*})+1$); so 
$\langle \Ccl ({\mathfrak p}^*)\rangle \ne 1$ is then
a direct factor in $\CH_{k^*}/\CH_{k^*}^3$.
From Theorem \ref{wp}\,(i) and Remark \ref{inutile}, we can restrict 
ourselves to the Kummer radical:
$$W_{\chi^*} = \langle \varepsilon^*, \alpha_1^*, \ldots, \alpha_{r^*}^* \rangle 
= Y_{k^*},\ \, {\rm with} \ r^* \geq 1, $$

\noindent
where $Y_{k^*}$ is the group of pseudo-units of $k^*$ (Definition 
\ref{pseudo-units}).
From relation \ref{reflection} of Theorem \ref{primarity} and the 
definition of ``Normal Split case'', we have:
\begin{equation} \label{NSC}
\left \{ \begin{aligned}
 \rk(\CH_k)-\rk(\CT_k) &= 
 \rk(\CH_k)-\rk(\CH_{k^*}) \\ 
& =  1- \rk(Y_{k^*}/Y_{k^*}^\prim) 
 = \rk(Y_k/Y_k^\prim) \in \{0,1\}.
\end{aligned}\right.
\end{equation}

Assume the contrary of the claim, that's to say, the disjunction of 
$H_k^\nr$ with $k^\acyc$ (whence $k^\acyc/k$ totally ramified and 
$\CH'_k = \CH_k$); from Schema \eqref{schemalog}, we would have:
\begin{equation}\label{NDF0}
\rk(\CH'_k) = \rk(\CH_k) = \rk(\CT_k^\bp)\ \ \& \ \ 
\rk(\CH_k)-\rk(\CT_k) \in \{-1, 0\}, 
\end{equation}  
depending on $\CW_k^\bp$ direct factor or not of $\CT_k$; 
so, from \eqref{NSC} and \eqref{NDF0}, we would have
$\rk(\CH_k)-\rk(\CT_k) = 0$, hence:
\begin{equation}\label{NDF}
\left\{\begin{aligned}
\hbox{$(1)$ \hspace{0.5cm}} & \rk(\CH_k) = \rk(\CH_{k^*}) = \rk(\CT_k), \\
\hbox{$(2)$ \hspace{0.5cm}} & \rk(Y_{k^*}/Y_{k^*}^\prim) 
= 1\ \ \& \  \ \rk(Y_k/Y_k^\prim) = 0, 
\end{aligned} \right.
\end{equation}
where the group $Y_k$ is of the form
$\langle \alpha_1 , \ldots, \alpha_{r} \rangle$, $r := \rk(\CH_k)$ 
since $(E_k^{S_k})^- = 1$ (no $S_k$-units 
of character $\chi$ in $k$ since $3$ does not split); whence, the Kummer 
radical corresponding to the $\chi^*$-component of $\Gal(H_{k^*, 1}^\pr/k^*)$, 
is $W_{\chi} = Y_k = Y_k^\prim$.  All the elements of $W_{\chi}$ 
are $3$-primary giving that the $\chi^*$-components of $\Gal(H_{k^*, 1}^\pr/k^*)$ 
and $\Gal(H_{k^*, 1}^\nr/k^*)$ coincide (all cubic $3$-ramified extensions 
of $k^*$ of character $\chi^*$ would be contained in $H_{k^*}^\nr$).

\smallskip
From \eqref{NDF0} and the first relation of \eqref{NDF} we would deduce 
$\rk(\CT_k) = \rk(\CH_k) = \rk(\CT_k^\bp)$, whence that $\CW_k^\bp$ is not 
direct factor of $\CT_k$ (i.e., $\CW_k^\bp \subseteq \CT_k^3$).

\smallskip
We use \cite[Section 2.2]{GJN2016} for the 
interpretation of $\CW_k^\bp$ as subgroup of infinitesimal classes in 
$\CT_k$. More precisely, put ${\mathcal K}^1 := \{x \in k^\times, \
\hbox{$x \equiv 1 \pmod {{\mathfrak p}}$}\} \otimes \Z_3$, with
${\mathfrak p}\mid 3$ in $k$, and let $\Ccl_\infty$ be the map
$I_k \otimes \Z_3 \too I_k \otimes \Z_3/{\mathcal P}_{k, \infty}$, where
${\mathcal P}_{k, \infty} := \hbox{$\bigcap_{n>0}$} (P_{k, (3^n)} \otimes \Z_3)$,  
so that $\Gal(H_k^\pr/k) \simeq I_k \otimes \Z_3/{\mathcal P}_{k, \infty}$, 
$\Gal(H_k^\pr/\wt k \cap H_k^\nr) \simeq 
\{(x), \ x \in {\mathcal K}^1\}/{\mathcal P}_{k, \infty}$, and:
\begin{equation}\label{zeta}
\left\{\begin{aligned}
\CW_k^\bp & \simeq \big \{(x), \ x \in {\mathcal K}^1,\  i(x) \in 
\tor_{\Z_3}(\CU_k^1) \big \} \big/ {\mathcal P}_{k, \infty} \\
& = \big \{\Ccl_\infty(x),  \  x \in {\mathcal K}^1,\ 
i(x) \in \tor_{\Z_3}(\CU_k^1) \big\}, 
\end{aligned}\right.
\end{equation}
where $i$ is the canonical embedding ${\mathcal K}^1 \to \CU_k^1$.
There would exist ideals ${\mathfrak a}=(x)$, $x \in {\mathcal K}^1$, and
${\mathfrak b} \in I_k \otimes \Z_p$ such that $\Ccl_\infty ({\mathfrak a})$ 
is a generator (of order $3$) of $\CW_k^\bp$, and
$\Ccl_\infty ({\mathfrak a}) = \Ccl_\infty ({\mathfrak b}^3)$,
so that $\Ccl_\infty ({\mathfrak b})$is of order $9$ in $\CT_k$;
then, ${\mathfrak a} = (x)$, with $i(x) = \zeta \in \tor_{\Z_3}(\CU_k^1)$; 
so $\zeta \ne 1$ since $\Ccl_\infty ({\mathfrak a}) \ne 1$. 
The projection (or arithmetic norm) of 
$\Ccl_\infty ({\mathfrak b}) \in \CT_k$  onto $\Ccl ({\mathfrak b}) 
\in \Gal(H_k^\nr/k) \simeq \CH_k$, yields $\Ccl ({\mathfrak b}^3) = 1$ 
since $\CW_k^\bp$ is the kernel of this projection, whence ${\mathfrak b}^3 
= (\beta)$ with $\beta \in Y_k = Y_k^\prim$ (second relation of \eqref{NDF}).
Finally one would get $i(x) = i(\beta) = \zeta$, which is absurd since 
$\zeta$ is not $3$-primary; thus $k^\acyc$ is not linearly disjoint from 
$H_k^\nr$, whence $k_1^\acyc/k$ unramified, and this proves the theorem.
\end{proof}

If $\CT_k^\bp = 1$, $H_k^\nr$ is contained in $k^\acyc$, the group 
$\CT_k$ coincides with $\CW_k^\bp$ which is direct factor
of $\Gal(k^\acyc/k)$ and leads to a very specific context since the 
cyclicity of $\CH_k$ implies $\hbox{\sc Val} := {\bf v}_3^{}(\hbox{\ft{\sc exp}}(k))
+{\bf v}_3^{}(\order \CT_k)- {\bf v}_3^{}({\rm h}(k))+2-{\bf v}_3^{}(m) = 2$ 
(Theorem \ref{Val}). See Section \ref{Hcyclic} for the consequences 
of this particularity.

\subsection{The torsion group \texorpdfstring{$\CT_{k^*}$}{Lg} and
the normalized regulator \texorpdfstring{$\CR_{k^*}$}{Lg}}\label{regulator}
The field $k^*$ being real, abelian $p$-ramification is simpler at the level 
of $\Z_p$-extensions but becomes more complicated due to the unit 
of the base field.

\subsubsection{Generalities about \texorpdfstring{$\CT_{k^*}$}{Lg} and
 \texorpdfstring{$\CR_{k^*}$}{Lg}}

Consider for $k^*$ the analog of the previous schemas of fields in 
$p$-ramification theory for $p=3$; put $k^* =: \Q(\sqrt{m^*})$
($m^*=3m$ or $m/3$). The 
tricky case for $k^*$ is the case $m^* \equiv -3 \pmod 9$  
(splitting of $3$ in $k$) since we have existence of
$\CW_{k^*}^\bp \simeq \Z/3\Z$; this only concerns some Non Split 
and Trivial cases (Programs I and IV).

\smallskip
In our context the $\chi^*$-component of $\Gal(H_{k^*}^\pr/k^*)$ 
does not involve any $\Z_3$-extension and the schema describing 
the $\chi^*$-component $\CT_{k^*}$, of $\Gal(H_{k^*}^\pr/k^*)$, becomes:

\unitlength=1.2cm
\begin{equation}\label{reel}
\begin{aligned}
\vbox{\hbox{\hspace{-4.8cm} 
\begin{picture}(9.7,1.2)
\put(6.6,0.20){\line(1,0){1.9}}
\put(3.6,0.20){\line(1,0){2.2}}
\put(5.9,0.1){$H_{k^*}^\nr$}
\put(8.7,0.1){${H_{k^*}^\bp}$} 
\put(9.4,0.20){\line(1,0){1.0}}
\put(10.45,0.1){${H_{k^*}^\pr}$}
\put(4.45,-0.1){\ft$\CH_{k^*}$}
\put(7.45,-0.1){\ft$\CR_{k^*}$}
\put(9.64,-0.1){\ft$\CW_{k^*}^\bp$}
\put(3.2,0.1){${k^*}$}
\put(4.15,0.3){\tiny${\rm unramfied}$}
\put(7.1,0.3){\tiny${\rm ramfied}$}
\put(9.45,0.3){\tiny${\rm ramfied}$}
\put(6.0,0.85){\ft$\CT_{k^*}^\bp$}
\bezier{500}(3.4,0.4)(6.1,1.1)(8.8,0.4)
\put(6.4,1.3){\ft$\CT_{k^*}$}
\bezier{500}(3.4,0.5)(6.3,1.95)(10.5,0.5)
\end{picture} }} 
\end{aligned}
\end{equation}
\unitlength=1.0cm

\smallskip\noindent
where the cyclic group $\CR_{k^*}$ 
is the normalized $3$-adic regulator. We have total ramification of 
$\wt {k^*}/k^* = k^* \Q^\cyc/k^*$; so $\wt {k^*}\! \cdot\! H_{k^*}^\nr/\wt k^*$ 
is unramified and $H_{k^*}^\pr/\wt {k^*}\!\cdot\! H_{k^*}^\nr$ ramified.

\smallskip
In the Non Split case, since by definition $\CH_{k^*} \ne 1$ and
$\CW_{k^*}^\bp \ne 1$, there is, a priori, no limitation for $\rk(\CT_{k^*})$ 
and the structure of the above schema may be more complicated;
we do not go further since this kind of information does not
intervene for the study of capitulation of $\CH_k$ in $k^\acyc$.
For instance, we have obtained the following example (Program I):

\smallskip
\ft\begin{verbatim}
m=113213 Disc=452852 kronecker(-m,3)=1 h_k=198 H_k=[66,3], T_k=[3,3] h_kstar=18 
H_kstar=[6,3] #H_k=9 #T_k^bp=9 #W_k^bp=1  k_1^ac/k is Ramified
SOLUTION:J=2 w=Mod(-26978654144*x+15722769856325,x^2-339639) 
Q^ac=x^3-5142*x-141940    H_kacyc=[198,6,6,3,3,3]    Val=3    a=1
W_k^bp=1 is not direct factor of T_k
Components of Cl(P1/P'1) and Cl(P1*P'1): [114,2,4,0,0,0] [0,2,0,0,0,1]
Algebraic norm of H_kacyc=[198,6,6,3,3,3] H_k=[66,3]
Norm of the component 1 of H_kacyc: [0,1,2,0,0,0]
Norm of the component 2 of H_kacyc: [0,0,0,0,0,0]
Norm of the component 3 of H_kacyc: [0,0,0,0,0,0]
Norm of the component 4 of H_kacyc: [0,0,0,0,0,0]
Norm of the component 5 of H_kacyc: [3,0,0,0,0,0]
Norm of the component 6 of H_kacyc: [6,0,0,0,0,0]
NO CAPITULATION IN H_k
H_kstar=[6,3]  T_kstar=[9,3,3]  #R_kstar=9
R_kstar is direct factor of T_kstar
\end{verbatim}\ns
 
\begin{lemma}\label{casramif}
For $k^*=\Q(\sqrt {m^*})$, we have the following formulas, where $\sim$
denotes equality up to a $p$-adic unit:

\smallskip
If $3 \nmid m^*$, then $\order\CR_{k^* } = \frac{1}{3} \cdot \log_3(\varepsilon^*)$.

\smallskip
If $3 \mid m^*$, $\order \CR_{k^*} \sim
\frac{1}{\sqrt {m^*}} \cdot \log_3(\varepsilon^*)$ (resp.
$\order \CR_{k^*} \sim \frac{1}{3\,\sqrt {m^*}} \cdot \log_3(\varepsilon^*)$)
if ${m^*} \not \equiv -3 \pmod 9$ (resp. ${m^*} \equiv -3 \pmod 9$). Note that
$\order \CR_{k^*} \, \order \CW_{k^*}^\bp \sim \frac{1}{\sqrt {m^*}} \, \log_3(\varepsilon^*)$.
\end{lemma}

\begin{proof}
Refer to the expression of $\order \CR_{k^*}$ in \cite[Proposition 5.1]{Gra2018}.
\end{proof}

Let's give few examples obtained up to $327500$ (we do not 
write all the outputs that may be obtained with Program II 
\eqref{program2}):

\ft\begin{verbatim}
    m=5142 Disc=20568 H_k=[6,6] T_k=[9] rk=2 rt=1 
W_k^bp is not direct factor of T_k
Q^ac=x^3-2*x^2-122*x+612    H_kacyc=[18,6]    Val=2    a=1
H_kstar=[12] T_kstar=[3,3] #R_kstar=3
R_kstar is direct factor of T_kstar
    m=175809 Disc=703236 H_k=[54,6] T_k=[81] rk=2 rt=1 
W_k^bp is not direct factor of T_k
Q^ac=x^3-x^2+164*x-14    H_kacyc=[162,6]    Val=4    a=1
H_kstar=[9] T_kstar=[9,9] #R_kstar=9
R_kstar is direct factor of T_kstar
    m=325065 Disc=1300260 H_k=[60,2,2,2] T_k=[3] T^bp=[]
W_k^bp is direct factor of T_k
Q^ac=x^3+198*x-764    H_kacyc=[20,2,2,2]    Val=2    a=1
H_kstar=[6,2] T_kstar=[81] #R_kstar=27
R_kstar is not direct factor of T_kstar
    m=325074 Disc=1300296 H_k=[144,2,2] T_k=[3] T^bp=[]
W_k^bp is direct factor of T_k
Q^ac=x^3-x^2+305*x-893    H_kacyc=[48,2,2,2,2]    Val=2    a=1
H_kstar=[12] T_kstar=[3] #R_kstar=1
R_kstar is not direct factor of T_kstar
\end{verbatim}\ns

\smallskip
As one can see, the two cases $\CR_{k^*} \ne 1$ direct factor or not in 
$\CT_{k^*}/\CT_{k^*}^3$ exist; for instance, in the following example, 
$\CT_{k^*} \simeq \Z/81\Z \times \Z/3\Z$, while $\CR_{k^*} \simeq \Z/27\Z$, 
the only possibility is $\CR_{k^*} = \langle h_{81}^3 \cdot h_3 \rangle$ of 
order $27$, where $h_{81},h_3$ are the generators of orders $81$ and $3$ 
of $\CT_{k^*}$, respectively; the quotient is indeed cyclic of order $9$:

\ft\begin{verbatim}
m=19677 Disc=78708 H_k=[6,6,2] T_k=[9] rk=2 rt=1 
W_k^bp is not direct factor of T_k
Q^ac=x^3-x^2-168*x+924    H_kacyc=[18,6,2,2,2]    Val=2    a=1
H_kstar=[18] T_kstar=[81,3] #R_kstar=27
R_kstar is direct factor of T_kstar
\end{verbatim}\ns

\subsubsection{A ``duality'' property of \texorpdfstring{$\CR_{k^*}$}{Lg}
and \texorpdfstring{$\CW_k^\bp$}{Lg} in the Normal Split case}

The following result seems also specific of the Normal Split case
for which we have proved that $k_1^\acyc/k$ is always unramified
(Theorem \ref{disjunction}):
\begin{theorem}\label{WvsR}
In the Normal Split case ($m \equiv 3 \pmod 9$, $\CH_{k^*} \ne 1$
and $\rk(\CT_k) = \rk(\CH_{k^*})$), $\CW_k^\bp$ is direct factor of $\CT_k$ 
if and only if the normalized $3$-adic regulator $\CR_{k^*}$ is not direct 
factor in $\CT_{k^*}/\CT_{k^*}^3$.
\end{theorem}

\begin{proof}
With regard to $k$ as the reflection of the field $k^*$ in the mirror
involution, the Kummer radical
$W_{\chi} \subseteq M^\times/M^{\times 3}$ corresponds to the 
${\chi^*}$-component of $\Gal(H_{k^*,1}^\pr/k^*) = 
\CT_{k^*}/(\CT_{k^*})^3$; in fact, $W_{\chi} = Y_k$ (pseudo-units  
of $k$) since $E_k=1$ and since there is no non-trivial 
$S_k$-unit of character $\chi$ ($3$ is ramified in $k$), nor any 
$w_{\mathfrak p}$ such that $(w_{\mathfrak p}) = 
{\mathfrak a}_3 \!\cdot\! {\mathfrak a}^3$ in $k$ (which gives 
formula (3) below). 

\smallskip
Whence, referring to Schema \eqref{schemalog} and to the formula 
\eqref{NSC} which is nothing else than the fundamental relation 
\ref{reflection} when $\rk(\CT_k) = \rk(\CH_{k^*})$:
\begin{equation*}
\left \{\begin{aligned}
&\hbox{$(1)$ \hspace{0.8cm} $\rk(\CT_k) = \rk(\CH_{k^*})$, } \\
&\hbox{$(2)$ \hspace{0.8cm} $\rk(\CH_k)-\rk(\CT_k) = 1-\rk(Y_{k^*}/Y_{k^*}^\prim)
 = \rk(Y_k/Y_k^\prim) \in \{0, 1\}$,}\\
& \hbox{$(3)$ \hspace{0.8cm} $\rk(\CT_{k^*}) = \rk(\CH_k) = \rk(Y_k)$.}
\end{aligned}\right .
\end{equation*}

Using $(1)$ and the Spiegelungssatz yields
$\rk(\CT_k) = \rk(\CH_k)-1$ or $\rk(\CT_k) = \rk(\CH_k)$:

\medskip
(i) Case $\rk(\CT_k) = \rk(\CH_k)-1$. 
This implies, from $(2)$: 
$$\rk(Y_k/Y_k^\prim) = 1 \ \,\& \ \,\rk(Y_{k^*}/Y_{k^*}^\prim) = 0. $$

Then $\CW_k^\bp$ is not a direct factor of $\CT_k$, otherwise this would imply
$\rk(\CT_k) = \rk(\CT_k^\bp) +1 = \rk(\CH'_k)+1 = \rk(\CH_k)-1$, 
whence $\rk(\CH_k)-\rk(\CH'_k) = 2$ (absurd). 

\smallskip
Thus $\CR_{k^*}$ is direct factor, otherwise 
$\rk(\CT_{k^*}) = \rk(\CH_{k^*})$ which would imply $\rk(\CH_{k^*}) = 
\rk(\CH_k)$ from $(3)$, then from (1), $\rk(\CH_{k^*}) = \rk(\CH_k)-1$ 
(absurd).

\medskip
(ii) Case $\rk(\CT_k) = \rk(\CH_k)$. We have, from $(3)$,
$\rk(\CT_{k^*})-\rk(\CH_{k^*}) 
= \rk(\CH_k)-\rk(\CH_{k^*}) = \rk(\CT_k)- \rk(\CH_{k^*}) = 0$ 
from (1), proving that $\CR_{k^*}$ is not a direct factor.

\smallskip 
It remains to prove that $\CW_k^\bp$ is a direct factor; one may assume
$\CT_k^\bp \ne 1$, otherwise the result is obvious.
Suppose that $\CW_k^\bp$ is not a direct factor, whence $\rk(\CT_k) =
\rk(\CT_k^\bp)$. Since this means $\CW_k^\bp \subseteq \CT_k^3$,
knowing that $\CT_k^\bp \ne 1$ and that $k_1^\acyc/k$ is unramified, 
we have $H_{k,1}^\pr \subseteq H_k^\nr$; from the assumption
$\rk(\CT_k) = \rk(\CH_k)$ and formula (2), we have
$\rk(Y_{k^*}/Y_{k^*}^\prim)=1$, whence the existence of a ramified 
cyclic cubic field ($F_0$ in Schema \eqref{schemalog}) of $k$ (contradiction).
So, $\CW_k^\bp$ is a direct factor. 
\end{proof} 

The case $\CR_{k^*}$ direct factor is rare and we 
obtain the list (Program II) \ft $5142$, $6690$, $6789$, $7977$, 
$8751$, $10173$, $11001$, $12837$, $19677$, $21018$, $22395$, 
$23862$,\,$\ldots$\ns; for $\CR_{k^*}$ non direct factor, 
Program II gives the list \ft$237$, $426$, $669$, $687$, $705$,
$1038$, $1281$, $1407$, $1722$, $2091$, $2190$, 
$2199$, $2622$,$2685$, $2694$, $2802$,\,$\ldots$ \ns

\section{Algorithms and {\sc pari/gp} Programs}

We intend to give maximal information about the arithmetic invariants
of $k$ and $k^*$; if this is the main goal for the class groups, we have seen 
that $\CT_k$ and $\CT_k^\bp$ (not given by specific {\sc pari/gp}
instructions), play a central role in the determination of $k_1^\acyc$
and its arithmetic, then characterize easily the Normal Split case
($\rk(\CT_k) = \rk(\CH_{k^*})$) from the Special Split case 
($\rk(\CT_k) = \rk(\CH_{k^*}) + 1$).

\subsection{Computation of the structure of \texorpdfstring{$\CT_k$}{Lg}} \label{calculT}
Finding the group structure of $\CT_k$ needs a modulus $(p^\nu)$, 
larger than $(p^2)$ giving the $p$-rank for $p \geq 3$; moreover, 
one must interpret the outputs given by {\sc pari/gp}. Indeed, the list 
${\sf Kpnu.cyc}$ gives the invariants of the Galois group of the 
ray class field $H_k(p^\nu)$, of modulus $(p^\nu)$, in decreasing order 
and begins with the two largest ones, related to 
$\Gal(H_k(p^\nu) \cap \wt k/k)$, 
which tend to infinity as $\nu \to \infty$.  
For $\nu$ large enough one gets the following structure of the $p$-ray 
class group of modulus~$(p^\nu)$: 
$${\CH_k(p^\nu) = [p^{z_\acyc}, p^{z_\cyc}; p^{t_1}, \ldots, 
p^{t_{\rm rt(k)}}] = [p^{z_\acyc}, p^{z_\cyc}] \oplus [\CT_k]}, \  
{z_\acyc \geq z_\cyc \geq t_1 \geq \cdots \geq t_{\rm rt(k)}}, $$ 
where ${\rm rt(k)}$ is the $p$-rank of $\CT_k$ and ${p^{z_\acyc}, 
p^{z_\cyc}}$ the components tending to infinity. 

\smallskip
For instance, for $m=1400529$, $p = 3$, where $\CH_k \simeq \Z/3\Z \times
\Z/3\Z$, taking successively $\nu = 12$, 
$\nu={\bf v}_3^{} (\hbox{\ft{\sc exp}}(k))=1$, 
$\nu={\bf v}_3^{} (\hbox{\ft{\sc exp}}(k))+1=2$, 
$\nu={\bf v}_3^{} (\hbox{\ft{\sc exp}}(k))+2=3$, 
$\nu={\bf v}_3^{} (\hbox{\ft{\sc exp}}(k))+3=4$, 
one obtains the following {\sc pari/gp} structures of $\CH_k(3^\nu)$: 

\medskip
\centerline{\ft${\sf [3^{12}, 3^{11}; 9]}$, \ \  ${\sf [9,3]}$,\ \ ${\sf [9,9,3]}$, \ \ 
${\sf [27,9; 9]}$,\ \ ${\sf [81,27; 9]}$, \ldots\ns, 
whence ${\CT_k \simeq \Z/9\Z}$.}

\medskip
It remains to check that for the suitable value of $\nu$,
${p^{z_\acyc} \geq p^{z_\cyc}}$ are the largest 
components, thus at the beginning of the {\sc pari/gp} list, and that
all others are strictly smaller and stable for any $\nu' > \nu$.

\begin{theorem}\label{T}
Let $k = \Q(\sqrt {-m})$ and let $p \geq 3$. One obtains the 
{\sc pari/gp}-structure of $\CT_k$, under the form ${[p^{t_1}, 
\ldots, p^{t_{\rm rt(k)}}]}$, from the computation of that of $\CH_k(p^\nu)$,
which is given by ${[p^{z_\acyc}, p^{z_\cyc}; p^{t_1}, \ldots, p^{t_{\rm rt(k)}}]}$,
taking the minimal value $\nu = {\bf v}_p^{} (\hbox{\ft{\sc exp}}(k))+2$, 
where $\hbox{\ft{\sc exp}}(k)$ is the exponent of $\CH_k$. Moreover, 
${z_\cyc = \nu-1}$ and ${z_\acyc \geq z_\cyc \geq 
t_1 \geq \cdots \geq t_{\rm rt(k)}}$.
\end{theorem}

\begin{proof}
Consider the subgroup $\CH'_k$ of $\CH_k$ (see Schema \eqref{schemalog}),
isomorphic to $\CT_k^\bp$; for $p  \ne 3$, $\hbox{\ft{\sc exp}}(k) =  p^{\nu-2}$ 
annihilates $\CT_k$. For $p=3$ and $\CT_k^\bp$ of index $3$ in $\CT_k$, 
it follows that $3 \, \hbox{\ft{\sc exp}}(k) = 3^{\nu-1}$ annihilates $\CT_k$.

\smallskip
We have the well-known formula in $k/\Q$ (see, e.g., \cite[Corollary I.4.5.6]
{Gra2005} for $E_k=1$), where $\varphi_k({\mathfrak p}^n) \sim \varphi_\Q^{}
(\BN{\mathfrak p}^n) \sim (\BN{\mathfrak p})^{n-1}$, for 
${\mathfrak p} \mid p$, and $\BN{\mathfrak p} = p^f$, where $f$ is 
the residue degree and $\varphi$ the Euler function; since in $k$, 
$(p) = {\mathfrak p}, {\mathfrak p}{\mathfrak p}'$ or ${\mathfrak p}^2$:
$$\order \CH_k(p^\nu) \sim \order \CH_k \cdot \varphi_k(p^\nu) \sim 
\order \CH_k \cdot p^{2(\nu-1)} (\hbox{resp.}\  \order \CH_k \cdot p^{2 \nu-1}
\ \hbox{in the ramified case}).$$

The component $\CH_k(p^\nu)^+$ may be deleted in the 
formula since it concerns the cyclotomic context of $\Q^\cyc$ for 
which $\CT_\Q = 1$ and $\order \CH_\Q(p^\nu) \sim p^{\nu-1}$.
So, the minus part fulfills the relation
$\order \CH_k(p^\nu)^- \sim \order \CH_k \cdot p^{\nu-1}\ 
(\hbox{resp.}\ \order \CH_k \cdot p^\nu)$.
This formula comes from the minus parts of the
exact sequence:
$$1 \to P_k \otimes \Z_p/P_{k,p^\nu} \otimes \Z_p \to
\CH_k(p^\nu) \simeq  I_k \otimes \Z_p/P_{k,p^\nu} \otimes \Z_p
\to \CH_k =  I_k \otimes \Z_p/P_k \otimes \Z_p \to 1, $$

\noindent
interpreting $\CH_k(p^\nu)$ by means of the isomorphism (since $E_k = 1$): 
\begin{equation*}
\left \{\begin{aligned}
P_k \otimes \Z_p/P_{k,p^\nu} \otimes \Z_p & \simeq \CU_k^1/\CU_k^\nu
= (1+ p\CO_k) / (1+p^\nu \CO_k)  \ \hbox{(if $p \nmid m$)}; \\
P_k \otimes \Z_p/P_{k,p^\nu} \otimes \Z_p& \simeq \CU_k^1/\CU_k^{2 \nu}
= (1+ \sqrt{-m}\CO_k) / (1+(\sqrt{-m})^{2 \nu} \CO_k) \ \hbox{(if $p \mid m$)},
\end{aligned}\right.
\end{equation*}

\noindent
Taking the minus parts, where $\CO_k^- = \Z_p\sqrt{-m}$
and $\tor_{\Z_p}^{}(\CU_k^1) \simeq \CW_k^\bp$, we obtain:
\begin{equation*}
\left \{\begin{aligned}
&1 \to \Z_p/p^{\nu-1}\Z_p \too \CH_k(p^\nu)^- \too \CH_k \to 1, 
\hbox{ if $p > 3$ {\it or} $p=3$ \& $m \not\equiv 3 \!\!\! \pmod 9$} ;\\
&1 \to \Z/3\Z \oplus (\Z_3/3^{\nu-1} \Z_3) \too \CH_k(3^\nu)^- \too \CH_k \to 1, 
\hbox{ if $p = 3$ \& $m \equiv 3 \!\!\! \pmod 9$}.
\end{aligned}\right.
\end{equation*}

This proves that the group structure of $\CH_k(p^\nu)^-$ 
contains a component of order $p^{\nu-1}$, 
and we recall that $p^{z_\cyc} = p^{\nu-1}$, so that $p^{z_\acyc}
\geq p^{z_\cyc}$ are necessarily the two first in the list.
Then, in the outputs, the minus part reduces to:
$$\hbox{${\CH_k(p^\nu)^- = [p^{z_\acyc};\, p^{t_1}, \ldots, p^{t_{\rm rt(k)}}]}$, 
with $z_\acyc \geq z_\cyc = \nu-1 \geq t_1\geq\, \cdots$}$$
So, it suffices to take $\nu-1 = {\bf v}_p (\hbox{\ft{\sc exp}}(k)) +1$ to be 
certain that ${\CT_k = [p^{t_1}, \ldots, p^{t_{\rm rt(k)}}]}$.
\end{proof}

The degree $p^{(z_\acyc-z_\cyc)}$ is given by $[k^\acyc \cap H_k^\nr : k]$;
for instance one finds for $p = 3$ the case (meaning that the $3$-Hilbert 
class field is contained in $k^\acyc$):

\medskip
\centerline{\hspace{-2.4cm} ${\sf m=335\ \ -m \equiv 1 \pmod 3\ \ 
\CH_k=[9]\ \ \nu = 4\ \  \CH_k(3^\nu) = [243,27]\ \ \CT_k = [\ ]}$,}

\medskip\noindent
and the cases  (meaning $H_k^\nr$ disjoint from $k^\acyc$, 
$\CT_k^\bp \simeq \CH_k$, and $\CW_k^\bp = \CT_k^3$):

\medskip
\centerline{\hspace{-2.9cm} ${\sf\  m=417\ \ \,m \equiv 0 \pmod 3\ \ 
\CH_k=[3]\ \ \,\nu = 3\ \ 
\CH_k(3^\nu) = [9,9,9]\ \ \, \CT_k = [9]}$.}

\medskip
\centerline{${\sf m=10938054\ \, m \equiv 0 \pmod 3\ \, 
\CH_k=[9,3] \ \, \nu = 4\ \, \CH_k(3^\nu) = [27,27,9,3,3] \ \, \CT_k = [9,3,3]}$.}

\begin{remark}
It is not difficult to see that, for $p=3$ and the real quadratic field $k^*$, the 
optimal value of $\nu^{{}_*} $, allowing the computation of the structure of 
$\CT_{k^*}$, is:
$$\nu^{{}_*} = {\bf v}_3^{}(\hbox{\ft{\sc exp}}(k^*))+ {\bf v}_3^{}(\order \CR_{k^*})
+ {\bf v}_3^{} (\CW_{k^*}^\bp)+2, $$
where $\hbox{\ft{\sc exp}}(k^*)$ is the exponent of  $\CH_{k^*}$ and $\CR_{k^*}$ 
the normalized $3$-adic regulator (a cyclic group). Indeed, the cyclotomic 
$\Z_3$-extension of $k^*$ is totally ramified, giving the previous Schema 
\eqref{reel} and the property that the largest component of 
$\Gal(H_{k^*}(3^{\nu^{{}_*}})/k^*)$ is of order $3^{\nu^{{}_*} -1}$; a specific 
intricacy is to take into account the index $(E_{k^*} : E_{k^*, 3^{\nu^{{}_*}}})$ 
in the formula of the order of the ray class group of modulus $3^{\nu^{{}_*}}$
(cf. Section \ref{regulator}).

\smallskip
So, it suffices that $3^{\nu^{{}_*} - 1}$ be larger than the product of the exponent of
$\CH_{k^*}$ by $\order \CR_{k^*}$ (this is used at the end of Programs II and III;
Programs I and IV do not compute this part).
\end{remark}

\subsection{Definition and programming of the algorithm}

In an algorithmic point of view (Algorithm \ref{method}), 
recall that we consider each radical $w \in W_{\chi^*}$ giving, by 
descent over $k$, a $3$-ramified cyclic cubic extension $K_1$ whose
irreducible polynomial is $Q = x^3-3a x-t$ ($a^3 = w w'$, $t=w+w'$); 
then we take prime ideals ${\mathfrak q}$ above prime numbers $q$, 
split in $k$, so that the condition $\Log_3({\mathfrak q})^- \in 
3 \,\Log_3(I_k \otimes \Z_3)^-$ be fulfilled (via $v_3(C_1) \geq \hbox{\sc Val}$); 
then we can test if the Artin symbol $\big (\frac {H_k^\pr/k}{{\mathfrak q}}\big)$ 
of ${\mathfrak q}$ fixes or not $K_1$; if it does not fix $K_1$ (or equivalently
if $Q$ modulo $q$ is irreducible), then $K_1$ can not be 
$k^\acyc$ (because of the condition $\Log_3({\mathfrak q})^- \in 
3 \,\Log_3(I_k \otimes \Z_3)^-$, meaning the splitting of $Q$ modulo $q$), 
and $w$ is eliminated. Since the Artin symbols fixe $k_1^\acyc$ 
for all the primes tested, it is enough to consider sufficiently primes 
${\mathfrak q}$ to get (by elimination) the unique solution $w = w^\acyc$ 
giving $k_1^\acyc$ and its defining polynomial $Q^\acyc$ (in ${\sf Q^{ac}}$).

\smallskip
We then have two particular cases giving simplifications:

\smallskip
(i) When $\CH_{k^*} = 1$ (Trivial case), then $w \in 
\langle \varepsilon^*, \eta^* \rangle$, where $\eta^*$, in the split 
case of $3$ in $k^*$, is the fundamental $S^*$-unit; 
if $3$ does not split in $k$, $w = \varepsilon^*$. 
Moreover, $\CH_k$ is cyclic or trivial.

\smallskip
(ii) If $k_1^\acyc/k$ is unramified and $\CH_k$ cyclic, there is 
no other unramified cubic extension, and its defining polynomial $Q^\acyc$ 
is that giving a discriminant equivalent to that of $k$; this is the case 
denoted by the conditions: ${\sf if(hk3>otbp\ \&\ rk==1}$) in 
the programs (i.e., $\order \CH_k  > \order \CT_k^\bp$
and $\rk(\CH_k)=1$).

\subsection{Specific notations in the programs} \label{notations}
The {\sc pari/gp} writings used in the programs have the following meaning:

\smallskip
${\sf P}$ (resp. ${\sf Pstar}$) = quadratic polynomial defining ${\sf k}  = 
\Q(\sqrt{-m})$ (resp. ${\sf kstar} = \Q(\sqrt {3m})$ or $\Q(\sqrt {m/3})$).

\smallskip
${\sf hk=k.no} = \order \BH_k$ (resp. ${\sf hkstar=kstar.no} 
= \order \BH_{k^*}$) = class number of $k$ (resp. of $k^*$).

\smallskip
${\sf rk}$ (resp. ${\sf rkstar}$) = $3$-rank $\rk(\CH_k)$
of the $3$-class group $\CH_k$ (resp. $3$-rank 
$\rk(\CH_{k^*})$ of $\CH_{k^*}$).

\smallskip
${\sf expk}$ (resp. ${\sf expkstar}$) = 
exponent $\hbox{\ft{\sc exp}}(k)$ of $\CH_k$ 
(resp. $\hbox{\ft{\sc exp}}(k^*)$ of $\CH_{k^*}$).

\smallskip
${\sf exptak}$ (resp. ${\sf exptakstar}$) = 
exponent $\hbox{\ft{\sc expta}}(k)$ of the tame component
of $\BH_k$ (resp. $\BH_{k^*}$).

\smallskip
${\sf nu} = {\bf v}_3^{}(\hbox{\ft{\sc exp}}(k))+2$, ${\sf nu^*} = 
{\bf v}_3^{}(\hbox{\ft{\sc exp}}(k^*)) + {\bf v}_3^{}(\CR_{k^*})+2$.

\smallskip
${\sf Tk}$ = list of the group invariants of $\CT_k :=
\Gal(H_k^\pr/\wt k)$, deduced from the ray class group 
$\CH_k(3^\nu)$, of modulus $3^\nu$, where $\wt k$ is the 
compositum of the $\Z_3$-extensions $k^\cyc$ and $k^\acyc$. 

\smallskip
${\sf rtk}$ = $3$-rank $\rk(\CT_k)$ of $\CT_k$.

\smallskip
${\sf Wk} = \CW_k^\bp$ = subgroup of $\CT_k$ fixing the 
Bertrandias--Payan field $H_k^\bp$.

\smallskip
${\sf ot}$ = $\order \CT_k$ and ${\sf otbp = ot/ow}$ = $\order \CT_k^\bp$,
where ${\sf ow}$ = $\order \CW_k^\bp$.

\smallskip
${\sf Val}$ = bound required to get $\Log_3(\mathfrak q) \in 3\, 
\Log_3(I_k \otimes \Z_3)$ (Theorem \ref{Val}).

\smallskip
${\sf Maxq}$ = bound for the auxiliary primes $q$.

\smallskip
${\sf ram = Ramified}$ (resp. ${\sf Unramified}$) indicates that $k_1^\acyc/k$ is 
ramified (resp. unramified).

\smallskip
${\sf Nrad} = {\rm rtk}+1$ = $3$-rank of $W_{\chi^*} \subset k^*/k^{*\,3}$, 
taking into account the fundamental unit $\varepsilon^*$ of $k^*$ and possibly 
specific radicals $\eta^*$ (fundamental $S^*$-unit) and $w_{{\mathfrak p}^*}^*$ 
defined from \eqref{tricky}.

\smallskip
${\sf nw=(3^{Nrad}-1)/2}$ = number of non-trivial radicals = number
of $3$-ramified cubic cyclic extensions of $k$, distinct from the cyclotomic one.

\smallskip
${\sf Lw}$ = list of the ${\sf nw}$ radicals $w \in W_{\chi^*}$.

\smallskip
${\sf Q = Q^{ac} = x^3-3*a*x-t}$, with ${\sf a^3 = norm(w)}$ and 
${\sf t = trace(w)}$, defines $k_1^\acyc/k$.

\smallskip
${\sf R = compositum(x^2+m,Q^{ac})}$ = polynomial of degree 
$6$ defining $k_1^\acyc/\Q$.

\smallskip
${\sf DiscR}$ = Discriminant of $k_1^\acyc$.

\smallskip
${\sf Rkstar}$ = normalized $3$-adic regulator $\CR_{k^*}$ of $k^*$.

\begin{remarks}
(i) The outputs give the polynomials ${\sf Q^{ac} = polredbest(x^3-3*a*x-t)}$; 
for the computation of $\BH_{k_1^\acyc}$ we have used this reduced form, 
and the reader may return to the original computation;
for instance, for $m=102203$, the output is:

\smallskip
\centerline{\ft${\sf Q=x^3-3*x+
24015965877041268782539898447628077395512816560774268}$\ns}

\noindent
\hspace{2.6cm}\ft${\sf 525708969597351078951941393554
079402148815543809546015031051950}$,\ns

\smallskip\noindent
while ${\sf Q^{ac} = polredbest(Q) = x^3-93*x-1160}$, of field-discriminant 
$3^4 \times 102203$ as expected.

\smallskip
(ii) The writting ${\sf Norm\ of\ the\ component \ i \ of\ H_{kacyc}:\  [a,b,...]}$ gives the
exponents of the generators ${\sf h^{ac}_j}$ of $\BH_{k_1^\acyc}$ for the algebraic 
norm of $h^{\acyc}_i$ (equal $[0,0,...,0]$ when $\BNu_{k_1^\acyc/k}(h^{\acyc}_i) = 1$).

\smallskip
(iii) In some programs, we have included the test of partial capitulation 
that we will study in Section \ref{verifcap} and we compute the torsion
group $\CT_{k^*}$ of $3$-ramification and the normalized $3$-adic regulator 
$\CR_{k^*}$, of $k^*$. We indicate if $\CW_k^\bp$ is direct factor or not
of $\CT_k$ and if $\CR_{k^*}$ is or not direct factor of $\CT_{k^*}^\bp$.
\end{remarks}

\section{The four programs for the determination of 
\texorpdfstring{$k_1^\acyc/k$}{Lg}}

These programs may be copied-and-pasted by the reader,
directly from the pdf file, respecting all line breaks and the fact that 
they are written on several pages. One has the
choice of considering a list of $m$'s or an interval of $m$'s (use suitably 
the $\backslash\!\backslash$ to favor a method). They are related to the 
cases Non Split, Normal Split, Special Split, Trivial, described in Definition 
\ref{fourcases}. They use a not too old version \cite[2.11.2]
{Pari2019} of {\sc pari/gp}.

\smallskip
Up to ${10^6}$ there are ${\sf Nm=607923}$ square-free integers ${m}$.
Then the four cases (Non Split, Normal Split, Special Split, and Trivial)
considered by Programs I, II, III, IV, respectively, have the following distribution:

\medskip
\centerline{\ft${\sf Nprog1/Nm=11.483\, \%,\ \, Nprog2/Nm=01.190\, \%,\ \, 
Nprog3/Nm=00.362\, \%,\ \, Nprog4/Nm=86.965\, \%}. $\ns}

\medskip
Of course, the last number counts the fields such that not only $\CH_{k^*}=1$ 
but also those with $\CT_k = \CT_k^\bp = \CH_k = \CW_k^\bp = 1$, in which 
case $k_1^\acyc$ comes from $w = \varepsilon^*$, non-$3$-primary.
For Program IV and $m \leq 150$, we have approximatively $80 \%$ of 
fields with $m \not \equiv 3 \pmod9$ and $10 \%$ with $m \equiv 3 \pmod9$.

\smallskip
To decide which program to use, one may test $m$ with the following:

\medskip\noindent
{\bf PROGRAM 0}: 

\smallskip
\ft\begin{verbatim}
{m=8139;m9=lift(Mod(m,9));P=x^2+m;k=bnfinit(P);Pstar=x^2-3*m;
if(Mod(m,3)==0,Pstar=x^2-m/3;);kstar=bnfinit(Pstar);hkstar=kstar.no;
rkstar=0;d=matsize(kstar.cyc)[2];for(j=1,d,C=kstar.cyc[j];
if(Mod(C,3)==0,rkstar=rkstar+1));nu=valuation(k.no,3)+2;
Kpnu=bnrinit(k,3^(nu));Hpnu=Kpnu.cyc;rt=0;d=matsize(Hpnu)[2];
for(j=3,d,C=Hpnu[j];if(Mod(C,3)==0,rt=rt+1));kstar3=lift(Mod(hkstar,3));
if(m9!=3 & kstar3==0,print("Prog. I"));
if(m9==3 & kstar3==0 & rt==rkstar,print("Prog. II"));
if(m9==3 & kstar3==0 & rt==rkstar+1,print("Prog. III"));
if(kstar3!=0,print("Prog. IV"))}
Prog. III
\end{verbatim}\ns

\smallskip
For each program, a sufficient precision ${\sf \backslash p \ {***}}$, is
often needed for few fields. The precision may be insufficient, especially 
because of the instructions of the form ${\sf b=bnfisprincipal(k,B)}$
computing the generators $\beta$ such that $(\beta) = 
{\mathfrak q}^{\hbox{\tiny{\sc exp}(k).{\sc hta}(k)}}$ in the class group $\BH_k$; 
if so, the following message appears which invites to increase the precision:

\smallskip
\ft\begin{verbatim}
 *** bnfisprincipal: Warning: precision too low for generators, not given.
 *** at top-level: ...beta=lift(Mod(k.zk[1]*b[2][1]+k.zk[2]*b[2][2],P)^(3^dres-1))
\end{verbatim}\ns

\smallskip
In the same way, the choice of ${\sf Maxq}$ may be insufficient
for huge integers $m$; for instance with the list of $49$ items of 
Program I and ${\sf Maxq=2*10^2}$ the final list is:
$${\sf ListSigma=List([58213,78730])}$$
which becomes empty with ${\sf Maxq=3*10^2}$. Of course,
the running time depends on the order of magnitude of ${\sf Maxq}$;
the value ${\sf Maxq=10^3}$ seems largely sufficient in practice.
 
\smallskip
Finally, in the outputs, we have removed the writing of the lists 
${\sf Lw}$ of radicals $w$, often enormous and with few interest.

\subsection{PROGRAM I (Non Split case)} \label{program1}
This program processes only cases $m \not \equiv 3 \pmod 9$ 
(equivalent to $3$ non-split in $k^*$ or to $\CW_k^\bp=1$), and 
$\CH_{k^*} \ne 1$. Since $3$ does not split in $k^*$, the radical 
$W_{\chi^*}$ is the group of pseudo-units $Y_{k^*} = \big\langle\,
\varepsilon^*, \alpha_1^*, \ldots, \alpha_{r^*}^*\,\big\rangle$
of $3$-rank $\rk(\CH_{k^*})+1$:

\smallskip
\ft\begin{verbatim}
\p 300
allocatemem(800000000)
{print("From Program I (Non Split case)");Maxq=500;
ListSigma=List;Listm=List;Nm=21;
Listm=List([28477,32573,34603,35353,39677,50983,55486,56773,58213,59221,
61379,63079,67054,68626,70977,78730,82834,86551,88415,96027,98281]);
\\List of fields k such that rk = rkstar = 2
\\***choose between List and interval:***
\\for(km=1,Nm,m=Listm[km];
for(m=5,2.5*10^4,if(core(m)!=m,next);
if(Mod(m,9)==3,next); \\***First condition***
Pstar=x^2-3*m;if(Mod(m,3)==0,Pstar=x^2-m/3);
kstar=bnfinit(Pstar);hkstar=kstar.no;
if(Mod(hkstar,3)!=0,next); \\***Second condition***
P=x^2+m;k=bnfinit(P);hk=k.no;Delta=0;
Lstar=List;Cstar=kstar.clgp[2];rkstar=0;
d=matsize(Cstar)[2];for(j=1,d,C= Cstar[j];
if(Mod(C,3)==0,rkstar=rkstar+1;listput(Lstar,j)));
rk=0;expk=1;d=matsize(k.cyc)[2];for(j=1,d,C=k.cyc[j];
v=valuation(C,3);if(v>0,rk=rk+1;expk=max(expk,3^v)));
exptak=1;d=matsize(k.cyc)[2];for(j=1,d,C=k.cyc[j];
v=valuation(C,3);exptakC=C/3^v;exptak=lcm(exptakC,exptak));

\\STRUCTURE OF T_k AND COMPUTATION OF Val:
nu=valuation(expk,3)+2;
Kpnu=bnrinit(k,3^nu);Hpnu=Kpnu.cyc;e=matsize(Hpnu)[2];
Tk=List;ot=1;rt=0;for(j=3,e,c=Hpnu[j];v=valuation(c,3);
if(v>0,listput(Tk,3^v);rt=rt+1;ot=ot*3^v));print();
ow=1;if(Mod(m,9)==3,ow=3);Wk=ow;hk3=3^valuation(hk,3);
otbp=ot/ow;ram=Ramified;if(hk3!=otbp,ram=Unramified);
Val=valuation(expk*ot/(m*hk),3)+2;

\\WRITINGS OF THE MAIN INVARIANTS:
kro=kronecker(-m,3);Disc=m;if(Mod(-m,4)!=1,Disc=4*m);
print();print("m=",m," Disc=",Disc," kronecker(-m,3)=",kro,
" h_k=",hk," H_k=",k.cyc," T_k=",Tk," h_kstar=",hkstar,
" H_kstar=",kstar.cyc," #H_k=",hk3," #T_k^bp=",otbp,
" #W_k^bp=",Wk,"  k_1^ac/k is ",ram);

\\COMPUTATION OF THE nw RADICALS IN Lw:
Lw0=List;Nrad=rkstar+1;w0=kstar.fu[1];listput(Lw0,w0);
for(i=1,rkstar,Idstar=kstar.clgp[3][Lstar[i]];
alpha=bnfisprincipal(kstar,idealpow(kstar,Idstar,Cstar[Lstar[i]])); 
wi=Mod(kstar.zk[1]*alpha[2][1]+kstar.zk[2]*alpha[2][2],Pstar);
listput(Lw0,wi));Lw=List;nw=3^Nrad-1;LD=List;for(i=1,nw,
D=digits(i,3);if(D[1]!=2,listput(LD,D)));nw=nw/2;for(i=1,nw,
d=matsize(LD[i])[2];d0=Nrad+1-d;w=1;for(j=d0,Nrad,
w=w*Lw0[j]^LD[i][j-d0+1]);listput(Lw,w));print("Lw=",Lw);

\\RESEARCH OF THE SOLUTION AMONG THE nw RADICALS:
\\Begin J
for(J=1,nw,w=Lw[J];Nw=norm(w);ispower(Nw,3,&a);
Q=x^3-3*a*x-trace(w);Q=polredbest(Q);

\\UNRAMIFIED CASE with RANK rk=1:
if(hk3>otbp & rk==1,Disc=nfdisc(Q);if(valuation(Disc,3)>1,
print("J=",J," w=",w," Q=",Q," does not define k_1^ac"));
if(valuation(Disc,3)<=1,Delta=Delta+1;
R=polcompositum(P,Q)[1];kacyc=bnfinit(R);hkacyc=kacyc.no;
print("SOLUTION:J=",J," w=",w," Q^ac=",Q);
print("H_kacyc=",kacyc.cyc)));

\\RAMIFIED CASE or RANK rk>1:
if(hk3==otbp || rk>1,Test=0;e0=1;if(Mod(m,3)==0,e0=2);
dres=1;if(kro==-1,dres=2);
forprime(q=5,Maxq,if(kronecker(-m,q)!=1,next);
A=component(component(idealfactor(k,q),1),1);
B=idealpow(k,A,exptak*expk);b=bnfisprincipal(k,B);
pib=Mod(k.zk[1]*b[2][1]+k.zk[2]*b[2][2],P)^(3^dres-1)-1;Log=pib;
t=1;while(t<=e0*(Val+log(t)+1),t=t+1;Log=Log-(-1)^t*pib^t/t);
Log=lift(Log);C1=polcoeff(Log,1);if(valuation(C1,3)<Val,next);
Qq=Q+O(q);if(polisirreducible(Qq)==1,print("J=",J," q=",q,
" Qq irreducible");Test=1;break));if(Test==0,Delta=Delta+1;
R=polcompositum(P,Q)[1];kacyc=bnfinit(R);hkacyc=kacyc.no;
print("SOLUTION:J=",J," w=",w," Q^ac=",Q);
print("H_kacyc=",kacyc.cyc))));
\\End J
if(Delta!=1,listput(ListSigma,m));

\\VALUE OF Val AND the unit index a:
print("W_k^bp=1 is not direct factor of T_k");
print("Val=",Val," a=",3*3^-lift(Mod(valuation(hkacyc/hk,3),2)));

\\COMPUTATION OF THE PARTIAL CAPITULATIONS:
\\(3 non-split in k) -Test of principality of P_1:
if((Mod(-m,3)!=1 & ram==Ramified),
Sideal=component(component(idealfactor(kacyc,3),1),1);
Pr=bnfisprincipal(kacyc,Sideal);
print("Components of Cl(P_1): ",Pr[1]));
\\(3 split in k)-Test of principality of P1/P'1 and P1*P'1:
if((Mod(-m,3)==1 & ram==Ramified),
Sideal1=component(component(idealfactor(kacyc,3),1),1);
Sideal2=component(component(idealfactor(kacyc,3),1),2);
Sminus=idealdiv(kacyc,Sideal1,Sideal2);
Splus=idealmul(kacyc,Sideal1,Sideal2);
Pminus=bnfisprincipal(kacyc,Sminus);
Pplus=bnfisprincipal(kacyc,Splus);
print("Components of Cl(P1/P'1) and Cl(P1*P'1): ", Pminus[1]," ",Pplus[1]));

if(Delta==1,Cacyc=kacyc.clgp;d=matsize(Cacyc[2])[2];
G=nfgaloisconj(kacyc);Id=x;for(k=1,6,Z=G[k];ks=1;while(Z!=Id,
Z=nfgaloisapply(kacyc,G[k],Z);ks=ks+1);if(ks==3,s=G[k];break));
print("Algebraic norm of H_kacyc=",kacyc.cyc," H_k=",k.cyc);
for(i=1,d,X0=Cacyc[3][i];X=1;for(t=1,3,
Xs=nfgaloisapply(kacyc,s,X);X=idealmul(kacyc,X0,Xs));
Y=bnfisprincipal(kacyc,X)[1];Normalg=List;
for(j=1,d,v=valuation(Cacyc[2][j],3);
c=lift(Mod(Y[j],3^v));listput(Normalg,c,j));
print("Norm of the component ",i," of H_kacyc: ",Normalg)));
print(ListSigma))}
\end{verbatim}\ns

\subsection{PROGRAM II (Normal Split case)} \label{program2}
It assumes $m \equiv 3 \pmod 9$, $\CH_{k^*} \ne 1$ and $\rk(\CT_k) 
= \rk(\CH_{k^*})$ (equivalent to $\Ccl ({\mathfrak p}^*) \notin \CH_{k^*}^3$). 
The radical $W_{\chi^*}$ is still the group of pseudo-units $Y_{k^*}$.
We have proved in Theorem \ref{disjunction} that $k_1^\acyc/k$ 
is unramified, which simplifies the  program; nevertheless, we 
test the relation ${\sf hk3!=otbp}$ characterizing the non-ramification.
If $\CH_k$ is cyclic, the unique unramified cubic cyclic extension of $k$ 
is identified with the field discriminant corresponding to the polynomials 
$R^\acyc$ defining $k_1^\acyc/\Q$:

\smallskip
\ft\begin{verbatim}
\p 300
allocatemem(800000000)
{print("From Program II (Normal Split case)");Maxq=500;
ListSigma=List;Listm=List;Nm=4;
Listm=List([3513,1400187,1400709,1400790]);
\\List of fields k with total capitulation of H'_k
\\***choose between List and interval:***
\\for(km=1,Nm,m=Listm[km];
for(m=5,10^5,if(core(m)!=m,next);
if(Mod(m,9)!=3,next); \\***First condition***
Pstar=x^2-m/3;kstar=bnfinit(Pstar);hkstar=kstar.no;
if(Mod(hkstar,3)!=0,next); \\***Second condition***
P=x^2+m;k=bnfinit(P);hk=k.no;Delta=0;
Lstar=List;Cstar=kstar.clgp[2];rkstar=0;
d=matsize(Cstar)[2];for(j=1,d,C= Cstar[j];
if(Mod(C,3)==0,rkstar=rkstar+1;listput(Lstar,j)));
rk=0;expk=1;d=matsize(k.cyc)[2];for(j=1,d,C=k.cyc[j];
v=valuation(C,3);if(v>0,rk=rk+1;expk=max(expk,3^v)));
exptak=1;d=matsize(k.cyc)[2];for(j=1,d,C=k.cyc[j];
v=valuation(C,3); exptakC=C/3^v;exptak=lcm(exptakC,exptak));

\\STRUCTURE OF T_k AND COMPUTATION OF Val:
nu=valuation(expk,3)+2;
Kpnu=bnrinit(k,3^nu);Hpnu=Kpnu.cyc;e=matsize(Hpnu)[2];
Tk=List;ot=1;rt=0;for(j=3,e,c=Hpnu[j];v=valuation(c,3);
if(v>0,listput(Tk,3^v);rt=rt+1;ot=ot*3^v));
if(rt!=rkstar,next); \\***Third condition***
ow=1;if(Mod(m,9)==3,ow=3);Wk=ow;hk3=3^valuation(hk,3);
otbp=ot/ow;ram=Ramified;if(hk3!=otbp,ram=Unramified);
Val=valuation(expk*ot/(m*hk),3)+2;

\\WRITINGS OF THE MAIN INVARIANTS:
kro=kronecker(-m,3);Disc=m;if(Mod(-m,4)!=1,Disc=4*m);
print();print("m=",m," Disc=",Disc," kronecker(-m,3)=",kro,
" h_k=",hk," H_k=",k.cyc," T_k=",Tk," h_kstar=",hkstar,
" H_kstar=",kstar.cyc," #H_k=",hk3," #T_k^bp=",otbp,
" #W_k^bp=",Wk,"  k_1^ac/k is ",ram);

\\COMPUTATION OF THE nw RADICALS IN Lw:
Lw0=List;Nrad=rkstar+1;w0=kstar.fu[1];listput(Lw0,w0);
for(i=1,rkstar,Idstar=kstar.clgp[3][Lstar[i]];
alpha=bnfisprincipal(kstar,idealpow(kstar,Idstar,Cstar[Lstar[i]])); 
wi=Mod(kstar.zk[1]*alpha[2][1]+kstar.zk[2]*alpha[2][2],Pstar);
listput(Lw0,wi));Lw=List;nw=3^Nrad-1;LD=List;for(i=1,nw,
D=digits(i,3);if(D[1]!=2,listput(LD,D)));nw=nw/2;for(i=1,nw,
d=matsize(LD[i])[2];d0=Nrad+1-d;w=1;for(j=d0,Nrad,
w=w*Lw0[j]^LD[i][j-d0+1]);listput(Lw,w));print("Lw=",Lw);

\\RESEARCH OF THE SOLUTION AMONG THE nw RADICALS:
\\Begin J
for(J=1,nw,w=Lw[J];Nw=norm(w);ispower(Nw,3,&a);
Q=x^3-3*a*x-trace(w);Q=polredbest(Q);

\\RANK rk=1:
if(rk==1,R=polcompositum(P,Q)[1];
DiscR=nfdisc(R);if(valuation(DiscR,3)>3,
print("J=",J," w=",w," Q=",Q," does not define k_1^ac"));
if(valuation(DiscR,3)<=3,Delta=Delta+1;
kacyc=bnfinit(R);hkacyc=kacyc.no;
print("SOLUTION:J=",J," w=",w," Q^ac=",Q);
print("H_kacyc=",kacyc.cyc)));

\\RANK rk>1:
if(rk>1,Test=0;e0=2;dres=1;
forprime(q=5,Maxq,if(kronecker(-m,q)!=1,next);
A= component (component(idealfactor(k,q),1),1);
B=idealpow(k,A,exptak*expk);b=bnfisprincipal(k,B);
pib=Mod(k.zk[1]*b[2][1]+k.zk[2]*b[2][2],P)^(3^dres-1)-1;Log=pib;
t=1;while(t<=e0*(Val+log(t)+1),t=t+1;Log=Log-(-1)^t*pib^t/t);
Log=lift(Log);C1=polcoeff(Log,1);if(valuation(C1,3)<Val,next);
Qq=Q+O(q);if(polisirreducible(Qq)==1,print("J=",J," q=",q,
" Qq irreducible");Test=1;break));if(Test==0,Delta=Delta+1;
R=polcompositum(P,Q)[1];kacyc=bnfinit(R);hkacyc=kacyc.no;
print("SOLUTION:J=",J," w=",w," Q^ac=",Q);
print("H_kacyc=",kacyc.cyc))));
\\End J
if(Delta!=1,listput(ListSigma,m));

\\TEST ABOUT W_k^bp -- PRINT OF Val AND unit index a:
TestW=0;for(j=1,nw,w=Lw[j];Nw=norm(w);ispower(Nw,3,&a);
Q=x^3-3*a*x-trace(w);Q=polredbest(Q);
R=polcompositum(P,Q)[1];if(valuation(nfdisc(R),3)>3,TestW=1));
if(TestW==0,print("W_k^bp is not direct factor of T_k"));
if(TestW==1,print("W_k^bp is direct factor of T_k"));
print("Val=",Val," a=",3*3^-lift(Mod(valuation(hkacyc/hk,3),2)));

\\COMPUTATION OF THE PARTIAL CAPITULATIONS:
if(Delta==1,Cacyc=kacyc.clgp;d=matsize(Cacyc[2])[2];
G=nfgaloisconj(kacyc);Id=x;for(k=1,6,Z=G[k];ks=1;while(Z!=Id,
Z=nfgaloisapply(kacyc,G[k],Z);ks=ks+1);if(ks==3,s=G[k];break));
print("Algebraic norm of H_kacyc=",kacyc.cyc," H_k=",k.cyc);
for(i=1,d,X0=Cacyc[3][i];X=1;for(t=1,3,
Xs=nfgaloisapply(kacyc,s,X);X=idealmul(kacyc,X0,Xs));
Y=bnfisprincipal(kacyc,X)[1];Normalg=List;
for(j=1,d,v=valuation(Cacyc[2][j],3);
c=lift(Mod(Y[j],3^v));listput(Normalg,c,j));
print("Norm of the component ",i," of H_kacyc: ",Normalg))); 

\\STRUCTURE OF T_kstar, COMPUTATION OF R_kstar: 
Ustar=kstar.fu[1]^2;h=valuation(norm(Ustar-1),3);
nustar=valuation(expkstar,3)+h/2+2;
Kpnustar=bnrinit(kstar,3^nustar);Hpnustar=Kpnustar.cyc;
e=matsize(Hpnustar)[2];Tkstar=List;rtstar=0;otstar=1;
for(j=2,e,c=Hpnustar[j];v=valuation(c,3);if(v>0,
listput(Tkstar,3^v);rtstar=rtstar+1;otstar=otstar*3^v));
hkstar3=3^valuation(hkstar,3);Rkstar=otstar/hkstar3;
print("H_kstar=",kstar.cyc," T_kstar=",Tkstar," #R_kstar=",Rkstar);
if(rtstar==rkstar,print("R_kstar is not direct factor of T_kstar"));
if(rtstar!=rkstar,print("R_kstar is direct factor of T_kstar"));
print(ListSigma))}
\end{verbatim}\ns

\subsection{PROGRAM III (Special Split case)}\label{program3}

It concerns the case $m \equiv 3 \pmod 9$, $\CH_{k^*} \ne 1$, and 
$\Ccl ({\mathfrak p}^*) \in \CH_{k^*}^3$ (including the case ${\mathfrak p}^*$ 
$3$-principal) equivalent to $\rk(\CT_k) = \rk(\CH_{k^*})+1$ (see Theorem 
\ref{wp} for the definition of the radical $W_{\chi^*} = Y_{k^*} \oplus 
\langle w_{{\mathfrak p}^*} \rangle$, of $3$-rank $\rk(\CH_{k^*})+2$, 
with $w_{{\mathfrak p}^*}$ such that ${\mathfrak p}^* = 
(w_{{\mathfrak p}^*})\,{\mathfrak a}^{* 3}$, which is not a pseudo-unit 
nor the $S^*$-unit $\eta$):

\smallskip
\ft\begin{verbatim}
\p 300
allocatemem(800000000)
{print("From Program III (Special Split case)");Maxq=500;
ListSigma=List;Listm=List;Nm=14;
Listm=List([3387,4962,5862,7257,8139,8913,11217,13413,
13953,14862,17814,18714,19803,20703]);
\\List of fields k with total capitulation of H_k = Z/3Z
\\***choose between List and interval:***
\\for(km=1,Nm,m=Listm[km];
for(m=12,10^5,if(core(m)!=m,next);
if(Mod(m,9)!=3,next); \\***First condition***
Pstar=x^2-m/3;kstar=bnfinit(Pstar);
hkstar=kstar.no;
if(Mod(hkstar,3)!=0,next); \\***Second condition***
P=x^2+m;k=bnfinit(P);hk=k.no;Delta=0;
Lstar=List;Cstar=kstar.clgp[2];rkstar=0;
d=matsize(Cstar)[2];for(j=1,d,C= Cstar[j];
if(Mod(C,3)==0,rkstar=rkstar+1;listput(Lstar,j)));
rk=0;expk=1;d=matsize(k.cyc)[2];for(j=1,d,C=k.cyc[j];
v=valuation(C,3);if(v>0,rk=rk+1;expk=max(expk,3^v)));
exptakstar=1;d=matsize(kstar.cyc)[2];for(j=1,d,C=kstar.cyc[j];
v=valuation(C,3); exptakstarC=C/3^v;
exptakstar=lcm(exptakstarC,exptakstar));
exptak=1;d=matsize(k.cyc)[2];for(j=1,d,C=k.cyc[j];
v=valuation(C,3);exptakC=C/3^v;exptak=lcm(exptakC,exptak));

\\STRUCTURE OF T_k AND COMPUTATION OF Val:
nu=valuation(expk,3)+2;
Kpnu=bnrinit(k,3^nu);Hpnu=Kpnu.cyc;e=matsize(Hpnu)[2];
Tk=List;ot=1;rt=0;for(j=3,e,c=Hpnu[j];v=valuation(c,3);
if(v>0,listput(Tk,3^v);rt=rt+1;ot=ot*3^v));
if(rt!=rkstar+1,next); \\***Third condition***
ow=1;if(Mod(m,9)==3,ow=3);Wk=ow;hk3=3^valuation(hk,3);
otbp=ot/ow;ram=Ramified;if(hk3!=otbp,ram=Unramified);
Val=valuation(expk*ot/(m*hk),3)+2;

\\WRITINGS OF THE MAIN INVARIANTS:
kro=kronecker(-m,3);Disc=m;if(Mod(-m,4)!=1,Disc=4*m);
print();print("m=",m," Disc=",Disc," kronecker(-m,3)=",kro,
" h_k=",hk," H_k=",k.cyc," T_k=",Tk," h_kstar=",hkstar,
" H_kstar=",kstar.cyc," #H_k=",hk3," #T_k^bp=",otbp,
" #W_k^bp=",Wk,"  k_1^ac/k is ",ram);

\\COMPUTATION OF THE nw RADICALS IN Lw:
Lw0=List;Nrad=rkstar+2;

\\Begin
\\COMPUTATION OF w_p^*:
Sideal=component(component(idealfactor(kstar,3),1),1);
Sideal=idealpow(kstar,Sideal,exptakstar);
Z=kstar.clgp[3];
Y=bnfisprincipal(kstar,Sideal)[1];d=matsize(Y)[1];Ideal=1;
for(i=1,d,Ideal=idealmul(kstar,Ideal,idealpow(kstar,Z[i],Y[i])));
X=idealdiv(kstar,Sideal,Ideal);c=bnfisprincipal(kstar,X);
Gamma=Mod(kstar.zk[1]*c[2][1]+kstar.zk[2]*c[2][2],Pstar);
wstar=Gamma*norm(Gamma);listput(Lw0,wstar);

\\UNITS and PSEUDO-UNITS:
w0=kstar.fu[1];listput(Lw0,w0);
for(i=1,rkstar,Idstar=kstar.clgp[3][Lstar[i]];
alpha=bnfisprincipal(kstar,idealpow(kstar,Idstar,Cstar[Lstar[i]])); 
wi=Mod(kstar.zk[1]*alpha[2][1]+kstar.zk[2]*alpha[2][2],Pstar);
listput(Lw0,wi));

\\COMPLETE LIST OF RADICALS:
Lw=List;nw=3^Nrad-1;LD=List;for(i=1,nw,
D=digits(i,3);if(D[1]!=2,listput(LD,D)));nw=nw/2;for(i=1,nw,
d=matsize(LD[i])[2];d0=Nrad+1-d;w=1;for(j=d0,Nrad,
w=w*Lw0[j]^LD[i][j-d0+1]);listput(Lw,w));print("Lw=",Lw);
\\End

\\RESEARCH OF THE SOLUTION AMONG THE nw RADICALS:
\\Begin J
for(J=1,nw,w=Lw[J];Nw=norm(w);ispower(Nw,3,&a);
Q=x^3-3*a*x-trace(w);Q=polredbest(Q);

\\UNRAMIFIED CASE with RANK rk=1:
if(hk3>otbp & rk==1,Disc=nfdisc(Q);if(valuation(Disc,3)>1,
print("J=",J," w=",w," Q=",Q," does not define k_1^ac"));
if(valuation(Disc,3)<=1,Delta=Delta+1;
R=polcompositum(P,Q)[1];kacyc=bnfinit(R);hkacyc=kacyc.no;
print("SOLUTION:J=",J," w=",w," Q^ac=",Q);
print("H_kacyc=",kacyc.cyc)));

\\SPECIAL CASE:
if(hk3==otbp || rk!=1,Test=0;e0=2;dres=1;
forprime(q=5,Maxq,if(kronecker(-m,q)!=1,next);
A=component(component(idealfactor(k,q),1),1);
B=idealpow(k,A,exptak*expk);b=bnfisprincipal(k,B);
pib=Mod(k.zk[1]*b[2][1]+k.zk[2]*b[2][2],P)^(3^dres-1)-1;Log=pib;
t=1;while(t<=e0*(Val+log(t)+1),t=t+1;Log=Log-(-1)^t*pib^t/t);
Log=lift(Log);C1=polcoeff(Log,1);if(valuation(C1,3)<Val,next);
Qq=Q+O(q);if(polisirreducible(Qq)==1,print("J=",J," q=",q,
" Qq irreducible");Test=1;break));if(Test==0,Delta=Delta+1;
R=polcompositum(P,Q)[1];kacyc=bnfinit(R);hkacyc=kacyc.no;
print("SOLUTION:J=",J," w=",w," Q^ac=",Q);
print("H_kacyc=",kacyc.cyc))));
\\End J
if(Delta!=1,listput(ListSigma,m));

\\TEST ABOUT W_k^bp -- PRINT OF Val AND unit index a:
if(ram==Ramified,
if(rt==rk+1,print("W_k^bp is direct factor of T_k"));
if(rt==rk,print("W_k^bp is not direct factor of T_k")));
if(ram==Unramified,
TestW=0;for(j=1,nw,w=Lw[j];Nw=norm(w);ispower(Nw,3,&a);
Q=x^3-3*a*x-trace(w);Q=polredbest(Q);
R=polcompositum(P,Q)[1];if(valuation(nfdisc(R),3)>3,TestW=1));
if(TestW==0,print("W_k^bp is not direct factor of T_k"));
if(TestW==1,print("W_k^bp is direct factor of T_k")));
print("Val=",Val," a=",3*3^-lift(Mod(valuation(hkacyc/hk,3),2)));

\\COMPUTATION OF THE PARTIAL CAPITULATIONS:
\\Test of principality of P_1:
if(ram==Ramified,
Sideal=component(component(idealfactor(kacyc,3),1),1);
Pr=bnfisprincipal(kacyc,Sideal);
print("Components of Cl(P_1): ",Pr[1]));

if(Delta==1,Cacyc=kacyc.clgp;d=matsize(Cacyc[2])[2];
G=nfgaloisconj(kacyc);Id=x;for(k=1,6,Z=G[k];ks=1;while(Z!=Id,
Z=nfgaloisapply(kacyc,G[k],Z);ks=ks+1);if(ks==3,s=G[k];break));
print("Algebraic norm of H_kacyc=",kacyc.cyc," H_k=",k.cyc);
for(i=1,d,X0=Cacyc[3][i];X=1;for(t=1,3,
Xs=nfgaloisapply(kacyc,s,X);X=idealmul(kacyc,X0,Xs));
Y=bnfisprincipal(kacyc,X)[1];Normalg=List;
for(j=1,d,v=valuation(Cacyc[2][j],3);
c=lift(Mod(Y[j],3^v));listput(Normalg,c,j));
print("Norm of the component ",i," of H_kacyc: ",Normalg)));

\\STRUCTURE OF T_kstar, COMPUTATION OF R_kstar: 
Ustar=kstar.fu[1]^2;h=valuation(norm(Ustar-1),3);
nustar=valuation(expkstar,3)+h/2+2;
Kpnustar=bnrinit(kstar,3^nustar);Hpnustar=Kpnustar.cyc;
e=matsize(Hpnustar)[2];Tkstar=List;rtstar=0;otstar=1;
for(j=2,e,c=Hpnustar[j];v=valuation(c,3);if(v>0,
listput(Tkstar,3^v);rtstar=rtstar+1;otstar=otstar*3^v));
hkstar3=3^valuation(hkstar,3);Rkstar=otstar/hkstar3;
print("H_kstar=",kstar.cyc," T_kstar=",Tkstar," #R_kstar=",Rkstar);
if(rtstar==rkstar,print("R_kstar is not direct factor of T_kstar"));
if(rtstar!=rkstar,print("R_kstar is direct factor of T_kstar"));
print(ListSigma))}
\end{verbatim}\ns

\smallskip
The only cases, in the Special Split case, where $\CW_k^\bp$ is 
not direct factor of $\CT_k$, in the tested interval, are for 
$m=23178$, $92685$, $93207$, $94998$, $116751$, $200406$.

\subsection{PROGRAM IV (Trivial case)}\label{program4}

We assume that $\CH_{k^*} = 1$; then $Y_{k^*} = \langle \varepsilon^* \rangle$, and 
from the relation \ref{reflection} of Theorem \ref{primarity}, $\rk(\CH_k) = 1$ if and only 
if $\varepsilon^*$ is $3$-primary, generating the unique unramified cubic subfield of 
$H_k^\nr$; this field is $k_1^\acyc$ if and only if there is non-disjunction. If $\varepsilon^*$ 
is not $3$-primary, $\CH_k = 1$, all the cubic subfields of $H_k^\nr$ are ramified. 
Whence:

\smallskip
\quad $\bullet$\ If $m \not \equiv 3 \pmod 9$ (i.e., $3$ inert or ramified in $k^*$), 
$W_{\chi^*} = \{\varepsilon^*\}$. 

\smallskip
\quad $\bullet$\ If $m \equiv 3 \pmod 9$, $W_{\chi^*} = \{\varepsilon^*, \eta^*\}$
($w_{{\mathfrak p}^*}^*$ is equivalent to $\eta^*$ since $k^*$ is $3$-principal):

\smallskip
\ft\begin{verbatim}
\p 300
allocatemem(800000000)
{print("From Program IV (Trivial case)");Maxq=500;
ListSigma=List;Listm=List;Nm=52;
Listm=List([23,26,29,31,38,53,59,61,83,87,89,106,107,109,
110,118,129,174,201,246,255,309,318,327,417,453,543,597,
606,714,741,759,813,885,930,966,993,1011,1173,1191,1218,
1254,1362,1371,1506,1515,1533,1713,1821,1830,1902,2001]);
\\List of fields k such that H_k is cyclic non trivial
\\***choose between List and interval:***
\\for(km=1,Nm,m=Listm[km];
for(m=1,10^4,if(m==3,next);if(core(m)!=m,next);
Pstar=x^2-3*m;if(Mod(m,3)==0,Pstar=x^2-m/3);
kstar=bnfinit(Pstar);hkstar=kstar.no;
if(Mod(hkstar,3)==0,next); \\***Unique condition***
P=x^2+m;k=bnfinit(P);hk=k.no;Delta=0;
rk=0;expk=1;if(Mod(hk,3)==0,rk=1;expk=3^valuation(hk,3));
exptak=1;d=matsize(k.cyc)[2];for(j=1,d,C=k.cyc[j];
v=valuation(C,3);exptakC=C/3^v;exptak=lcm(exptakC,exptak));

\\STRUCTURE OF T_k AND COMPUTATION OF Val:
nu=valuation(expk,3)+2;
Kpnu=bnrinit(k,3^nu);Hpnu=Kpnu.cyc;e=matsize(Hpnu)[2];
Tk=List;ot=1;rt=0;for(j=3,e,c=Hpnu[j];v=valuation(c,3);
if(v>0,listput(Tk,3^v);rt=rt+1;ot=ot*3^v));print();
ow=1;if(Mod(m,9)==3,ow=3);Wk=ow;hk3=3^valuation(hk,3);
otbp=ot/ow;ram=Ramified;if(hk3!=otbp,ram=Unramified);
Val=valuation(expk*ot/(m*hk),3)+2;

\\WRITINGS OF THE MAIN INVARIANTS:
kro=kronecker(-m,3);Disc=m;if(Mod(-m,4)!=1,Disc=4*m);
print();print("m=",m," Disc=",Disc," kronecker(-m,3)=",kro,
" h_k=",hk," H_k=",k.cyc," T_k=",Tk," #H_k=",hk3," #T_k^bp=",otbp,
" #W_k^bp=",Wk,"  k_1^ac/k is ",ram);

\\COMPUTATION OF THE nw RADICALS IN Lw:
Lw0=List;w0=kstar.fu[1];listput(Lw0,w0);
if(Mod(m,9)!=3,Lw=Lw0;Nrad=1;nw=1);
if(Mod(m,9)==3,SList=component(idealfactor(kstar,3),1);
Sunit=component(component(bnfsunit(kstar,SList),1),1);
w3=Mod(Sunit,Pstar);w3=w3*norm(w3);listput(Lw0,w3);
Lw=List([w0,w3,w0*w3,w0^2*w3]);Nrad=2;nw=4);
print("Lw=",Lw);

\\RESEARCH OF THE SOLUTION AMONG THE nw RADICALS:
\\Begin J
for(J=1,nw,w=Lw[J];Nw=norm(w);ispower(Nw,3,&a);
Q=x^3-3*a*x-trace(w);Q=polredbest(Q);

\\UNRAMIFIED CASE:
if(hk3>otbp,Disc=nfdisc(Q);if(valuation(Disc,3)>1,
print("J=",J," w=",w," Q=",Q," does not define k_1^ac"));
if(valuation(Disc,3)<=1,Delta=Delta+1;
R=polcompositum(P,Q)[1];kacyc=bnfinit(R);hkacyc=kacyc.no;
print("SOLUTION:J=",J," w=",w," Q^ac=",Q);
print("H_kacyc=",kacyc.cyc)));

\\RAMIFIED CASE:
if(hk3==otbp,Test=0;e0=1;if(Mod(m,3)==0,e0=2);
dres=1;if(kro==-1,dres=2);
forprime(q=5,Maxq,if(kronecker(-m,q)!=1,next);
A=component(component(idealfactor(k,q),1),1);
B=idealpow(k,A,exptak*expk);b=bnfisprincipal(k,B);
pib=Mod(k.zk[1]*b[2][1]+k.zk[2]*b[2][2],P)^(3^dres-1)-1;Log=pib;
t=1;while(t<=e0*(Val+log(t)+1),t=t+1;Log=Log-(-1)^t*pib^t/t);
Log=lift(Log);C1=polcoeff(Log,1);if(valuation(C1,3)<Val,next);
Qq=Q+O(q);if(polisirreducible(Qq)==1,print("J=",J," q=",q,
" Qq irreducible");Test=1;break));if(Test==0,Delta=Delta+1;
R=polcompositum(P,Q)[1];kacyc=bnfinit(R);hkacyc=kacyc.no;
print("SOLUTION:J=",J," w=",w," Q^ac=",Q);
print("H_kacyc=",kacyc.cyc))));
\\End J
if(Delta!=1,listput(ListSigma,m));

\\TEST ABOUT W_k^bp -- PRINT OF Val AND unit index a:
if(Wk==1,print("W_k^bp is not direct factor of T_k"));
if(Wk!=1,
if(ram==Ramified,
if(rt==rk+1,print("W_k^bp is direct factor of T_k"));
if(rt==rk,print("W_k^bp is not direct factor of T_k")));
if(ram==Unramified,
if(otbp==1, print("W_k^bp is direct factor of T_k"));
if(otbp!=1, print("W_k^bp is not direct factor of T_k"))));
print("Val=",Val," a=",3*3^-lift(Mod(valuation(hkacyc/hk,3),2)));

\\COMPUTATION OF THE PARTIAL CAPITULATIONS:
\\(3 non-split in k) -Test of principality of P_1:
if((Mod(-m,3)!=1 & ram==Ramified),
Sideal=component(component(idealfactor(kacyc,3),1),1);
Pr=bnfisprincipal(kacyc,Sideal);
print("Components of Cl(P_1): ",Pr[1]));
\\(3 split in k)-Test of principality of P1/P'1 and P1*P'1:
if((Mod(-m,3)==1 & ram==Ramified),
Sideal1=component(component(idealfactor(kacyc,3),1),1);
Sideal2=component(component(idealfactor(kacyc,3),1),2);
Sminus=idealdiv(kacyc,Sideal1,Sideal2);
Splus=idealmul(kacyc,Sideal1,Sideal2);
Pminus=bnfisprincipal(kacyc,Sminus);
Pplus=bnfisprincipal(kacyc,Splus);
print("Components of Cl(P1/P'1) and Cl(P1*P'1): ", Pminus[1]," ",Pplus[1]));

if(Delta==1,Cacyc=kacyc.clgp;d=matsize(Cacyc[2])[2];
G=nfgaloisconj(kacyc);Id=x;for(k=1,6,Z=G[k];ks=1;while(Z!=Id,
Z=nfgaloisapply(kacyc,G[k],Z);ks=ks+1);if(ks==3,s=G[k];break));
print("Algebraic norm of H_kacyc=",kacyc.cyc," H_k=",k.cyc);
for(i=1,d,X0=Cacyc[3][i];X=1;for(t=1,3,
Xs=nfgaloisapply(kacyc,s,X);X=idealmul(kacyc,X0,Xs));
Y=bnfisprincipal(kacyc,X)[1];Normalg=List;
for(j=1,d,v=valuation(Cacyc[2][j],3);
c=lift(Mod(Y[j],3^v));listput(Normalg,c,j));
print("Norm of the component ",i," of H_kacyc: ",Normalg)));
print(ListSigma))}
\end{verbatim}\ns

\section{Disproof of assertions about \texorpdfstring
{$k^\acyc \cap H_k^\nr/k$}{Lg} when \texorpdfstring{$\CH_k$}{Lg} is cyclic}

Some claims in the literature need rectifications and comments; they are 
related to the unramified sub-extension $k^\acyc \cap H_k^\nr$ in the case 
where $\CH_k$ is cyclic non-trivial.

\subsection{Case \texorpdfstring{$\CH_k \ne 1$}{Lg}, cyclic, and 
\texorpdfstring{$H_k^\nr$}{Lg}
not contained in \texorpdfstring{$k^\acyc$}{Lg}}

Besides the known cases of \cite[Theorem 6.1]{KW2023}, generalized in 
\cite[Theorem 5.6]{Gra2024b} and recalled in Proposition \ref{noncyclic},
about the non-cyclicity of $\CH_{k_n^\acyc}$ for $n \geq n_0^{}$, the cyclic 
case when $k^\acyc \cap H_k^\nr \ne k$ and $H_k^\nr \not\subset k^\acyc$
has some subtleties:

\smallskip
In \cite[Theorem 1]{Oh2015} the Author claims that
( when $m \not\equiv 3 \pmod 9$) any subfield $L$ of $H_k^\nr$,
such that $L$ is cyclic over $k$ and extension of
$k_e^\acyc := k^\acyc \cap H_k^\nr$ for $e \geq 1$ (case
of non linear disjunction of $H_k^\nr$ and $k^\acyc$ over $k$), 
is equal to $k_e^\acyc$; then the Author deduces (in Corollary 1, 
from Theorem 1), that under the same conditions, the linear 
disjunction of $k^\acyc$ and $H_k^\nr$ is equivalent to 
$\rk(\CH_{k^*}) = \rk(\CH_{k})$, which may be false, e.g., $m=362$
where:

\smallskip
\centerline{${\sf H_k=[18]\ \ T_k=[3]\ \ H_{kstar}=[6]\ \ T_k^\bp=[3]\ \  
W_k^\bp=[\ ]\  \  k_1^\acyc/k\  is\  Unramified}$,}

\smallskip\noindent
since from Theorem \ref{TH}\,(1), $\rk(\CT_k) = \rk(\CH_{k^*})$;
the result does not hold. Program I gives three ramified cyclic cubic 
extensions of $k$ (discriminants $2^9\cdot 3^8 \cdot 181^3$) and the 
unique unramified one of discriminant $2^9 \cdot 181^3$.

\smallskip
Nevertheless, Theorem 3, referred to Minardi's thesis, comes from 
classical reflection theorems, and we have given a general statement in 
Theorem \ref{TH}. 

\smallskip
For counterexamples, consider for instance $L=H_k^\nr$ when 
$\CH_k$ is cyclic of order $p^e$, $e \geq 2$. 
Let $H'^\bp_k =: k^\acyc H_0$ with $\Gal(H'^\bp_k/k^\acyc)
=: \langle s \rangle$, and $\Gal(H'^\bp_k/H_0) =: 
\langle \sigma \rangle$ (Schema \eqref{schemalog}); then 
we may assume $H_k^\nr$ fixed by $\sigma^{p^{e-1}}\!\! \cdot s$ and 
$k_{e-1}^\acyc$ fixed by $\langle \sigma^p, s \rangle$, whence 
$\CT_k^\bp \simeq \Z/p\Z$.

\smallskip
From Program I we get the case of $m=1879$:

\begin{verbatim}
m=1879 Disc=1879 kronecker(-m,3)=-1 H_k=[27] T_k=[3] k_1^ac/k is Unramified
Q^ac=x^3-x^2-2*x-16  H_kacyc=[36,4]
\end{verbatim}

These data immediately give $[k^\acyc \cap H_k^\nr : k] = 9$ but $L=H_k^\nr$ 
is not contained in $k^\acyc$. 

\smallskip
Moreover, one can give an independent checking 
by computing, from its definition, the index $[\Log_3(I_k \otimes \Z_3)^- : 
\Log_3(P_k\otimes \Z_3)^-]$; since $\BH_k$ is generated by the class of the prime 
ideal ${\sf A = [7, 3; 0, 1]}$ above $7$, ${\sf A^{27}}$ is generated by:

\smallskip
\centerline{${\sf a=Mod(54389650573 +11613066642*(-1/2+1/2*x), x^2+ 1879)}$. }

\smallskip
Computing $\Log_3(a)$ with the usual series gives
$\Log_3(a)^- \equiv 36 \sqrt{-1879}\! \pmod{81}$, whence:
$$\Log_3(I_k \otimes \Z_3)^- = \Log_3(A)^- \,\Z_3 + \Log_3(P_k\otimes \Z_3)^- 
= \ffrac{9}{27} \,\Z_3 \sqrt{-1879} + \Log_3(P_k\otimes \Z_3)^-; $$ 
$\Log_3(P_k\otimes \Z_3)^- = 3\,\Z_3 \sqrt{-1879}$ (Lemma \ref{logP})
gives $[\Log_3(I_k \otimes \Z_3)^- : \Log_3(P_k\otimes \Z_3)^- ] = 9$. 

\smallskip
Program I (case $m \not \equiv 3 \pmod 9$) gives $m=362$, $367$, 
$419$, $1018$, $1087$, $1193$$,\,\ldots$.

\smallskip
Program II gives no solutions (indeed, the conditions $\rk(\CH_k)=\rk(\CH_{k^*})=1$ 
implying $Y_{k^*}/Y_{k^*}^\prim \ne 1$, whence $\CW_k^\bp$ direct factor, are 
incompatible with $\rk(\CH_{k^*}) = \rk(\CT_k)=2$).

\smallskip
Program III gives $m =7257$, $8913$, $22161$, $23259$, 
$25635$, $27102$, $34401$$,\,\ldots$.

\smallskip
Program IV gives $m=1821$, $2361$, $5007$, $5430$, 
$7113$, $7545$$,\,\ldots$. 

\smallskip
This shows that Schema \eqref{schemalog} does exist with cyclic 
$\CH_k$'s of orders $\geq 9$.

\subsection{Case \texorpdfstring{$\CH_k \ne 1$}{Lg}, cyclic, and
\texorpdfstring{$H_k^\nr$}{Lg} contained in 
\texorpdfstring{$k^\acyc$}{Lg}}\label{Hcyclic}

So, in that case, $\CT_k^\bp = 1$. We have the following 
property giving an example of triviality of the Iwasawa invariants 
for $k^\acyc$ (see \cite[Remarks, p.\,286]{Fu2013} 
using the Lang proof of Chevalley--Herbrand formula \cite{Che1933}
\footnote{In the literature, references for this formula are rather folkloric 
and consist, most often, in citing authors having used or proved it again. 
The modern cohomological proof in Lang's book is of course the most 
useful, even if Lang does not explain correctly the contributions. 
It may be useful to recall historical facts, quoting the following excerpt 
from Chevalley's Thesis (footnote 19, p. 402, originally in french): 

{\it The calculation made here is the generalization to the arbitrary cyclic 
case of the calculation made by Takagi in the cyclic case of prime degree, 
a calculation whose essential idea is already found in Hilbert's Zahlbericht
in the proof of the following theorem: if a relatively cyclic extension $K$ of 
prime degree of $k$ is unramified, there is at least one ideal of $k$ which is 
not principal in $k$ but becomes principal in $K$. The extension to the cyclic 
case of arbitrary degree was made possible by Herbrand's unit theorem.}

This explains why I use the expression: "Chevalley--Herbrand formula", 
once and for all.}).

\begin{theorem}\label{CHcyclic}
Let $p \geq 3$ non-split in $k$. Assume that $\CH_k \simeq \Z/p^e \Z$
and that $H_k^\nr \subset k^\acyc$. Then $\CH_{k_{n}^\acyc} = 1$
for all $n \geq e$ and consequently, $\lambda_p(k^\acyc/k) = \mu_p(k^\acyc/k) = 
\nu_p(k^\acyc/k) = 0$. Moreover, every unit of $H_k^\nr$, local norm in 
$k_{n}^\acyc/H_k^\nr$, is the $(p^{n-e})^{\rm th}$-power of a unit of $H_k^\nr$.
\end{theorem}

\begin{proof}
Since $k^\acyc/k_e^\acyc$ is totally ramified,
Chevalley--Herbrand formula in $k_{n}^\acyc/k$ and 
$k_n^\acyc/k_e^\acyc$ give the following results:

\unitlength=1.15cm 
\begin{equation}\label{CH}
\begin{aligned}
\vbox{\hbox{\hspace{-6.5cm}
\begin{picture}(10.0,4.35)
\put(3.3,2.3){$H_k^\nr\! =\! k_e^\acyc$}
\put(3.3,4.45){$k_{n}^\acyc$}
\put(3.45,0.2){$k$}
\put(1.8,0.2){$\Q$}
\put(2.25,0.1){\hbox{\tiny$p\,{\rm non\,split}$}}
\put(3.50,0.5){\line(0,1){1.7}}
\put(3.50,2.6){\line(0,1){1.7}}
\put(2.15,0.3){\line(1,0){1.2}}
\bezier{300}(3.9,0.3)(6.0,2.3)(3.9,4.4)
\bezier{200}(3.9,2.5)(4.2,3.4)(3.9,4.3)
\put(4.1,3.0){$g_{n-e}$}
\put(2.8,3.3){$p^{n-e}$}
\put(3.08,1.3){$p^e$}
\put(5.0,2.3){$G_{n}$}
\put(5.6,2.8){$\ds (1)\  \order \CH_{k_{n}^\acyc}^{G_{n}} =
\order \CH_k \times \frac{p^{n-e}}{p^{n} \times 1} = 1$}
\put(5.6,1.8){$\ds (2)\  \order \CH_{k_{n}^\acyc}^{g_{n-e}} =
1 \times \frac{(p^{n-e})^{p^e}}{p^{n-e} \times (E_{H_k^\nr} : 
E_{H_k^\nr} \cap \CN_{k_{n}^\acyc/H_k^\nr})} = 1$,}
\put(5.6,1.1){\hbox{where $\CN_{k_{n}^\acyc/H_k^\nr}$ is
the group of local norms in $k_{n}^\acyc/H_k^\nr$,}}
\put(5.6,0.6){\hbox{hence a global norm in 
 $k_{n}^\acyc/H_k^\nr$}}
\end{picture} }}
\end{aligned}
\end{equation}
\unitlength=1.0cm 

Formula (1) holds for all $n \geq e$; then $k_{n}^\acyc$ is $p$-principal.
So, the unique prime dividing $p$ in $k$ totally splits in 
$k_e^\acyc$, which implies Formula (2) with $\CH_{k_e^\acyc} = 1$.
The comparison of the formulas leads to 
$(E_{H_k^\nr} : E_{H_k^\nr} \cap \CN_{k_n^\acyc/H_k^\nr}) = 
(p^{n-e})^{p^e-1}$, hence $E_{H_k^\nr} \cap \CN_{k_{n-e}^\acyc/H_k^\nr} 
= (E_{H_k^\nr})^{p^{n-e}}$ by maximality of the index since the $\Z$-rank
of $E_{H_k^\nr}$ is $p^e-1$.
\end{proof}

The assertion \cite[Theorem 2]{Oh2015} is incorrect, probably due to a 
misinterpretation of the distribution of the inertia groups of the $p$-places of 
$H_k^\nr$ in the compositum $\wt {H_k^\nr}$ of the $\Z_p$-extensions of 
$H_k^\nr$, from the action of complex conjugation $\tau$ in $\Gal(\wt {H_k^\nr}/
H_k^\nr)$. I thank Jaulent having answered me, in a personal communication 
\cite{Jau2025}, what would be a proof, in this Galois context, of the total 
ramification of $\wt {H_k^\nr}/H_k^\nr$ contrary to the reasoning of the Author; 
Jaulent proves that there exists a unique $p$-place of $H_k^\nr$, fixed by 
$\tau$, and $(p^e-1)/2$ pairs of conjugate places, giving rise to a direct compositum 
of the $p^e$ inertia groups in $\Gal(\wt {H_k^\nr}/H_k^\nr)$, which implies the total 
ramification of $\wt {H_k^\nr}/H_k^\nr$. This method is particularly tricky 
compared to the use of Chevalley--Herbrand formula.

\begin{remark}
We still assume $p$ non split in $k$.
If $n \leq e$, $\order \CH_{k_n^\acyc}^{G_n} = p^{e-n}$ 
with $\CH_{k_n^\acyc}^{G_n} = \BJ_{k_n^\acyc/k} (\CH_k)$
since $\CH_{k_n^\acyc}^\ram = 1$; so this yields
$\order \BN_{k_n^\acyc/k} (\CH_{k_n^\acyc}^{G_n}) = p^{e-n}$.
The order of the second element of the filtration of $\CH_{k_1^\acyc}$ 
(Definition \ref{filtre}) is $\order\CH_{k_n^\acyc}^2 = 
\order\CH_{k_n^\acyc}^1 \times \frac{p^{e-n}}{p^{e-n}} = 
\order\CH_{k_n^\acyc}^1$, whence $\order \CH_{k_n^\acyc} 
= p^{e-n}$ for $n \leq e$, with $\CH_{k_n^\acyc}=\BJ_{k_n^\acyc/k} (\CH_k)$.

\smallskip
When $p$ splits in $k$, see \cite[Lemme 4]{Jau2024a} 
for the properties of totally $p$-adic fields.
\end{remark}

\section{Capitulations of \texorpdfstring{$\CH_k$}{Lg} in the anti-cyclotomic
\texorpdfstring{$\Z_p$-extension of $k$}{Lg}}\label{verifcap}

We begin by an elementary necessary condition which gives rise
to a severe constraint which is in fact a crucial point of the problem:

\begin{proposition}\label{NCcap1}
Assume that $k^\acyc/k$ is totally ramified from the layer $k_e^\acyc$, 
$e \geq 0$. Then a necessary condition for the complete capitulation 
of $\CH_k$ in $k^\acyc$ is the existence of $n_0^{} \geq e$ such that 
$\order \CH_{k_n^\acyc}^\ram = \order \CH_k \times p^{(n - e) 
\cdot \order S_k - n}$ for all $n \geq n_0^{}$, where $\CH_{k_n^\acyc}^\ram$ is the 
subgroup generated by the $p$-classes of the ideals $\prod_{{\mathfrak P}_n
\mid {\mathfrak p}}{\mathfrak P}_n$, for ${\mathfrak p} \in S_k$.
\end{proposition}

\begin{proof}
Since the group of units of $k$ is trivial one has, for all $n \geq 0$, the formula:
\begin{equation*}
\CH_{k_n^\acyc}^{G_n} = \BJ_{k_n^\acyc/k}(\CH_k) \cdot \CH_{k_n^\acyc}^\ram,
\end{equation*} 

\noindent
as particular case of the classical exact sequence in a $p$-cyclic extension $L/K$:
\begin{equation}\label{SECH}
1 \to \BJ_{L/K}(\CH_K) \cdot \CH_{L}^\ram \too \CH_{L}^{\Gal(L/K)} 
\too E_K \cap \BN_{L/K} (L^\times)/\BN_{L/K} (E_L) \to 1.
\end{equation}

Then the Chevalley--Herbrand formula, in $k_n^\acyc/k$, $n \geq e$, is
$\order \CH_{k_n^\acyc}^{G_n} = \order \CH_k \times 
p^{(n - e) \cdot \order S_k - n}$, where $S_k$ contains one 
or two $p$-places. So we have, for all $n \geq e$, 
$\order \big(\BJ_{k_n^\acyc/k}(\CH_k) \cdot 
\CH_{k_n^\acyc}^\ram \big) =  \order \CH_k \times p^{(n - e) 
\cdot \order S_k - n}$. 
The capitulation of $\CH_k$ in $k^\acyc$ implies the existence of $n_0^{}$ large 
enough such that $\order \CH_{k_n^\acyc}^\ram = \order \CH_k \times 
p^{(n - e) \cdot \order S_k - n}$ for all $n \geq n_0^{}$, thus 
$\order \CH_{k_n^\acyc}^\ram = \order \CH_k$ (resp. $\order \CH_k \times p^n$) 
in the case $\order S_k = 1$ (resp. $\order S_k = 2$) totaly ramified in $k^\acyc/k$. 
\end{proof}

\begin{remark}\label{Pstructure}
Let ${\mathfrak P}'_n := {\mathfrak P}_n^\tau$ be the complex conjugate of 
${\mathfrak P}_n$. If $p$ does not split in $k$, $\langle \Ccl_n({\mathfrak P}_n \rangle$ 
is a plus component and $\BJ_{k_n^\acyc/k}(\CH_k)$ a minus component.
If $p$ splits in $k$, $\BJ_{k_n^\acyc/k}(\CH_k)$ and $\langle \Ccl_n({\mathfrak P}_n 
\cdot {\mathfrak P}'^{-1}_n)\rangle$ are minus components, while 
$\langle \Ccl_n({\mathfrak P}_n \cdot {\mathfrak P}'_n) \rangle$ is a plus component. 
So, $\CH_k$ capitulates (partially or not) in $k_n^\acyc$ as soon as 
$\langle \Ccl_n({\mathfrak P}_n \rangle$,
$\langle \Ccl_n({\mathfrak P}_n \cdot {\mathfrak P}'^{-1}_n)\rangle$ and
$\langle \Ccl_n({\mathfrak P}_n \cdot {\mathfrak P}'_n) \rangle$ have suitable 
structures implying $\order \BJ_{k_n^\acyc/k}(\CH_k) < \order \CH_k$; 
for this the following decompositions may force the result:
$$\CH_{k_n^\acyc}^{G_n} = \BJ_{k_n^\acyc/k}(\CH_k) \plus \,
\langle \Ccl_n({\mathfrak P}_n \rangle, \  {\rm resp.} \  \big[ \BJ_{k_n^\acyc/k}(\CH_k) 
\cdot \langle \Ccl_n({\mathfrak P}_n \cdot {\mathfrak P}'^{-1}_n)\rangle \big] \plus \,
\langle \Ccl_n({\mathfrak P}_n \cdot {\mathfrak P}'_n) \rangle, $$
with $\order \CH_{k_n^\acyc}^{G_n} = \order \CH_k \times p^{n\,(\order S_k - 1)}$
for $k^\acyc/k$ totally ramified. Let's examine the two possibilities:
\end{remark}

\subsection{Capitulation of \texorpdfstring{$\CH_k$}{Lg} in 
\texorpdfstring{$k^\acyc/k$}{Lg} when
\texorpdfstring{$p$ does not split in $k$}{Lg}}

We assume that $p \geq 3$ does not split in $k$ and that $k^\acyc/k$ is 
totally ramified. From the above remark, a capitulation (partial or not) in 
$k_n^\acyc$, $n \geq 1$, is equivalent to the non-$p$-principality of 
${\mathfrak P}_n$.

\smallskip
The case $p=3$, $n = 1$, $\CH_k \simeq \Z/3\Z$, may be illustrated 
as follows; we have obtained many capitulations, where, as predicted 
by the Kundu--Washington result and \cite[Theorem 5.6)]{Gra2024b}, when
$\CH_{k_1^\acyc} \simeq \Z/3\Z \times \Z/3\Z$ or $\Z/9\Z \times \Z/9\Z$. 
This holds for the following cases where ${\mathfrak P}_1$ is never 
$3$-principal:

\smallskip\noindent
$m \in$\,$\{$$298$, $397$, $622$, $643$, $685$, $706$, $771$,
$835$,\,$\ldots$$\}$ with the structure $\Z/3\Z \times \Z/3\Z$, 

\smallskip\noindent
$m \in$\,$\{$$762$, $2355$, $3073$, $3469$, $4405$, $4798$\,$\ldots$$\}$ 
with the structure $\Z/9\Z \times \Z/9\Z$. 

\smallskip 
More precisely, we can state an analogous result of ``smooth complexity'':

\begin{theorem}\label{capitule2}
Assume that $p \geq 3$ does not split in $k$, that $k^\acyc/k$ is totally 
ramified, that $\CH_k \simeq \Z/p\Z$ and $\CH_{k_1^\acyc}^{} \simeq 
\Z/p\Z \times \Z/p\Z$. Then $\CH_k$ capitulates in $k_1^\acyc$. 
\end{theorem}

\begin{proof}
We consider the first two elements of the filtration of $\CH_{k_1^\acyc}$ 
(cf. Definition \ref{filtre}). We have $\order \CH_{k_1^\acyc}^1 = 
\order \CH_k = p$ and $\order(\CH_{k_1^\acyc}^2/\CH_{k_1^\acyc}^1) =
\ds \frac{\order \CH_k}{\order \BN_{k_1^\acyc/k}(\CH_{k_1^\acyc}^1)} = p$
since $\BN_{k_1^\acyc/k}(\CH_{k_1^\acyc}^1) = (\CH_k)^p \cdot 
\langle \Ccl({\mathfrak p}) \rangle =1$; whence $\order \CH_{k_1^\acyc}^2 
= p \cdot \order \CH_{k_1^\acyc}^1= p^2 = \order \CH_{k_1^\acyc}$ 
by assumption, and $\CH_{k_1^\acyc}^2 \simeq \Z/p\Z \times \Z/p\Z$; but
$\BNu_{k_1^\acyc/k} = (\sigma_1-1)^2 \cdot A_2+p \cdot B_2,
\ \, A_2, B_2 \in \Z[\sigma_1-1, p]$ (apply \cite[Theorem 2.8]
{Gra2024a} for any $p$, $N=1$, $s=0$, $k=2 \in [1, p-1]$, $f(k) = 1$), 
which yields the relation $\BNu_{k_1^\acyc/k}(\CH_{k_1^\acyc}) = 
\BNu_{k_1^\acyc/k}(\CH_{k_1^\acyc}^2) = 1$. This capitulation
implies that $\Ccl_1({\mathfrak P}_1)$ is of order $p$.
\end{proof}

\subsection{Capitulation of \texorpdfstring{$\CH_k$}{Lg} in 
\texorpdfstring{$k^\acyc/k$}{Lg} when 
\texorpdfstring{$p$ splits in $k$}{Lg}}

\begin{remark}\label{independent}
Taking into account Proposition \ref{NCcap1}, for $e=0$, $n=1$, 
$\order \CH_k=p$, and the constraint $\order \CH_{k_1^\acyc}^\ram = 
\order \CH_k \times p$ for a capitulation in $k_1^\acyc$, we see that necessarily the
$p$-classes $\Ccl_1 \big({\mathfrak P}_1 \big)$ and $\Ccl_1 \big({\mathfrak P}'_1 \big)$ 
of $k_1^\acyc$ must be independent or of order $p^2$. For instance, 
for $p=3$, $m=302$, one finds:

\smallskip
\ft\begin{verbatim}
From Program I (Non Split case) m=302 Disc=1208 h_k=12 H_k=[12] 
Q^ac=x^3-93*x-458    H_kacyc=[12,3]    Val=3    a=1
Components of Cl(P1/P'1) and Cl(P1*P'1): [6,1] [8,2]
Algebraic norm of H_kacyc=[12,3]
Norm of the component 1 of H_kacyc: [0,0]
Norm of the component 2 of H_kacyc: [0,0]
\end{verbatim}\ns

\noindent
proving that $\CH_{k_1^\acyc}^\ram = \CH_{k_1^\acyc} =
\langle \Ccl_1({\mathfrak P}_1) \rangle \oplus 
\langle \Ccl_1({\mathfrak P}'_1) \rangle \simeq \Z/3\Z \times \Z/3\Z$.

For $m=503$, $\CH_k \simeq \Z/3\Z$, $\CH_{k_1^\acyc} \simeq 
\Z/9\Z \times \Z/3\Z$, $\Ccl_1({\mathfrak P}_1) = h_1^\acyc$,
$\Ccl_1({\mathfrak P}'_1) = h_1^{\acyc \,62}$, for $h_1$ of order $63$,
hence $\Ccl_1 \big({\mathfrak P}_1/{\mathfrak P}'_1 \big)$ of $3$-order $9$,
$\Ccl_1 \big({\mathfrak P}_1 {\mathfrak P}'_1 \big) = 1$. But
there is not capitulation. 
\end{remark}

\begin{theorem}\label{capitule1}
Assume that $p \geq 3$ splits in $k$, that $k^\acyc/k$ is totally ramified, 
that $\CH_k \simeq \Z/p\Z$ and $\CH_{k_1^\acyc}^{} \simeq 
\Z/p\Z \times \Z/p\Z$. Then $\CH_k$ capitulates in $k_1^\acyc$.
\end{theorem}

\begin{proof}
The Chevalley--Herbrand formula, applied in $k_1^\acyc/k$, gives by 
assumption $\order \CH_{k_1^\acyc}^{G_1} = p^2$, whence $\CH_{k_1^\acyc} 
= \CH_{k_1^\acyc}^{G_1}$. So, $\BNu_{k_1^\acyc/k}(\CH_{k_1^\acyc}) 
= \CH_{k_1^\acyc}^p = 1$ proving the result. 
\end{proof} 

This is another illustration of the notion of ``smooth complexity''
(apply \cite[Theorem 2.8]{Gra2024a} with 
$L=k_1^\acyc$, $e(L)=1$, $s(L) = 0$, $b(L) \in [1, p-1]$ since 
$b(L)=1$ from $\CH_{k_1^\acyc} = \CH_{k_1^\acyc}^{G_1}$).
For $p = 3$, this holds for:

\smallskip\noindent
$m \in$\ft$\{$$302$, $602$, $617$, $713$, $863$, 
$1007$, $1046$, $1427$, $1454$, $1478$, $1547$, $1583$, 
$1619$, $1865$, $1895$, $1997$, $2141$, $2213$, $2273$, 
$2555$, $2627$, $2801$, $2822$, $2879$, $3359$, $3617$, 
$3653$, $3665$, $3899$, $3917$,\,$\ldots$$\}$.\ns

\smallskip
When $p$ splits in $k$, see \cite[Lemme 4]{Jau2024a} 
for the properties of these totally $p$-adic fields.

\subsection{Characterization of \texorpdfstring{$\lambda_p(K/k) 
= \mu_p(K/k) = 0$, for $p$}{Lg} non-split in \texorpdfstring{$k$}{Lg}}

Many sufficient conditions are given in the literature in the easier case 
of a unique $p$-place totally ramified in a $\Z_p$-extension $K/k$; 
most of them are recalled or proved in \cite[Theorems 2.2, 3.3, Propositions 3.4, 
3.6, 3.7]{KW2023} including that given by Theorem \ref{CHcyclic} using 
Chevalley--Herbrand formula, and some others, equivalent to stability 
theorems (cf. \cite[Theorem 3.1]{Gra2022} for a more complete bibliography
and generalizations).
 
\smallskip
A characterization, based on the principle of capitulation, does exist (Greenberg 
\cite[Theorem 1]{Gree1976} for totally real $k$, Gerth \cite[Theorem 1.3\,(c)]
{Ger1977} for any base field $k$), saying, still under a unique place totally 
ramified in the $\Z_p$-extension $K/k$, that $\lambda_p(K/k) = \mu_p(K/k) = 0$ 
if and only if $\CH_k$ capitulates in $K$. If, moreover, $k$ is an imaginary 
quadratic field, the proof of this result becomes elementary; we 
give the proof, by adapting some reasonings of Greenberg and Gerth. 
The case $K=k^\acyc$ may be illustrated by our programs. 
In fact, the proof is valid for any $K \ne k^\cyc$ since capitulation of 
$\CH_k \ne 1$ in $k^\cyc$ does not hold.

\begin{theorem}\label{lambdamu}
If $p \geq 3$ does not split in $k = \Q(\sqrt{-m})$ and is totally ramified in $K/k$,
a necessary and sufficient condition for $\lambda_p(K/k) = \mu_p(K/k) = 0$ is 
that $\CH_k$ capitulates in $K$.
\end{theorem}

\begin{proof}
We use some obvious simplified notations and give first general results.
For $n \geq 0$, set $G_n = \Gal(K_n/k) =: \langle \sigma_n \rangle$, where 
$\sigma_n = \sigma_{\vert _{K_n}}$ for a fixed topological generator $\sigma$ 
of $\Gal(K/k)$. Consider $m \geq n \geq 0$, and let $g_n^m = \Gal(K_m/K_n)$.
Chevalley--Herbrand formula leads to:

\smallskip
\quad (i) $\order\CH_n^{G_n} = \order \CH_0$ (with $\CH_0 := \CH_k$), 
$\order\CH_m^{g_n^m} = \order \CH_n$, and $\CH_n^{G_n} = \BJ_0^n(\CH_0) 
\cdot \CH_n^\ram = \BJ_0^n(\CH_0) \cdot \langle \Ccl_n({\mathfrak P}_n) \rangle$, 
where ${\mathfrak P}_n$ is the unique prime ideal dividing $p$ in $K_n$.

\smallskip
\quad (ii) The order of the sub-module $\CH_n^\ram$ is bounded by 
$\order \CH_0$ for all $n$.

\smallskip
{\bf (a)} Assume that $\lambda_p(K/k) = \mu_p(K/k) = 0$. 
Let $n_0^{} \geq 0$ be such that $\order \CH_n = p^\nu$ (constant) 
for all $n \geq n_0^{}$. Thus, $\order \CH_m^{g_n^m} = \order \CH_n 
= p^\nu$, whence $\CH_m^{g_n^m} = \CH_m$ of same order $p^\nu$; 
the relative arithmetic norms $\BN_n^m$ are isomorphisms and the 
relative transfer maps are such that: 
$$\BJ_n^m(\CH_n) = \BJ_n^m(\BN_n^m (\CH_m)) 
= \BNu_n^m (\CH_m) = (\CH_m)^{p^{m-n}}, $$ 
since $\CH_m = \CH_m^{g_n^m}$, yielding, for all $n \geq n_0^{}$, 
capitulation of $\CH_n$ in  $K_m$ for $m$ large enough, since 
$\order \CH_m$ is a constant, whence capitulation of $\CH_n$ in $K$ 
for all $n \geq 0$.

\smallskip
{\bf (b)} Assume that $\CH_k =: \CH_0$ capitulates in $K$.
Let's prove first that for all $n \geq 0$, $\CH_n$ capitulates in $K$.
Assume on the contrary that there exists $n \geq 1$ such that $\CH_n$ 
does not totally capitulate in $K$; since $(\sigma_n-1)^N \to 0$, 
$p$-adically, there exists $c_n \in \CH_n$ such that $c_n$ does not  
capitulate in $K$, but $c_n^{\sigma_n-1}$ capitulates in some $K_m$.

\smallskip
Let $c_m = \BJ_n^m(c_n)$; so $c_m^{\sigma_m-1} = \BJ_n^m(c_n^{\sigma_m-1}) =
\BJ_n^m(c_n^{\sigma_n-1}) = 1$ and $c_m \in \CH_m^{G_m}
= \BJ_0^m(\CH_0) \cdot \langle \Ccl_m({\mathfrak P}_m) \rangle$ 
from (i), and, by assumption, we may take $m$ such that $\BJ_0^m(\CH_0) = 1$;
so $c_m \in \langle \Ccl_m({\mathfrak P}_m) \rangle$.

In an extension $K_s/K_m$, we have $\BJ_m^s({\mathfrak P}_m)
= {\mathfrak P}_s^{p^{s-m}}$, and for $s$ large enough, ${\mathfrak P}_s^{p^{s-m}}$
is $p$-principal from (ii); hence $c_m$ capitulates in $K_s$, and, a fortiori, $c_n$
capitulates (contradiction).

\smallskip
So, all the $\CH_n$'s capitulate in $K$. Using the fact that ${\rm Cap}_n := 
\{c \in \CH_n,\ c \hbox{ capitulates in $K$}\}$ is bounded independently of $n$ 
(\cite[\S\,5, Theorem 10]{Iw1973}, cited and used in \cite[Proposition 2]{Gree1976}), 
the theorem follows since this set coincides with $\CH_n$. 

\smallskip
A more precise result may be found in Grandet--Jaulent \cite[Theorem,\,(i,\,ii), 
p.\,214]{GJ1985}, assuming the triviality of the $\mu$ invariant, but valid for 
any $\Z_p$-extension, and showing that for $n$ large enough, $\CH_n 
\simeq \Big(\plus_{i=1}^\lambda \Z/p^{n+\alpha_i}\Z\Big) \oplus 
\Big(\plus_{i = \lambda + 1}^\kappa \Z/p^{\alpha_i}\Z\Big)$,
$\alpha_1, \ldots, \alpha_\lambda \in \Z$, $\alpha_{\lambda+1}, \ldots, 
\alpha_\kappa \geq 0$, where $\bigoplus_{i = \lambda + 1}^\kappa \Z/p^{\alpha_i}\Z 
\simeq {\rm Cap}_n$, these $\alpha_i$'s being independent of $n$ (see \cite{Gil1985} 
and many other papers for the study of $\mu$ in this context).
\end{proof}

It follows that in the conditions of Theorems \ref{capitule2}, \ref{capitule1}, the
Iwasawa invariants $\lambda_p(k^\acyc/k)$, $\mu_p(k^\acyc/k)$ are zero. 
On the contrary, Theorem \ref{lambdamu} has some consequence for $k^\cyc$:

\begin{corollary} \label{cyclotomic}
Assume $p \geq 3$ non-split in $k$; if $\CH_k \ne 1$, then
$\lambda_p(k^\cyc/k) \geq 1$ and $\mu_p(k^\cyc/k) = 0$.
\end{corollary}

\begin{proof}
Since $k^\cyc = k \Q^\cyc$, $k^\cyc/k$ is totally ramified at a unique 
prime ideal above $p$; moreover each ${\mathfrak P}_n$ is $p$-principal 
in $k_n^\cyc$, as extension of a principal ideal of $\Q_n^\cyc$; thus, 
$\BJ_0^n$ is injective for all $n \geq 0$ implying the claim.
\end{proof}

A short excerpt giving $\order \CH_n$ for $n = 0, 1, 2$ can be analyzed regarding 
the important tables \cite[Table I]{DFKS1991} giving $\lambda_p(k^\cyc/k)$ from 
$p$-adic ${\rm L}$-functions computations:

\ft\begin{verbatim}
m   n=0  n=1  n=2 lambda     m   n=0  n=1  n=2 lambda      m   n=0  n=1  n=2 lambda  
31  3    9    27     2       199 9    27   81     2        214 3    27   2187   4
186 3    27   243    2       211 3    27   243    2        286 3    27   729    ?
\end{verbatim}\ns 

The following property confirms once again the significance of capitulation
that may be the object of numerical computations in $k^\acyc$, at least for $n=1$:

\begin{proposition}\label{nocap}
Assume $p \geq 3$ non-split in $k$ and totally ramified in $k^\acyc/k$; then:

\smallskip
(i) For all $n \geq 1$, $\CH_{k_n^\acyc}^{G_n} = \BJ_{k_n^\acyc/k}(\CH_k) \plus 
\CH_{k_n^\acyc}^\ram$. 

\smallskip
(ii) The prime ideal ${\mathfrak P}_n \mid p$ in $k_n^\acyc$, $n \geq 1$, 
is non $p$-principal, if and only if there is at least a partial capitulation of 
$\CH_k$ in $k_n^\acyc$.

\smallskip
(iii) If $\BJ_{k_n^\acyc/k}$ is injective for all $n \geq 0$ (never capitulation), 
the primes ${\mathfrak P}_n$ are $p$-principal in $k_n^\acyc$ for all $n \geq 0$; 
this yields relations of the form $p = \alpha_n^{p^n}\!\! \cdot \eta_n$, $\alpha_n 
\in k_n^\acyc$, $\eta_n \in E_{k_n^\acyc}$, $\eta_n \notin E_{k_n^\acyc}^p$.
\end{proposition}

\begin{proof}
(i) From Remark \ref{Pstructure}, $\BJ_{k_n^\acyc/k}(\CH_k)$ is a minus
component while $\CH_{k_n^\acyc}^\ram = \langle \Ccl_n({\mathfrak P}_n) 
\rangle$ is a plus component, whence $\CH_{k_n^\acyc}^{G_n} = 
\BJ_{k_n^\acyc/k}(\CH_k) \plus \,\langle \Ccl_n({\mathfrak P}_n \rangle$.
Points (ii), (iii) are immediate consequences of the above relation since 
$\order \CH_n^{G_n} = \order \CH_k$. 
\end{proof}

Property (iii) (if any) seems pathological, in a $p$-adic viewpoint, in this 
non-cyclotomic context; see Theorem \ref{degreecap} and the Remarks
\ref{cyclocap} that follow it.

\smallskip
The case $n=1$ is very accessible to computation and states that, on the contrary, if 
${\mathfrak P}_1$ is not $p$-principal, $\BJ_1$ can not be injective, that goes in the 
good direction for a total capitulation in $k^\acyc$. Let's give two examples for $p=3$:

\smallskip
(i) Partial capitulation of order $3$ of $\CH_k \simeq \Z/9\Z$:

\ft\begin{verbatim}
m=29959 Disc=29959 kronecker(-m,3)=-1
Components of P_1: [117,18,0]
Algebraic norm of H_kacyc=[351,27,3] H_k=[117]
Norm of the component 1 of H_kacyc: [9,0,2]
Norm of the component 2 of H_kacyc: [9,0,2]
Norm of the component 3 of H_kacyc: [0,0,0]
\end{verbatim}\ns

(ii) Non capitulation of $\CH_k$ in $k_1^\acyc$ deduced from the 
$3$-principality of ${\mathfrak P}_1$ (Program 3):

\ft\begin{verbatim}
m=11217 Disc=44868 kronecker(-m,3)=0
Components of P_1: [15,0,0]
Algebraic norm of H_kacyc=[30,6,3] H_k=[30,2]
Norm of the component 1 of H_kacyc: [0,1,1]
Norm of the component 2 of H_kacyc: [0,0,0]
Norm of the component 3 of H_kacyc: [0,0,0]
\end{verbatim}\ns

The programming, in the split or non-split cases, of the calculations of 
$\Ccl_1({\mathfrak P}_1/{\mathfrak P}'_1)$ and $\Ccl_1({\mathfrak P}_1{\mathfrak P}'_1)$, 
when $k^\acyc/k$ is totally ramified, appears at the beginning of the sub-programs 
``{\sc computation of the partial capitulations}'' (Programs I, III, IV).

\section{Conclusion} 
This large study of the phenomena of capitulation in $k^\acyc/k$
strengthens the fact that the non-cyclotomic $\Z_p$-extensions of an 
imaginary quadratic field behave like the cyclotomic $\Z_p$-extension
of a {\it totally real} number field, in the sense that Iwasawa's invariants
may have minimal canonical values \cite[Conjecture 5.8]{Gra2024b}
with $\mu_p(k^\acyc/k) = 0$, $\lambda_p(k^\cyc/k) = 0$ in the totally real 
case, then $\lambda_p(k^\acyc/k) \in \{0, 1\}$ in the imaginary case, 
according to the decomposition of $p$ in $k$, because of unbounded 
groups $\CH_{k_n^\acyc}^{G_n}$ of invariant classes in $k_n^\acyc/k$ as 
$n$ tends to $\infty$ when $p$ splits in $k$ (see Proposition \ref{NCcap1}
and Remark \ref{Pstructure}). Even if $k^\acyc/k$ is the unique 
$\Z_p$-extension which can contain a non-ramified sub-extension, we 
can say that $p^{\lambda_p(k^\acyc/k)}$ is nothing but the order of magnitude 
given by the Chevalley--Herbrand formula in $k_n^\acyc/k_e^\acyc$ as $n \to \infty$.

\smallskip
We have extensively checked the role of non-totally real units in $k_n^\acyc$ 
(Theorem \ref{degreecap} and Remark \ref{cyclocap}), and some computations 
are done in \cite{Gra2024b} where we produce precise justifications. 

\smallskip
The case $H_k^\nr \subset k^\acyc$ of Theorem \ref{CHcyclic}, then 
Theorems \ref{capitule2}, \ref{capitule1}, Proposition \ref{nocap} and 
the criterion of \cite[Theorem 1.2\,(i)]{Gra2024a} with total ramification, 
show that there is a priori no obstructions for systematic capitulations. 

\smallskip
The question arises of replacing the $p$-class groups $\CH_{k_n^\acyc}$
by the quotients by the subgroups $\CH_{k_n^\acyc}^{G_n}$,
which absorb an unbounded part in the split case; 
in other words, considering $\CH_{k_n^\acyc}/\CH_{k_n^\acyc}^{G_n}$, 
leads to $\big(\CH_{k_n^\acyc}/\CH_{k_n^\acyc}^{G_n} \big)^{G_n} 
= \CH_{k_n^\acyc}^2/\CH_{k_n^\acyc}^1$ in terms of the canonical 
filtration (Definition \ref{filtre}), and we have proved the following
properties (without any hypothesis on the ramification in $k^\acyc/k$):

\begin{theorem}(\cite[Theorem 5.1 \& Appendix by Jaulent]{Gra2024b}).
Let $k=\Q(\sqrt {-m})$ in which $p \geq 3$ totally splits into 
${\mathfrak p}\,{\mathfrak p}'$, and let $\eta_k$ be the generating 
${\mathfrak p}$-unit (whence $(\eta_k) = {\mathfrak p}^h$, 
$h$ being the order of $\Ccl({\mathfrak p}))$. Then:

\smallskip
(i) For $n$ large enough, $\CH_{k_n^\acyc}^2/\CH_{k_n^\acyc}^1$ is 
isomorphic to the logarithmic class group $\wt \CH_k$.

\smallskip
(ii) For $n$ large enough, $\order \wt \CH_k = \order\,\big[
 (\CH_{k_n^\acyc}^{S_n})^{G_n} \big]$, where $S_n = S_{k_n^\acyc}$
(expression given in \cite[Th\'eor\`eme III.1.9, p.\,177]{Jau1986},
\cite[Corollary 3.9 or Theorem 3.11]{Gra2017}).

\smallskip
(iii) $\wt \CH_k = 1$ (giving $\CH_{k_n^\acyc} = \CH_{k_n^\acyc}^1$, 
hence $\lambda_p(k^\acyc/k)=1$ and $\mu_p(k^\acyc/k)=0$) if and only if 
${\bf v}_{{\mathfrak p}'}(\eta_k^{p-1} - 1) = 1$ and $\CH_k = 
\langle \Ccl({\mathfrak p}) \rangle$.
\end{theorem}

The forthcoming paper of Jaulent \cite{Jau2025} will show that for $p$ split 
in $k$, $\lambda_p(k^\acyc/k)=1$ and $\mu_p(k^\acyc/k)=0$ if and only if $\wt \CH_k$ 
capitulates in $k^\acyc$. A proof, analogous to that of Theorem \ref{lambdamu},
does exist with $S$-class groups and $S$-units.

\smallskip
It would be necessary to initiate a systematic study of the modules 
$\widehat \CH_{k_n^\acyc} := \CH_{k_n^\acyc}/\CH_{k_n^\acyc}^1$, hoping 
that the corresponding Iwasawa invariants have the canonical values described 
above, and that proving this will use analogous techniques as for the totally real case 
yielding the criterion of capitulation of the logarithmic class group \cite{Jau2019}; 

\smallskip
Nevertheless, whatever the notion of classes and units, proving capitulation 
is not trivial and probably goes beyond classical Iwasawa's theory.

\subsection*{Acknowledgments}
I warmly thank the Referee for his/her work and the suggestions of
corrections to be made in the paper.

I would like to thank D. Kundu and L.C. Washington for email exchanges 
and the communication of unpublished data \cite{KW2024b} giving  
polynomials of the first layer of $k^\acyc/k$ for $p=3$. This allowed me to 
check and improve my programs, with all possibilities regarding 
the decomposition of $3$ and the structures of the $3$-class groups.

All my friendship to Jean-Fran\c cois for the discussions we had 
around the properties of anti-cyclotomy and for having analyzed and 
disproved results of an article, on Iwasawa invariants.

\begin{appendices}

\section{Appendix -- Numerical illustrations of the programs}\label{AppA}

Recall that $\order \CW_k^\bp=3$ if and only if $m \equiv 3 \pmod 9$. 
Then $k_1^\acyc/k$ is unramified (non disjunction of $H_k^\nr$ with 
$k^\acyc$) if and only if $\order \CH_k \geq 3 \cdot \order \CT_k^\bp=
3 \cdot \order \CT_k \cdot (\order \CW_k^\bp)^{-1}$. 
Otherwise (when $k^\acyc/k$ is totally ramified), $\order \CH_k = \order
\CT_k^\bp = \order \CT_k \cdot (\order \CW_k^\bp)^{-1}$.

\smallskip
In the unramified case, $\BN_{k_1^\acyc/k}(\CH_{k_1^\acyc}) =
\CH'_k \simeq \Gal(H_k^\nr/k^\acyc \cap H_k^\nr)$.

\medskip
In a first part, we give the detailed data computed by the programs, then
only the structures of the class groups in order to study the capitulation.

\subsection{Examples from Program I  (Non Split Case)
\texorpdfstring{\eqref{program1}}{Lg}}
This allows any square-free integer $m \not\equiv 3 \!\pmod 9$:

\medskip
(i) $m=586$, from Program I (Non Split case):

\ft\begin{verbatim}
m=586 Disc=2344 kronecker(-m,3)=-1 h_k=18 H_k=[18] T_k=[9] h_kstar=6 
H_kstar=[6] #H_k=9 #T_k^bp=9 #W_k^bp=1  k_1^ac/k is Ramified
J=1 q=5 Qq irreducible
J=2 q=5 Qq irreducible
J=3 q=5 Qq irreducible
SOLUTION:J=4 w=Mod(24282790*x+1018152659,x^2-1758) 
Q^ac=x^3-99*x-630    H_kacyc=[90,15]    Val=4    a=1
W_k^bp=1 is not direct factor of T_k
Components of P_1: [30,5]
Algebraic norm of H_kacyc=[90,15] H_k=[18] 
Norm of the component 1 of H_kacyc: [3,0]
Norm of the component 2 of H_kacyc: [0,0]
PARTIAL CAPITULATION OF H_k
\end{verbatim}\ns

\medskip
(ii) $m=262$, from Program I (Non Split case):

\ft\begin{verbatim}
m=262 Disc=1048 kronecker(-m,3)=-1 h_k=6 H_k=[6] T_k=[3] h_kstar=6 
H_kstar=[6] #H_k=3 #T_k^bp=3 #W_k^bp=1  k_1^ac/k is Ramified
J=1 q=37 Qq irreducible
J=2 q=37 Qq irreducible
J=3 q=37 Qq irreducible
SOLUTION:J=4 w=Mod(1043744*x-29262089,x^2-786) 
Q^ac=x^3+42*x-40    H_kacyc=[54,9]    Val=3    a=3
W_k^bp=1 is not direct factor of T_k
Components of P_1: [0,0]
Algebraic norm of H_kacyc=[54,9] H_k=[6] 
Norm of the component 1 of H_kacyc: [18,0]
Norm of the component 2 of H_kacyc: [0,0]
NO CAPITULATION IN H_k
\end{verbatim}\ns

\medskip
(iii) $m=14935391$, from Program I (Non Split case).
Using a field $k^*$ with $\CH_{k^*}$ of $3$-rank $3$ (see, 
e.g., \cite[Appendice 2]{Diaz1974}), we obtain, among $40$ cubic 
fields to be tested, the following data in the unramified case 
(a huge precision is required):

\ft\begin{verbatim}
m=14935391 Disc=14935391 kronecker(-m,3)=1 h_k=3645 H_k=[405,3,3] T_k=[9,3,3]  
H_kstar=[3,3,3] #T_k^bp=81 #W_k^bp=1  k_1^ac/k is Unramified 
SOLUTION:J=11 w=Mod(3327909141/2*x-22276162963303/2,x^2-44806173) 
Q^ac=x^3-x^2+120*x-587    H_kacyc=[405,9,3,3,3]    Val=4    a=1
W_k^bp is not direct factor of T_k
Algebraic norm of H_kacyc=[405,9,3,3,3] H_k=[405,3,3] 
Norm of the component 1 of H_kacyc: [30,3,0,0,0]
Norm of the component 2 of H_kacyc: [27,3,0,0,0]
Norm of the component 3 of H_kacyc: [0,0,0,0,0]
Norm of the component 4 of H_kacyc: [54,6,0,0,0]
Norm of the component 5 of H_kacyc: [27,3,0,0,0]
PARTIAL CAPITULATION OF H'_k
\end{verbatim}\ns

\subsection{Examples from Program II (Normal Split case)
\texorpdfstring{\eqref{program2}}{Lg}}

Recall that $m \equiv 3\! \pmod 9$, $\CH_{k^*} \ne 1$ and
$\Ccl ({\mathfrak p}^*)\notin \CH_{k^*}^3$ (equivalent to  
$\rk(\CT_k) = \rk(\CH_{k^*})$). From Theorem \ref{disjunction}, 
$k_1^\acyc/k$ is always unramified and $\CW_k^\bp$ 
is direct factor of $\CT_k$ if and only if $\CR_{k^*}$ is not:

\smallskip
(i) $m=3513$, from Program II (Normal Split case):

\ft\begin{verbatim}
m=3513 Disc=14052 kronecker(-m,3)=0 h_k=36 H_k=[18,2] T_k=[3]  
H_kstar=[3] #T_k^bp=1 #W_k^bp=3  k_1^ac/k is Unramified 
SOLUTION:J=1 w=Mod(-2606*x+89177,x^2-1171) 
Q^ac=x^3-15*x-72    H_kacyc=[6,2]    Val=2    a=1
W_k^bp is direct factor of T_k
Algebraic norm of H_kacyc=[6,2] H_k=[18,2] 
Norm of the component 1 of H_kacyc: [0,0]
Norm of the component 2 of H_kacyc: [0,0]
TOTAL CAPITULATION OF H'_k
H_kstar=[3] T_kstar=[3] #R_kstar=1
R_kstar is not direct factor of T_kstar
\end{verbatim}\ns

\smallskip
(ii) $m=1400187$, from Program II (Normal Split case):

\ft\begin{verbatim}
m=1400187 Disc=1400187 kronecker(-m,3)=0 h_k=312 H_k=[312] T_k=[3]  
H_kstar=[9] #T_k^bp=1 #W_k^bp=3  k_1^ac/k is Unramified 
SOLUTION:J=2 w=Mod(67882475701274136939792147679778746902443577727038324439813*x
    -46375656013230959481690169204517935501631098135432942574419920,x^2-466729) 
Q^ac=x^3-x^2-107*x-2058    H_kacyc=[104]    Val=2    a=1
W_k^bp is direct factor of T_k
Algebraic norm of H_kacyc=[104] H_k=[312] 
Norm of the component 1 of H_kacyc: [0]
TOTAL CAPITULATION OF H'_k
H_kstar=[9] T_kstar=[27] #R_kstar=3
R_kstar is not direct factor of T_kstar
\end{verbatim}\ns

\smallskip
(iii) $m=1400790$, from Program II (Normal Split case):

\ft\begin{verbatim}
m=1400790 Disc=5603160 kronecker(-m,3)=0 h_k=1152 H_k=[144,2,2,2] T_k=[3]  
H_kstar=[6,2,2] #T_k^bp=1 #W_k^bp=3  k_1^ac/k is Unramified 
SOLUTION:J=4 w=Mod(355922351834329992826003357*x+243209845276745662417851320463,x^2-466930) 
Q^ac=x^3-357*x-2934    H_kacyc=[48,2,2,2]    Val=2    a=1
W_k^bp is direct factor of T_k
Algebraic norm of H_kacyc=[48,2,2,2] H_k=[144,2,2,2] 
Norm of the component 1 of H_kacyc: [0,0,0,0]
Norm of the component 2 of H_kacyc: [0,0,0,0]
Norm of the component 3 of H_kacyc: [0,0,0,0]
Norm of the component 4 of H_kacyc: [0,0,0,0]
TOTAL CAPITULATION OF H'_k
H_kstar=[6,2,2] T_kstar=[3] #R_kstar=1
R_kstar is not direct factor of T_kstar
\end{verbatim}\ns

\subsection{Examples from Program III (Special Split case)
\texorpdfstring{\eqref{program3}}{Lg}} 
Here $m \equiv 3 \pmod 9$, with $\CH_{k^*} \ne 1$ 
and $\Ccl ({\mathfrak p}^*) \in \CH_{k^*}^3$ (equivalent to  
$\rk(\CT_k) = \rk(\CH_{k^*})+1$):

\smallskip
(i) $m=8139$, from Program III (Special Split case):

\ft\begin{verbatim}
m=8139 Disc=8139 kronecker(-m,3)=0 h_k=36 H_k=[36] T_k=[9,3]  
H_kstar=[3] #T_k^bp=9 #W_k^bp=3  k_1^ac/k is Ramified 
SOLUTION:J=10 w=Mod(177*x+9390,x^2-2713) 
Q^ac=x^3+108*x-651    H_kacyc=[36,6,2]    Val=4    a=1
W_k^bp is direct factor of T_k 
Components of P_1: [18,2,0]
Algebraic norm of H_kacyc=[36,6,2] H_k=[36] 
Norm of the component 1 of H_kacyc: [3,0,0]
Norm of the component 2 of H_kacyc: [0,0,0]
Norm of the component 3 of H_kacyc: [0,0,0]
PARTIAL CAPITULATION OF H_k
H_kstar=[3] T_kstar=[3] #R_kstar=1
R_kstar is not direct factor of T_kstar
\end{verbatim}\ns

\smallskip
(ii) $m=8913$, from Program III (Special Split case):

\ft\begin{verbatim}
m=8913 Disc=35652 kronecker(-m,3)=0 h_k=36 H_k=[18,2] T_k=[3,3]  
H_kstar=[3] #T_k^bp=3 #W_k^bp=3  k_1^ac/k is Unramified 
SOLUTION:J=3 w=Mod(-16536757906485*x-901367083753012,x^2-2971) 
Q^ac=x^3-x^2+8*x-38    H_kacyc=[6,2,2,2]    Val=3    a=1
W_k^bp is direct factor of T_k
Algebraic norm of H_kacyc=[6,2,2,2] H_k=[18,2] 
Norm of the component 1 of H_kacyc: [0,0,0,0]
Norm of the component 2 of H_kacyc: [0,0,0,0]
Norm of the component 3 of H_kacyc: [0,0,0,0]
Norm of the component 4 of H_kacyc: [0,0,0,0]
TOTAL CAPITULATION OF H'_k
H_kstar=[3] T_kstar=[3] #R_kstar=1
R_kstar is not direct factor of T_kstar
\end{verbatim}\ns

\smallskip
(iii) $m=46983$, from Program III (Special Split case):

\ft\begin{verbatim}
m=46983 Disc=46983 kronecker(-m,3)=0 h_k=168 H_k=[168] T_k=[3,3]  
H_kstar=[3] #T_k^bp=3 #W_k^bp=3  k_1^ac/k is Ramified 
SOLUTION:J=9 w=Mod(429566157772765447440605727/2*x-53757591659366243794048388487/2,
    x^2-15661) 
Q^ac=x^3-1161*x-15231    H_kacyc=[168,3,3]    Val=3    a=3
W_k^bp is direct factor of T_k 
Components of P_1: [84,0,0]
Algebraic norm of H_kacyc=[168,3,3] H_k=[168] 
Norm of the component 1 of H_kacyc: [0,0,0]
Norm of the component 2 of H_kacyc: [1,0,1]
Norm of the component 3 of H_kacyc: [0,0,0]
NO CAPITULATION IN H_k
H_kstar=[3] T_kstar=[3] #R_kstar=1
R_kstar is not direct factor of T_kstar
\end{verbatim}\ns

\smallskip
(iv) $m=245307$, from Program III (Special Split case): 

\ft\begin{verbatim}
m=245307 Disc=245307 kronecker(-m,3)=0 h_k=108 H_k=[108] T_k=[27,3]
H_kstar=[9] #T_k^bp=27 #W_k^bp=3  k_1^ac/k is Ramified
SOLUTION:J=11 w=Mod(-67602622739952754931055646822194594178831160058*x+
19331158807359840568708838261222200689708647399667,x^2-81769) 
Q^ac=x^3-1035*x-14538    H_kacyc=[108,3]    Val=5    a=1
W_k^bp is direct factor of T_k
Components of P_1: [90,2]
Algebraic norm of H_kacyc=[108,3] H_k=[108]
Norm of the component 1 of H_kacyc: [3,0]
Norm of the component 2 of H_kacyc: [0,0]
PARTIAL CAPITULATION OF H_k
H_kstar=[9] T_kstar=[9] #R_kstar=1
R_kstar is not direct factor of T_kstar
\end{verbatim}\ns

In the part ${\sf COMPUTATION\ OF\ w_p^*}$ of Program III,
the intermediate data are:

\ft\begin{verbatim}
Z=[[7,6;0,1]] Y=[0] X=[3,0;0,1] c=[[0]~,[788309253,5494348]]
Gamma =Mod(2747174*x+785562079, x^2-81769)
wstar=Mod(-8241522*x-2356686237,x^2-81769)
\end{verbatim}\ns

The radical solution, giving $k_1^\acyc$, is 
$w=w_{{\mathfrak p}^*}^* \times \varepsilon^{* 2}$ and ${\mathfrak p}^*$ 
is $3$-principal since we compute ${\sf Y=bnfisprincipal(kstar,Sideal)[1]=[0]}$.

\subsection{Examples from Program IV (Trivial case)
\texorpdfstring{\eqref{program4}}{Lg}}

This is the case $\CH_{k^*} = 1$ (whatever the decomposition of 
$3$ in $k^*$). Program IV gives many zero matrices.

\smallskip
We note that for the smallest values of $m$, for which $\CH_k= \CT_k=1$
(e.g., $m \in \{$$1$, $2$, $3$, $5$, $6$, $7$, $10$, $11$, $13$, $14$, $15$, 
$17$, $19$$\}$), the defining 
polynomials of $k_1^\acyc$ coincide with those of the table of \cite[Section IV, 
$p=3$]{Br2007} (the two cases $m=13$ and $m=15$ coincide after applying the
instruction ${\sf polredbest}$); all other numerical examples give same results.

\smallskip
(i) $m=87$, from Program IV (Trivial case):

\ft\begin{verbatim}
m=87 Disc=87 kronecker(-m,3)=0 h_k=6 H_k=[6] T_k=[]  
#T_k^bp=1 #W_k^bp=1  k_1^ac/k is Unramified 
SOLUTION:J=1 w=Mod(1/2*x-5/2,x^2-29) 
Q^ac=x^3-x^2-2*x+3    H_kacyc=[2]    Val=1    a=1
W_k^bp is not direct factor of T_k
Algebraic norm of H_kacyc=[2] H_k=[6] 
Norm of the component 1 of H_kacyc: [0]
TOTAL CAPITULATION OF H_k
\end{verbatim}\ns

\smallskip
(ii) $m=714$, from Program IV (Trivial case):

\ft\begin{verbatim}
m=714 Disc=2856 kronecker(-m,3)=0 h_k=24 H_k=[6,2,2] T_k=[9]  
#T_k^bp=3 #W_k^bp=3  k_1^ac/k is Ramified 
SOLUTION:J=2 w=Mod(18*x+279,x^2-238) 
Q^ac=x^3-27*x-558    H_kacyc=[18,18,2]    Val=3    a=1
W_k^bp is not direct factor of T_k
Components of P_1: [12,9,0]
Algebraic norm of H_kacyc=[18,18,2] H_k=[6,2,2] 
Norm of the component 1 of H_kacyc: [0,0,0]
Norm of the component 2 of H_kacyc: [0,0,0]
Norm of the component 3 of H_kacyc: [0,0,0]
TOTAL CAPITULATION OF H_k
\end{verbatim}\ns

\smallskip
(iii) $m=2181$, from Program IV (Trivial case):

\ft\begin{verbatim}
m=2181 Disc=8724 kronecker(-m,3)=0 h_k=36 H_k=[18,2] T_k=[27]  
#T_k^bp=9 #W_k^bp=3  k_1^ac/k is Ramified 
SOLUTION:J=2 w=Mod(-243*x+5346,x^2-727) 
Q^ac=x^3+81*x-396    H_kacyc=[18,6,2,2]    Val=4    a=1
W_k^bp is not direct factor of T_k
Components of P_1: [12,1,0,1]
Algebraic norm of H_kacyc=[18,6,2,2] H_k=[18,2] 
Norm of the component 1 of H_kacyc: [3,0,0,0]
Norm of the component 2 of H_kacyc: [0,0,0,0]
Norm of the component 3 of H_kacyc: [0,0,0,0]
Norm of the component 4 of H_kacyc: [0,0,0,0]
PARTIAL CAPITULATION OF H_k
\end{verbatim}\ns

\subsection{Other examples regarding Capitulation}

From Theorem \ref{degreecap}, a complete capitulation of $\CH_k$ 
in $k^\acyc$ cannot take place before the layer $k_{\hbox{\tiny{\sc exp}}(k)}^\acyc$.

\medskip
{\bf a)} From Program I (Non Split case).
One may use for instance Program I with its given list.
We comment some examples; see the Algorithm \ref{abcd} for the 
details.
 
\medskip
(i) $m=28477$ (Unramified):

\ft\begin{verbatim}
Q^ac=x^3-47*x-140
Algebraic norm of H_kacyc=[36,12,6] H_k=[18,3] 
Norm of the component 1 of H_kacyc: [3,0,0]
Norm of the component 2 of H_kacyc: [3,0,0]
Norm of the component 3 of H_kacyc: [0,0,0]
PARTIAL CAPITULATION OF H'_k
\end{verbatim}\ns

\medskip
(ii) $m=32573$ (Ramified):

\ft\begin{verbatim}
Q^ac=x^3+402*x-5428
Components of Cl(P1/P'1) and Cl(P1*P'1): [88,2,1,1] [0,2,2,2]
Algebraic norm of H_kacyc=[126,3,3,3] H_k=[42,3] 
Norm of the component 1 of H_kacyc: [3,0,0,0]
Norm of the component 2 of H_kacyc: [0,0,0,0]
Norm of the component 3 of H_kacyc: [0,0,0,0]
Norm of the component 4 of H_kacyc: [0,0,0,0]
PARTIAL CAPITULATION OF H_k
\end{verbatim}\ns

\medskip
(iii) $m=58213$ (Ramified):

\ft\begin{verbatim}
Q^ac=x^3-1110*x-14332
Components of P_1: [0,0,0,0,0]
Algebraic norm of H_kacyc=[180,30,3,3,3] H_k=[12,6] 
Norm of the component 1 of H_kacyc: [3,0,1,2,2]
Norm of the component 2 of H_kacyc: [6,0,1,2,2]
Norm of the component 3 of H_kacyc: [0,0,0,0,0]
Norm of the component 4 of H_kacyc: [0,0,0,0,0]
Norm of the component 5 of H_kacyc: [0,0,0,0,0]
NO CAPITULATION IN H_k
\end{verbatim}\ns

\medskip
(iv) $m=111749$ (Ramified):

\ft\begin{verbatim}
Q^ac=x^3-687*x-23068
Components of Cl(P1/P'1) and Cl(P1*P'1): [134, 2, 6, 1] [0,12,0,2]
Algebraic norm of H_kacyc=[216,18,9,3] H_k=[72,6] 
Norm of the component 1 of H_kacyc: [6,6,0,0]
Norm of the component 2 of H_kacyc: [15,6,0,0]
Norm of the component 3 of H_kacyc: [15,6,0,0]
Norm of the component 4 of H_kacyc: [18,0,0,0]
PARTIAL CAPITULATION OF H_k
\end{verbatim}\ns

Indeed, the image of $\CH_{k^\acyc}$ by the algebraic norm is
the cyclic group of order $9$ generated by $h_1^3 \,h_2^3$.

\medskip
(v) $m=78730$ (Ramified): 

\ft\begin{verbatim}
Q^ac=x^3-2658*x-54472
Components of P_1: [36,2,0,1]
Algebraic norm of H_kacyc=[54,6,3,3] H_k=[18,6] 
Norm of the component 1 of H_kacyc: [3,0,0,0]
Norm of the component 2 of H_kacyc: [18,0,0,0]
Norm of the component 3 of H_kacyc: [9,0,0,0]
Norm of the component 4 of H_kacyc: [18,0,0,0]
PARTIAL CAPITULATION OF H_k
\end{verbatim}\ns

\medskip
(vi) $m=113213$ (Ramified): 

\ft\begin{verbatim}
Q^ac=x^3-5142*x-141940
Components of Cl(P1/P'1) and Cl(P1*P'1): [114,2,4,0,0,0] [0,2,0,0,0,1]
Algebraic norm of H_kacyc=[198,6,6,3,3,3]  H_k=[66,3] 
Norm of the component 1 of H_kacyc: [0,1,2,0,0,0]
Norm of the component 2 of H_kacyc: [0,0,0,0,0,0]
Norm of the component 3 of H_kacyc: [0,0,0,0,0,0]
Norm of the component 4 of H_kacyc: [0,0,0,0,0,0]
Norm of the component 5 of H_kacyc: [3,0,0,0,0,0]
Norm of the component 6 of H_kacyc: [0,0,0,0,0,0]
NO CAPITULATION IN H_k
\end{verbatim}\ns

\medskip
{\bf b)} From Program II (Normal Split case).

\smallskip
(i) $m=128451$ (Unramified):

\ft\begin{verbatim}
Q^ac=x^3-x^2+76*x+84
Algebraic norm of H_kacyc=[36,6] H_k=[12,6] 
Norm of the component 1 of H_kacyc: [0,0]
Norm of the component 2 of H_kacyc: [0,0]
TOTAL CAPITULATION OF H'_k
\end{verbatim}\ns

Since $\CH_{k_1^\acyc}^{G_1} = \BJ_{k_1^\acyc/k}
(\CH_k) \oplus \CH_{k_1^\acyc}^\ram$, it follows that
$\BJ_{k_1^\acyc/k}(\CH_k) \simeq \Z/3\Z$.

\smallskip
(ii) $m = 129189$ (Unramified):

\ft\begin{verbatim}
Q^ac=x^3-x^2+144*x+150
Algebraic norm of H_kacyc=[42,6,6,2] H_k=[42,6] 
Norm of the component 1 of H_kacyc: [0,1,1,0]
Norm of the component 2 of H_kacyc: [0,0,0,0]
Norm of the component 3 of H_kacyc: [0,0,0,0]
Norm of the component 4 of H_kacyc: [0,0,0,0]
NO CAPITULATION IN H'_k
\end{verbatim}\ns

\medskip
{\bf c)} From Program III (Special Split case).

\smallskip
(i) $m=11217$ (Ramified):

\ft\begin{verbatim}
Q^ac=x^3-162*x-2340
Components of P_1: [15,0,0]
Algebraic norm of H_kacyc=[30,6,3] H_k=[30,2] 
Norm of the component 1 of H_kacyc: [0,1,1]
Norm of the component 2 of H_kacyc: [0,0,0]
Norm of the component 3 of H_kacyc: [0,0,0]
NO CAPITULATION IN H_k
\end{verbatim}\ns

\smallskip
(ii) $m=17814$ (Ramified):

\ft\begin{verbatim}
Q^ac=x^3+27*x-1386
Components of P_1: [90,25,1,0]
Algebraic norm of H_kacyc=[270,30,2,2] H_k=[54,2] 
Norm of the component 1 of H_kacyc: [3,0,0,0]
Norm of the component 2 of H_kacyc: [0,0,0,0]
Norm of the component 3 of H_kacyc: [0,0,0,0]
Norm of the component 4 of H_kacyc: [0,0,0,0]
PARTIAL CAPITULATION OF H_k
\end{verbatim}\ns

\medskip
{\bf d)} From Program IV (Trivial case).

There are many capitulations in the Trivial case since
the class groups have small orders from the assumption $\CH_{k^*}=1$:

\smallskip
(i)  $m=201$ (Ramified):

\ft\begin{verbatim}
Q^ac=x^3+54*x-252    H_kacyc=[6,6]    Val=3    a=1
W_k^bp is not direct factor of T_k
Components of P_1: [2,3]
Algebraic norm of H_kacyc=[6,6] H_k=[6,2] 
Norm of the component 1 of H_kacyc: [0,0]
Norm of the component 2 of H_kacyc: [0,0]
TOTAL CAPITULATION OF H_k
\end{verbatim}\ns

\smallskip
(ii) $m=3882$ (Ramified):

\ft\begin{verbatim}
Q^ac=x^3-81*x-354    H_kacyc=[12,6,3]    Val=3    a=3
W_k^bp is not direct factor of T_k
Components of P_1: [6,0,0]
Algebraic norm of H_kacyc=[12,6,3] H_k=[12,2] 
Norm of the component 1 of H_kacyc: [0,1,1]
Norm of the component 2 of H_kacyc: [0,0,0]
Norm of the component 3 of H_kacyc: [0,0,0]
NO CAPITULATION IN H_k
\end{verbatim}\ns

Including, for checking, the values of $m$, given by Theorems 
\ref{capitule1} and \ref{capitule2}, always give the zero matrices 
as expected.

\section{Appendix -- An example of unit 
\texorpdfstring{$\varepsilon_n^\capitul$}{Lg}}\label{capunit}
We intend to illustrate Theorem \ref{degreecap} with the 
field $k=\Q(\sqrt{-298})$, $K=k^\acyc$, and $n=1$. Due to
the {\sc pari/gp} version considered, the reader may obtain other
numerical values, other systems of classes, units, etc., but same 
results. From Program I one obtains the data:

\smallskip
\ft\begin{verbatim}
m=298 Disc=1192 kronecker(-m,3)=-1  
H_kstar=[6] k_1^ac/k is Ramified 
SOLUTION:J=1 w=Mod(670*x-20153,x^2-894) 
Q^ac=x^3-99*x-522    H_kacyc=[6,6,2]    Val=3    a=1
Components of P_1: [2,0,0]
Algebraic norm of H_kacyc=[6,6,2] H_k=[6] T_k=[3] #T_k^bp=3
Norm of the component 1 of H_kacyc: [0,0,0]
Norm of the component 2 of H_kacyc: [0,0,0]
Norm of the component 3 of H_kacyc: [0,0,0]
H_kstar=[6] T_kstar=[3] #W_kstar^bp=1 #R_kstar=1
TOTAL CAPITULATION OF H_k
\end{verbatim}\ns

\smallskip
The field $k_1^\acyc$ may be defined by the polynomial $R^\acyc$
(${\sf R^{ac}=x^6-66*x^4+1089*x^2+4768}$) from {\sf polredbest}
(an example where generators of $\BH_k$ may vary for the same field).
The class group of $k$, of order $6$, is generated by $\Ccl({\mathfrak l})$, 
${\mathfrak l} \mid 13$. The two ideals ${\mathfrak l}_1, {\mathfrak l}_2
\mid 13$ of $k$ are inert in $k_1^\acyc$. Taking the square of the ideals 
${\mathfrak L}_1, {\mathfrak L}_2 \mid 13$ in $k_1^\acyc$, and testing their 
principality gives:

\ft\begin{verbatim}
idealfactor(kacyc,13)=
L1=[13,[-4,39,0,0,0,4]~,1,3,[4,0,0,0,0,-1192;0,-32,-28,-120,0,0;
    0,28,40,0,120,0;0,44,0,40,-28,0;0,0,-44,28,-32,0;4,0,0,0,0,4]]
L2=[13,[4,39,0,0,0,4]~,1,3,[-4,0,0,0,0,-1192;0,-40,-28,-120,0,0;
    0,28,32,0,120,0;0,44,0,32,-28,0;0,0,-44,28,-40,0;4,0,0,0,0,-4]]

L1^2=idealpow(kacyc,L1,2);
bnfisprincipal(kacyc,L1^2)=[[0,0,0]~,[-11515,998,3398,1060,-2318,-848]~]

L2^2=idealpow(kacyc,L2,2)
bnfisprincipal(kacyc,L2^2)=[[0,0,0]~,[-11515,-5456,1080,3378,2318,848]~]
\end{verbatim}\ns

\smallskip
Thus, the squares of ${\mathfrak L}_1$ and ${\mathfrak L}_2$ are principal 
in $k_1^\acyc$; then, let $\{z_1,z_2,z_3,z_4,z_5,z_6\}$ be an integral basis 
of $k_1^\acyc$ given by the {\sc pari/gp} instruction ${\sf kacyc.zk}$:

\ft\begin{verbatim}
[1,x,1/116*x^4-55/116*x^2+1/2*x+121/29,1/58*x^4-13/29*x^2+1/2*x-77/29,
1/232*x^5-1/232*x^4-55/232*x^3-3/232*x^2+445/116*x+99/29,1/4*x^3-33/4*x].
\end{verbatim}\ns

\ft\begin{verbatim}
z1=Mod(1,R^ac);z2=Mod(x,R^ac);z3=Mod(1/116*x^4-55/116*x^2+1/2*x+121/29,R^ac);
z4=Mod(1/58*x^4-13/29*x^2+1/2*x-77/29,R^ac);
z5=Mod(1/232*x^5-1/232*x^4-55/232*x^3-3/232*x^2+445/116*x+99/29,R^ac);
z6=Mod(1/4*x^3-33/4*x,R^ac);
\end{verbatim}\ns

\noindent
Then, ${\mathfrak L}_1^2 = (\alpha_1)$
and ${\mathfrak L}_2^2 =: (\alpha_2)$ where:
\begin{equation*}
\left\{\begin{aligned}
\alpha_1 & = -11515\,z_1+ 998\,z_2+ 3398\,z_3+ 1060\,z_4 - 2318\,z_5 -848\,z_6, \\
 & = -\ffrac{1159}{116}\,x^5+\ffrac{6677}{116}\,x^4+
\ffrac{39153}{116}\,x^3-\ffrac{238533}{116}\,x^2 
+ \ffrac{77179}{58}\,x-\ffrac{233879}{29}, \\
\alpha_2 & = -11515\,z_1 -5456\,z_2+ 1080\,z_3+ 3378\,z_4+ 2318\,z_5+ 848\,z_6. \\
& =\ffrac{1159}{116}\,x^5+\ffrac{6677}{116}\,x^4 -
\ffrac{39153}{116}\,x^3-\ffrac{238533}{116}\,x^2 
- \ffrac{77179}{58}\,x-\ffrac{233879}{29}.
\end{aligned}\right.
\end{equation*}

In particular, $\alpha_1$ is an {\it integer} whose norm is, quite spectacularly, 
given by $\BN_{k/\Q}(\alpha_1)=4826809=13^6$ ($\alpha_2$ is the
conjugate of $\alpha_1$).

\smallskip
In $k$, the class of ${\mathfrak a} := {\mathfrak l}_1^2$ is a generator of 
order $3$ of $\CH_k$ and becomes principal in $k_1^\acyc$; we verify that
${\mathfrak a}^3 = (1899 \pm 64\,\sqrt{-298}) =: (\alpha_0^{})$ (of norm $13^6$).
So that, $\varepsilon_1^\capitul \in k_1^\acyc$ is given by:
\begin{equation}\label{E}
\left\{\begin{aligned}
\varepsilon_1^\capitul =
\frac{\alpha_0^{}}{\alpha_1^3} & = 
\frac{1899 \pm 64\,\sqrt{-298}}
{\big (-11515\,z_1+ 998\,z_2+ 3398\,z_3+ 1060\,z_4 -2318\,z_5 -848\,z_6 \big)^3} \\
& = \frac{1899 \pm 64\,\sqrt{-298}}{\big (-\frac{1159}{116}\,x^5+\frac{6677}{116}\,x^4+
\frac{39153}{116}\,x^3-\frac{238533}{116}\,x^2 
+ \frac{77179}{58}\,x-\frac{233879}{29} \big)^3};
\end{aligned}\right.
\end{equation}

\noindent
indeed, for a suitable sign in $1899 \pm 64\,\sqrt{-298}$, the ideals generated
by the numerator and the denominator are equal to ${\mathfrak L}_1^6$.

\smallskip
The fundamental units $\varepsilon, \varepsilon'$, given by the instruction 
${\sf kacyc.fu}$ are:
\ft\begin{verbatim}
E=Mod(225/116*x^5-343/116*x^4-12375/116*x^3+5815/116*x^2+
84291/58*x+78267/29,R^ac),
Es=Mod(225/116*x^5+343/116*x^4-12375/116*x^3-5815/116*x^2+
84291/58*x-78267/29,R^ac).
\end{verbatim}\ns

Using ${\sf nfgaloisapply(kacyc,s,E)}$, where 
${\sf s =  -1/116*x^4+55/116*x^2-1/2*x-121/29}$ denotes a generator of
$\Gal(k_1^\acyc/k)$, one finds that $\varepsilon'=-\varepsilon^s$.
It is easy to verify that $\varepsilon_1^\capitul =\varepsilon^{1+2s}$, 
using an embedding $x \to \rho$, where $\rho$ is a root of $R^\acyc$ 
and the precision arbitrary large: 

\ft\begin{verbatim}
\p 200
{R^ac=x^6-66*x^4+1089*x^2+4768;rho=polroots(R^ac)[1];print("rho=",rho);
E=225/116*rho^5-343/116*rho^4-12375/116*rho^3+5815/116*rho^2+84291/58*rho+78267/29;
Es=225/116*rho^5+343/116*rho^4-12375/116*rho^3-5815/116*rho^2+84291/58*rho-78267/29;
epsilon=(1899-64*sqrt(-298))/(-1159/116*rho^5+6677/116*rho^4+39153/116*rho^3-
238533/116*rho^2+77179/58*rho-233879/29)^3;
e=E*Es^2;print("e=",e);print("epsilon=",epsilon)}

rho=-5.9728134065...4562843417-0.9441927751...0618111082*I
e=
7.270516706789487836656178009435402414333453130182776091448039916032628461513661700
933698392988290285146819638483366394260981671112379267906499128649552172074837697810
4617926445495125962790428943177430 E-12
-1.0277841502329176112109162259015562800933793307507503826689008917123211914048449964
8650585897853522364539655813108248817044645543907234059344713003099779746659605382060
84979908985917818027637154804682 E-11*I
epsilon=
7.270516706789487836656178009435402414333453130182776091448039916032628461513661700
933698392988290285146819638483366394260981671112379267906499128649552172074837697810
4617926445495125962790428943177429 E-12
-1.0277841502329176112109162259015562800933793307507503826689008917123211914048449964
8650585897853522364539655813108248817044645543907234059344713003099779746659605382060
84979908985917818027637154804681 E-11*I
\end{verbatim}\ns

The two numbers differ at the last digit of the real and imaginary parts; 
so, it is clear that the unit, given by the formula \eqref{E}, is the unit 
$\varepsilon_1^\capitul$, and the relations of conjugation illustrate Theorem \ref{degreecap}.

\section{Appendix -- Algebraic complements}
Our main tool to give algebraic results specifying the arithmetic
framework and also to find again some results of the literature, often with
simplification of the proofs (e.g., Theorem \ref{capitule2} and 
\cite[Theorem 5.1]{Gra2024b}), is the following one that may 
be considered as a higer Chevalley--Herbrand formula:

\subsection{Filtrations of the class groups}\label{filtrations}
We know that the original Chevalley--Herbrand formula, given in 
\cite[pp.\,402-406]{Che1933} and more generally the properties of 
the filtration attached to the $p$-class group of a cyclic $p$-extension, 
give interesting information which depends a lot on the decomposition 
of $p$ in $k$; in the following definition, base fields and extensions are arbitrary:

\begin{definition}\label{filtre}
Let $K/k$ be a cyclic $p$-extension or a $\Z_p$-extension, let $G_n$
be the Galois group of the layer $K_n$ of degree $p^n$, $n \geq 0$,
over $k$. The canonical filtration of $\CH_{K_n}$ is defined by
$\CH_{K_n}^0 :=1$, $\CH_{K_n}^1 := \CH_{K_n}^{G_n}$, and
$\CH_{K_n}^{i+1}/\CH_{K_n}^i := \big(\CH_{K_n}/\CH_{K_n}^i\big)^{G_n}$, 
$i \geq 0$, whose order is:
$$\order(\CH_{K_n}^{i+1}/\CH_{K_n}^i) = \frac{\order \BN_{{K_n}/k}(\order \CH_{K_n})}
{\order \BN_{{K_n}/k}(\CH_{K_n}^i)} \times \frac{\order \Omega_{K_n/k}}
{\order \omega_{K_n/k} (\Lbda_k^i)}, \ \, \Lbda_k^i := 
\{a \in k^\times, \, (a) \in \BN_{{K_n}/k}(\CI_{K_n}^i)\}, $$
where $\Ccl_n(\CI_{K_n}^i) = \CH_{K_n}^i$ for a module $\CI_{K_n}^i$ 
of prime-to-$p$ ideals of $K_n$; the second factor is defined
in terms of genus theory as follows: let $G_{n, v}$ be 
the inertia groups of the places $v$ ramified in $K_n/k$, and:
$\Omega_{K_n/k} := \big\{(s_v^{}) \in 
\prd{}_v \, G_{n,v},\  \prd{}_v \,s_v^{} = 1 \big\}$; 
we consider the map $\omega_{K_n/k} : k^\times \to 
\Omega_{K_n/k}$, defined on the set of elements $a \in k^\times$, 
assumed to be local norm in $K_n/k$ for all unramified place, by
$\omega_{K_n/k}(a) = \big( \big(\frac{a\,,\,{K_n}/k}{v} \big)
\big)_{\! v} \in \Omega_{K_n/k}$, since these symbols fulfill 
the product formula of class field theory.
If $K/k$ is a $\Z_p$-extension, the map $\omega_{K_n/k}$ is defined 
by $\omega_{K_n/k}(a) = \big( \big(\frac{a\,,\,{K_n}/k}{{\mathfrak p}} \big)
\big)_{\!{\mathfrak p} \mid p}$.
\end{definition}

Typically, to prove that $\lambda_p(K/k) 
= \mu_p(K/k) = 0$, one tries to show that $\order \CH_{K_n}$ is bounded
regarding $n$. To prove that $\lambda_p(K/k)$ may have a standard value
(e.g., $\lambda_p(k^\acyc/k)=1$ in the split case of $p$ in $k=\Q(\sqrt {-m})$),
one hopes that this comes from the subgroup of invariant classes.
An abundant literature shows that most of the techniques used are equivalent 
to purely algebraic considerations on the previous formulas, adding strong 
assumptions. We do not intend to describe them, which can be found, among 
others, in \cite{Gree1976,Ger1977,San1993,HW2018,Jau2019,Gra2022,
KW2024a} and especially in their bibliographies. 

In general, context and assumptions yield trivial norm factors 
$\ds \frac{\order\Omega_{K_n/k}}{\order\omega_{K_n/k} (\Lbda_k^i)}$;
it is the case where one assumes that $p$ does not split in $k$, which
gives $\Omega_{K_n/k} = 1$. This removes most of the arithmetic part 
containing random $p$-adic aspects that we have analyzed in \cite{Gra2021}, 
it simplifies some capitulation phenomena (Theorem \ref{lambdamu}, 
\cite{Jau2019}), etc.

\subsection{Group structure of \texorpdfstring{$\CH_{K}$}{Lg} in a
degree-\texorpdfstring{$p$}{Lg} cyclic extension \texorpdfstring{$K/k$}{Lg} }

Kundu--Washington, in \cite[Theorem 2.1]{KW2024a}, determine 
the possible structures of $p$-class groups, in a more 
general setting that for the first layer of a $\Z_p$-extension of~$k$; but
to simplify we restrict ourselves to the case we are studying:

\medskip\noindent
{\bf Theorem} (Kundu--Washington (2024)).
Assume $p \ne 2$ non-split in the imaginary quadratic field $k$, and let 
$K$ be a totally ramified $\Z_p$-extension of $k$; let $K_1$ be the first 
layer of $K/k$. We assume that $\CH_k$ is cyclic of order $p^e$, 
$e \geq 1$, without partial capitulation in $K_1$ (i.e., $\BJ_{K_1/k}$ injective).
Under these assumptions, $\CH_{K_1}$ is one of the following groups:
\begin{equation*}
\left\{ \begin{aligned}
(\Z/p^e\Z) ^p,\ & \\
(\Z/p^{e-1}\Z) &\!\times\! (\Z/p^{s+1}\Z)^{a} \!\times\! (\Z/p^s\Z)^{p-1-a},\,
1 \leq e \le s ,\ 1\le a\le p-1, \\
(\Z/p^{e+1}\Z) & \!\times\!  (\Z/p^{s+1}\Z)^{b} \!\times\!  (\Z/p^s\Z)^{p-1-b}, \, 
0 \le s < e,\  0\le b\le p-2 \  [b\ne p-2 \text{ if } e=s+1].
\end{aligned}\right.
\end{equation*}

We had proven similar algebraic results, in the case $\CH_k$ cyclic 
of order $p$ (i.e., $e=1$),  without some assumptions of the 
above theorem \cite[Theorems 4.1, 4.3]{Gra2017}. In particular, 
no hypothesis is done on the base field $k$, nor in the splitting of 
$p$ in $k$. However, we assume the strong hypothesis that the 
group of invariant classes in $K_1/k$ is of order $p$;
since Chevalley--Herbrand formula is $\order \CH_{K_1}^{G_1} 
= \order \CH_k \times {\mathcal N}$, where the ``norm factor'' 
${\mathcal N}$ is an integer if $K/k$ is totally ramified, 
this is equivalent to have a trivial norm factor since $\order \CH_k = p$; 
this implies some local norm conditions on units of $k$, which 
applies, for usual class groups, only for suitable base fields when
$p$ does not split. 
But one may replace the $p$-class groups by $S$-class groups for any 
set of places $S$ for which the filtrations (Definition \ref{filtre}) follow 
analogous formulas with $S$-units instead of units; many generalizations 
are available (see \cite[Theorem 3.6 \& Corollaries]{Gra2017}).

\smallskip
When there is only partial capitulation (or non-capitulation as
in \cite{KW2024a}), we had obtained the following general result, for the  
filtration $(\CM_{K_1}^i)_i$ of finite $\Z_p[G_1]$-modules $\CM_{K_1}$
constituting an ``arithmetic family'' in the meaning of \cite{Gra2023}:

\begin{theorem}\label{classification}
Let $K/k$ be a cyclic extension of degree $p$, where $k$ is
arbitrary as well as the decomposition of $p$ in $k/\Q$.
Assume the condition $\order \CM_K^1 = p$. 
Let ${\rm maxi}$ be the least integer $i$ of the filtration
such that $\CM_K^i=\CM_K$. 

\medskip
{\bf (a)} Case $\BNu_{K/k}(\CM_K) \ne 1$ (i.e., $\BJ_{K/k}(\CM_k) \ne 1$, 
equivalent to $\BJ_{K/k}$ injective since $\order \CM_K^1 = p$).
Put ${\rm maxi} = a \cdot (p-1)+ b$, with $a \geq 0$ and $0 \leq b \leq p-2$.
Then we have necessarily ${\rm maxi} \geq 2$ and the following possibilities, 
all giving $\order \CM_K = p^{\rm maxi}$:

\smallskip
\quad (i) Case ${\rm maxi} < p$. Then $\CM_K \simeq \big( \Z/ p^{2} \Z \big) 
\oplus \big( \Z/ p \Z \big)^{{\rm maxi}-2}$.

\smallskip
\quad(ii) Case ${\rm maxi} = p$. Then $\CM_K \simeq \big( \Z/ p \Z \big)^p$ 
or $\big( \Z/ p^{2} \Z \big) \oplus \big( \Z/ p \Z \big)^{p-2}$.

\smallskip
\quad(iii) Case ${\rm maxi} > p$. Then $\CM_K \simeq \big( \Z/ p^{a+1} \Z \big)^b 
\oplus \big( \Z/ p^{a} \Z \big)^{p-1-b}$.

\medskip
{\bf (b)} Case $\BNu_{K/k}(\CM_K) = 1$ (complete capitulation of $\CM_k$ in $K_1$).
Then there exist integers $m_j$, $1 \leq m_1 \leq \cdots \leq m_t$, such that
$\CM_K \simeq \bigoplus_{j=1}^t \big[(\Z/p^{a_j+1}\Z)^{b_j} \,
\oplus\, (\Z/p^{a_j}\Z)^{p-1-b_j}\big]$,  
where $m_j=a_j\,(p-1)+b_j$, $a_j \geq 0$ and $0 \leq b_j \leq p-2$.
\end{theorem}

\begin{remark}
Note that for $\CM_K = \CH_K$ with $e=1$, then
$\BJ_{K/k}$ injective and $\order \CH_K^1 = \order \CH_k$ imply
$\CH_K^\ram \subseteq \BJ_{K/k} (\CH_k)$. So the assumptions in the
two approaches are similar for $e=1$.

\smallskip\noindent
This classification in the case $e=1$ shows that a ``smooth'' complexity 
of $\CH_K$ holds, roughly speaking, for small ${\rm maxi}$'s; so, in a statistic 
point of view, probabilities are decreasing regarding $m$. It seems 
reasonable to say that assumptions such as $\order \CH_k = p^e$ and the
relations $\order(\CH_K^{i+1}/\CH_K^i) = \ffrac{p^e}{\order \BN_{K/k}(\CH_K^i)}$, 
when norm factors are trivial, gives probabilities $\frac{1}{p}$ that 
$\CH_K^{i+1} = \CH_K^i$ under the condition $\order \CH_K^i = 
p^{e_i}$, $e_i \le e$. We refer to the algorithm, allowing to switch 
from the step $i$ to the step $i+1$, that we have described in many papers.
\end{remark}

The following result was proved in \cite[Theorem 6.1]{KW2023}; then we have 
obtained the following similar statement for $k^\acyc$, of which we have given 
a simpler proof using a specific property of a metabelian framework (here dihedral).

\begin{proposition} \label{noncyclic} 
\cite[Theorem 5.6]{Gra2024b}.
We assume that $\CH_k$ is cyclic non-trivial and that the Hilbert class 
field $H_k^\nr$ is not contained in $k^\acyc$. Put $k^\acyc \cap H_k^\nr 
= k^\acyc_e$, $e \geq 0$. Then, for all $n \geq e+1$, the $p$-class group 
of $k^\acyc_n$ is non-cyclic. So, if $e=0$, we have necessarily $n\geq 3$ 
in the above classification \ref{classification}, and only the two possibilities 
(ii) and (iii) for the structure of the $p$-class group of $k_1^\acyc$.
\end{proposition}

\begin{remark}
In fact, this dihedral context has been studied extensively from \cite{Sch1933},
\cite{Mar1969}, \cite{BC1971}, \cite{Mos1974}, \cite{HKM1978} and others); 
then one has the $p$-class formula (see \cite[Formula 13]{HKM1978}, 
\cite[Theorem 10]{Jau1981} in a whole metabelian study, \cite[Corollary 4.2]
{CN2020} in a similar way, \cite[Theorem 2.2]{Lem2005} giving moreover
a view using reflection theorems for relations between $p$-ranks of 
class groups, then \cite{Hub2008}). Let $Q$ be $\Q$  or an imaginary 
quadratic field; for $p=3$, the result is:
\begin{equation}\label{dihedral}
\order \BH_{k_1^\acyc} = {\bf a} \times \frac{\order \BH_k \times (\order \BH_L)^2}
{(\order \BH_Q)^2} \times \frac{1}{3^{{\bf b}(d)}} = {\bf a} \times 
\frac{\order \BH_k \times (\order \BH_L)^2} {3}, \ \, {\bf a} \in \{1, 3\}, 
\end{equation}

\noindent
where $L $ is one of the non-Galois cubic subfield of $k_1^\acyc$, ${\bf a}$ 
is a unit index equal to $1$ or~$3$, ${\bf b}(d) = \ffrac{d^2+2d-4}{4} = 1$ for 
$d= [k : \Q] = 2$ in the imaginary case. Moreover, ${\bf a}$ may be computed, 
by means of the {\sc pari/gp} formula ${\sf a=3 *3^{-lift(Mod(valuation(kacyc.no/k.no,3),2))}}$ 
(in the four programs, the output of ${\sf a}$ is given with that of ${\sf Val}$).
For instance, some cases of the structure do not hold if the relation 
modulo $2$ on the exponents is not fulfilled.
\end{remark}

\end{appendices}

\tableofcontents


\begin{thebibliography}{}

\end{thebibliography}


\begin{thebibliography}{xxxxxxxxx}

\bibliographystyle{plain}
\bibitem[BC1971]{BC1971} {\sc  P. Barrucand \& H. Cohn}, ``Remarks 
on principal factors in a relative cubic field'', J. Number Theory {\bf 3} (1971),
no. 2, p. 226--239.
\url{https://doi.org/10.1016/0022-314X(71)90040-0}

\bibitem[BP1972]{BP1972} {\sc F. Bertrandias \& J-J. Payan}, 
``$\Gamma$-extensions et invariants cyclotomiques'', Ann. Sci. Ec. 
Norm. Sup. 4e s\'erie {\bf 5} (1972), no. 4, p. 517--548.
\url{https://doi.org/10.24033/asens.1236}

\bibitem[Br2007]{Br2007} {\sc D. Brink}, ``Prime decomposition in the
anti-cyclotomic extension'', Math. Comp. {\bf 76} (2007), no. 260, p. 2127--2138.
\url{https://doi.org/10.1090/S0025-5718-07-01964-3}

\bibitem[Car1975]{Car1975} {\sc J.E. Carroll}, ``On determining the 
quadratic subfields of $Z_2$-extensions of complex quadratic fields'', 
Compositio Math. {\bf 30} (1975), no. 3, p. 259--271. \\
\url{http://www.numdam.org/item/CM_1975__30_3_259_0/}

\bibitem[Che1933]{Che1933} {\sc C. Chevalley}, ``Sur la th\'eorie du corps de 
classes dans les corps finis et les corps locaux'', Th\`ese no. 155, J. of the 
Faculty of Sciences Tokyo {\bf 2} (1933), p. 365--476. \\
\url{http://archive.numdam.org/item/THESE_1934__155__365_0/}

\bibitem[CK1976]{CK1976} {\sc J.E. Carroll \& H. Kisilevsky}, ``Initial layers of 
$\Z_\ell$-extensions of complex quadratic fields'', 
Compositio Math. {\bf 32} (1976), no. 2, p. 157--168. \\
\url{http://www.numdam.org/item/CM_1976__32_2_157_0/}

\bibitem[DFKS1991] {DFKS1991} {\sc D. Dummit, D. Ford, H. Kisilevsky 
\& J. Sands}, ``Computation of Iwasawa lambda invariants for imaginary 
quadratic fields'', J. Number Theory {\bf 37} (1991), p. 100--121.  \\
\url{https://doi.org/10.1016/S0022-314X(05)80027-7}

\bibitem[CN2020]{CN2020} {\sc L. Caputo \& F.A.E. Nuccio Mortarino Majno Di
Capriglio}, ``Class number formula for dihedral extensions'', Glasgow Math. J.
{\bf 62} (2019), no. 2, p. 323--353.\\
\url{https://doi.org/10.1017/S0017089519000144}

\bibitem[Diaz1974]{Diaz1974} {\sc F. Diaz Y Diaz}, ``Sur les corps quadratiques 
imaginaires dont le $3$-rang du groupe des classes est sup\'erieur
\`a $1$'', S\'eminaire Delange--Pisot--Poitou, Th\'eorie des nombres 
{\bf 15} (1973/74), no. 2, expos\'e no G15.
\url{http://www.numdam.org/item/SDPP_1973-1974__15_2_A10_0/}

\bibitem[FK2002]{FK2002} {\sc T. Fukuda \&  K. Komatsu}, ``Noncyclotomic
$\Z_p$-Extensions  of Imaginary Quadratic Fields'',
Experimental Mathematics {\bf 11} (2002), no. 4, p. 469--475. \\
\url{https://doi.org/10.1080/10586458.2002.10504699}

\bibitem[Fu2013]{Fu2013} {\sc S. Fujii}, On a bound of $\lambda$
and the vanishing of $\mu$ of $\Z_p$-extensions  of an imaginary 
quadratic field, J. Math. Soc. Japan {\bf 65} (2013), no. 1, p. 277--298.\\
\url{https://doi.org/10.2969/jmsj/06510277}

\bibitem[FV2002]{FV2002} {\sc I.B. Fesenko \& S.V. Vostokov}, ``Local 
Fields and Their Extensions'', American Math. Soc., Translations of Math. 
Monographs {\bf 121}, Second Edition, 2002. \\                                                                                                                 
\url{https://ivanfesenko.org/wp-content/uploads/2021/10/vol.pdf}   

\bibitem[Ger1977]{Ger1977} {\sc F. Gerth}, ``Structure of $\ell$-class groups of
certain number fields and $\Z_\ell$-extensions'', Mathematika {\bf 24}
(1977), no. 1, p. 16--33.
\url{https://doi.org/10.1112/S002557930000886X}

\bibitem[Gil1985]{Gil1985} {\sc R. Gillard}, ``Fonctions $L$ $p$-adiques des 
corps quadratiques imaginaires et de leurs extensions ab\'eliennes'', J. Reine 
Angew. Math. {\bf 358} (1985), p. 76--91. 
\url{http://eudml.org/doc/152722}

\bibitem[GJ1985]{GJ1985} {\sc M. Grandet \& J-F. Jaulent}, 
``Sur la capitulation dans une $\Z_\ell$-extension'', J. reine angew. 
Math. {\bf 362} (1985), p. 213--217.
\url{http://eudml.org/doc/152777}

\bibitem[GJN2016]{GJN2016} {\sc G. Gras, J-F. Jaulent \& T.
Nguyen Quang Do}, ``Sur le module de Bertrandias--Payan dans une
$p$-extension -- Noyau de capitulation'', p. 25--44;
``Sur la capitulation pour le module de Bertrandias--Payan'', p. 45--58;
``Descente galoisienne et capitulation entre modules de 
Bertrandias--Payan'', p. 59--79, Pub. Math. de Besan\c con, 
Alg\`ebre et th\'eorie des nombres (2016). \\
\url{https://doi.org/10.5802/pmb.o-3} 
\url{https://doi.org/10.5802/pmb.o-4} \\
\url{https://doi.org/10.5802/pmb.o-5}

\bibitem[Gra1983]{Gra1983} {\sc G. Gras},  ``Sur les $\Z_2$-extensions 
d'un corps quadratique imaginaire'', Annales de l'Institut Fourier
{\bf 33} (1983), no. 4, p. 1--18. 
\url{http://www.numdam.org/articles/10.5802/aif.939/}

\bibitem[Gra1985]{Gra1985} {\sc G. Gras},  ``Plongements Kumm\'eriens dans les
$\Z_p$-extensions'', Compositio Math. {\bf 55} (1985), no. 3, p. 383--396.
\url{http://www.numdam.org/item/?id=CM_1985__55_3_383_0}

\bibitem[Gra2005]{Gra2005} {\sc G. Gras}, \textit{Class Field Theory: 
from theory to practice}, corr. 2nd ed. Springer Monographs in Mathematics, 
Springer, xiii+507 pages (2005). \\
\url{https://doi.org/10.1007/978-3-662-11323-3}

\bibitem[Gra2017]{Gra2017} {\sc G. Gras}, ``Invariant generalized ideal 
classes -- Structure theorems for $p$-class groups in $p$-extensions'',
Proc. Indian Acad. Sci. (Math. Sci.) {\bf 127} (2017), no. 1, p. 1--34. \\
\url{https://doi.org/10.1007/s12044-016-0324-1}.

\bibitem[Gra2018]{Gra2018} {\sc G. Gras}, ``The $p$-adic Kummer--Leopoldt 
Constant: Normalized $p$-adic Regulator'', Int. J. Number Theory {\bf 14} 
(2018), no. 2, p. 329--337. 
\url{https://doi.org/10.1142/S1793042118500203}

\bibitem[Gra2019$^c$]{Gra2019c} {\sc G. Gras},  ``Practice of the Incomplete 
$p$-Ramification Over a Number Field -- History of Abelian $p$-Ramification'',
Communications in Advanced Mathematical Sciences {\bf 2} (2019), 
no. 4, p. 251--280. \url{https://doi.org/10.33434/cams.573729}

\bibitem[Gra2021]{Gra2021} {\sc G. Gras},  ``Algorithmic complexity of Greenberg's 
conjecture'', Arch. Math. {\bf 117} (2021), p. 277--289. 
\url{https://doi.org/10.1007/s00013-021-01618-9}

\bibitem[Gra2022]{Gra2022} {\sc G. Gras},  ``On the $\lambda$-stability 
of $p$-class groups along $p$-towers of a number field'', 
Int. J. Number Theory {\bf 18} (2022), no. 10, p. 2241--2263.
\url{https://doi.org/10.1142/S1793042122501147}

\bibitem[Gra2023]{Gra2023} {\sc G. Gras}, ``Notion of abelian arithmetic 
$\varphi$-objects for the study of $p$-class groups and $p$-ramified 
torsion groups'', North-Western European Journal of Mathematics
{\bf 9}, p. 109--199. \\
\url{https://nwejm.univ-lille.fr/index.php/nwejm/article/view/71}

\bibitem[Gra2024$^a$]{Gra2024a} {\sc G. Gras}, ``Algebraic norm and 
$p$-class group capitulations in totally ramified cyclic $p$-extensions'', 
Math. Comp. (2024) (to appear). 
\url{https://doi.org/10.1090/mcom/3920}

\bibitem[Gra2024$^b$]{Gra2024b} {\sc G. Gras}, ``Initial layer of the 
anti-cyclotomic $\Z_p$-extension of an imaginary quadratic field -- With 
an appendix by Jean-Fran\c cois Jaulent'' (preprint 2024). \\
\url{https://arxiv.org/abs/2403.16603}

\bibitem[Gree1976]{Gree1976} {\sc R. Greenberg}, ``On the Iwasawa invariants 
of totally real number fields'', Amer. J. Math. \textbf{98} (1976), no. 1, p.  263--284.  
\url{https://doi.org/10.2307/2373625}

\bibitem[HBW2019]{HBW2019} {\sc D. Hubbard, R. Br\"oker \& L.C. Washington},
Explicit Computations in Iwasawa theory, Proceedings of the Thirteenth 
Algorithmic Number Theory Symposium, R. Scheidler  \&  J. Sorenson (Eds), 
Open Book Series 2, Mathematical Sciences Publishers, 
Berkeley, 2019, pp. 137--153. \\
\url{https://doi.org/10.2140/obs.2019.2.137}

\bibitem[HKM1978]{HKM1978} {\sc F. Halter-Koch \& N. Moser}, ``Sur le 
nombre de classes de certaines extensions m\'etacycliques sur $\Q$ ou
sur un corps quadratique imaginaire'', J. Math. Soc. Japan {\bf 30} (1978),
no. 2, p. 237--148. \\
\url{https://doi.org/10.2969/JMSJ/03020237}

\bibitem[Hub2008]{Hub2008} {\sc D. Hubbard}, ``Dihedral side extensions 
and class groups'', J. Number Theory {\bf 128} (2008), p. 731--737.
\url{https://doi.org/10.1016/j.jnt.2007.04.008}

\bibitem[HW2010]{HW2010} {\sc D. Hubbard \& L.C. Washington}, ``Kummer generators 
and lambda invariants'', J. Number Theory {\bf 130} (2010), no. 1, p. 61--81. 
\url{https://doi.org/10.1016/j.jnt.2009.06.001}

\bibitem[HW2018]{HW2018} {\sc D. Hubbard \& L.C. Washington}, ``Iwasawa 
invariants of some non-cyclo\-tomic $\Z_p$-extensions'', J. Number Theory 
{\bf 188} (2018), p. 18--47. 
\url{https://doi.org/10.1016/j.jnt.2018.01.009}

\bibitem[Iw1973]{Iw1973} {\sc K. Iwasawa}, ``On $\Z_\ell$-extensions of algebraic 
number fields, Ann. Math. {\bf 98} (1973), no. 2, p. 246--326.
\url{https://doi.org/10.2307/1970784}

\bibitem[Jau1981]{Jau1981} {\sc J-F. Jaulent}, ``Unit\'es et classes dans les
extensions m\'etab\'eliennes de degr\'e $n\ell^s$ sur un corps de nombres
alg\'ebriques'', Annales de l'Institut Fourier {\bf 31} (1981), no. 1, p. 39--62.\\
\url{http://www.numdam.org/item/?id=AIF_1981__31_1_39_0}

\bibitem[Jau1986] {Jau1986}  {\sc J-F. Jaulent}, ``L'arithm\'etique des 
${\ell}$-extensions'' (Th\`ese d'\'etat), Publications Math\'ematiques de 
Besan\c con (1986), vol. 1, no. 1, 1--357.
\url{https://doi.org/10.5802/pmb.a-42}

\bibitem[Jau1988]{Jau1988} {\sc J-F. Jaulent}, ``L'\'etat actuel du probl\`eme
de la capitulation'', S\'em. Th\'eor. Nombres Bordeaux 1987--1988, exp. no.17.
\url{https://arxiv.org/pdf/1808.01779}

\bibitem[Jau1994]{Jau1994} {\sc J-F. Jaulent},  ``Classes logarithmiques des 
corps de nombres'', J. Th\'eor. Nombres Bordeaux {\bf 6} (1994), no. 2, p. 301--325. 
\url{https://doi.org/10.5802/jtnb.117}

\bibitem[Jau1998]{Jau1998} {\sc J-F. Jaulent}, ``Th\'eorie $\ell$-adique 
globale du  corps de classes'', J. Th\'eor. Nombres Bordeaux {\bf 10}(2) 
(1998), no. 2, p. 355--397. 
\url{https://doi.org/10.5802/jtnb.233}

\bibitem[Jau2016${}^a$]{Jau2016a} {\sc J-F. Jaulent},
``Sur le radical kumm\'erien des $\Z_\ell$-extensions'', Acta Arithmetica,
{\bf 175} (2016), no. 3, p. 245--253. \url{https://doi.org/10.4064/aa8366-7-2016}

\bibitem[Jau2016${}^b$]{Jau2016b}  {\sc J-F. Jaulent},  ``Classes logarithmiques 
et capitulation'', Functiones et Approximatio {\bf 54} (2016), no. 2, p. 227--239. 
\url{https://doi.org/10.7169/facm/2016.54.2.6}       

\bibitem[Jau2019]{Jau2019} {\sc J-F. Jaulent}, ``Note sur la conjecture de 
Greenberg'', J. Ramanujan Math. Soc. {\bf 34} (2019), no. 1, p. 59--80. 
\url{http://www.mathjournals.org/jrms/2019-034-001/2019-034-001-005.html}

\bibitem[Jau2024${}^a$]{Jau2024a} {\sc J-F. Jaulent}, ``Sur la trivialit\'e 
de certains modules d'Iwasawa'', Funct. Approx. Comment. Math. {\bf 70} 
(2024), no. 1, p. 29--39. 
\url{https://doi.org/10.7169/facm/2092}

\bibitem[Jau2024${}^b$]{Jau2024b} {\sc J-F. Jaulent}, ``Classes 
logarithmiques imaginaires des corps ab\'eliens'' (preprint 2024).\\
\url{https://arxiv.org/abs/2406.18929}

\bibitem[Jau2025]{Jau2025} {\sc J-F. Jaulent}, ``Sur les $\Z_\ell$-extensions 
anti-cyclotomiques des corps de nombres'' (preprint 2025).

\bibitem[KO2004]{KO2004} {\sc J.M. Kim \& J. Oh}, ``Defining polynomial of the first 
layer of anti-cyclotomic $\Z_3$-extension of imaginary quadratic fields of class 
number $1$'', Proc. Japan Acad. {\bf 80}(A) (2004), p. 18--19. \\
\url{https://doi.org/10.3792/pjaa.80.18}

\bibitem[KW2023]{KW2023} {\sc D. Kundu \& L.C. Washington}, ``Heuristics 
for Anti-cyclotomic $\Z_p$-extensions'', Experimental Mathematics {\bf 33} (2023), 
no. 4, p. 644--662. \\
\url{https://doi.org/10.1080/10586458.2023.2221866}

\bibitem[KW2024${}^a$]{KW2024a} {\sc D. Kundu \& L.C. Washington}, 
``The first level of $\Z_p$-extensions and compatibility of heuristics'' (2024).
\url{https://doi.org/10.48550/arXiv.2410.06193}

\bibitem[KW2024${}^b$]{KW2024b} {\sc D. Kundu \& L.C. Washington}, 
personal communication of some examples of layers of anti-cyclotomic 
$\Z_3$-extensions of imaginary quadratic fields (2024).

\bibitem[Lem2005]{Lem2005} {\sc F. Lemmermeyer}, ``Class groups 
of dihedral extensions'', Mathematische Nachrichten {\bf 278} (2005),
no. 6, p. 627--737. \url{http://www.fen.bilkent.edu.tr/~franz/publ/mndih.pdf}\\
\url{https://doi.org/10.1002/mana.200310263}

\bibitem[Mai2011]{Mai2011} {\sc C. Maire}, ``Sur la structure Galoisienne 
de certaines pro-$p$-extensions de corps de nombres'', Math. Z. {\bf 267}
(2011), p. 887--913.
\url{https://doi.org/10.1007/s00209-009-0652-2}

\bibitem[Mar1969]{Mar1969} {\sc  J. Martinet}, ``Sur l'arithm\'etique 
des extensions galoisiennes \`a groupe de Galois di\'edral d'ordre $2p$'',
Annales de l'Institut Fourier {\bf 19} (1969), no. 1, p. 1--80.\\
\url{http://www.numdam.org/articles/10.5802/aif.307/}

\bibitem[Mos1974]{Mos1974} {\sc N. Moser}, ``Unit\'es  et nombre de 
classes d'une extension galoisienne di\'edrale de $\Q$'', S\'eminaire 
de th\'eorie des nombres de Grenoble  {\bf 3} (1973--1974), Expos\'e no. 4, 22 p.\\
\url{http://www.numdam.org/item/STNG_1973-1974__3__A4_0/}
Published in: Abh. Math. Sem. Univ. Hamburg {\bf 48} (1979), p. 54--75.
\url{https://doi.org/10.1007/BF02941290}

\bibitem[MPS2025]{MPS2025} {\sc P. Mercuri, M. Paoluzi, R. Schoof},
``Greenberg's conjecture for real quadratic number fields'' (2025).
\url{https://doi.org/10.48550/arXiv.2503.00819}

\bibitem[NQD1986]{NQD1986} {\sc T. Nguyen Quang Do}, ``Sur la 
${\mathbb {Z}}_p$-torsion de certains modules galoisiens'',
Annales de l'Institut Fourier {\bf 36} (1986), no. 2, p. 27--46. 
\url{https://doi.org/10.5802/aif.1045}

\bibitem[MR2019]{MR2019} {\sc C. Maire \& M. Rougnant}, ``Composantes 
isotypiques de pro-$p$-extensions de corps de nombres et $p$-rationalit\'e'',
Publicationes Mathematicae Debrecen {\bf 94} (2019), no. 1/2, p. 123--155. \\
\url{https://publi.math.unideb.hu/load_doc.php?p=2284&t=abs} \\
\url{https://hal.science/hal-02107052v1/document}

\bibitem[Oh2015]{Oh2015} {\sc J. Oh}, ``On the anti-cyclotomic 
$\Z_p$-extension of an imaginary quadratic field'', Korean J. Math. 
{\bf 23} (2015), no. 3, p. 323--326.
\url{http://dx.doi.org/10.11568/kjm.2015.23.3.323}

\bibitem[Pag2022]{Pag2022} {\sc L. Pagani}, ``Greenberg's conjecture 
for real quadratic fields and the cyclotomic $\Z_2$-extension'', Math. 
Comp. {\bf 91} (2022), p. 1437--1467. 
\url{https://doi.org/10.1090/mcom/3712}.

\bibitem[Pari2019]{Pari2019} The PARI Group,  {\it PARI/GP, 
Version \texttt{2.11.2} (2019)}, Univ. Bordeaux. 

\bibitem[San1993]{San1993} {\sc J.W. Sands}, ``On the non-triviality 
of the basic Iwasawa $\lambda$-invariant for an infinitude of imaginary 
quadratic fields'', Acta Arith. {\bf 65} (1993), no. 3, p. 243--248. \\
\url{https://doi.org/10.4064/aa-65-3-243-248}

\bibitem[Sch1933]{Sch1933} {\sc A. Scholz}, ``Idealklassen und Einheiten in 
kubischen K\"orpern'', Monatsh. f. Mathematik und Physik {\bf 40} (1933),
p. 211--222.
\url{https://doi.org/10.1007/BF01708865}

\bibitem[Wa1997]{Wa1997} {\sc L.C. Washington}, ``Introduction to 
cyclotomic fields'', Graduate Texts in Math. {\bf 83}, Springer enlarged 
second edition 1997.

\end{thebibliography}
\end{document}